\documentclass[11pt,reqno]{amsart} 
\usepackage[T1]{fontenc}
\usepackage{changebar}
\usepackage{hyperref}
\usepackage{amsmath,amssymb,amsthm,amsfonts,stmaryrd,rotating,dsfont,bm}
\usepackage[latin9]{inputenc} 
 \usepackage[a4paper]{geometry}
\usepackage[matrix,arrow,curve]{xy}
\usepackage{tensor}
\CompilePrefix{xy/integrals}

\DeclareMathOperator{\ev}{ev}

\newcommand{\intd}{\,\mathrm{d}}

\newcommand{\eps}{\varepsilon}

\newcommand{\bphi}{{_{B}\phi}}

\newcommand{\cphi}{{_{C}\phi}}
\newcommand{\bpsi}{{_{B}\psi}}
\newcommand{\cpsi}{{_{C}\psi}}
\newcommand{\phib}{{\phi_{B}}}
\newcommand{\psib}{{\psi_{B}}}
\newcommand{\phic}{{\phi_{C}}}
\newcommand{\psic}{{\psi_{C}}}

\newcommand{\dual}[1]{#1^{\vee}}

\newcommand{\dA}{\dual{A}}
\newcommand{\dB}{\dual{B}}

\newcommand{\dbA}{\dual{(\bA)}}
\newcommand{\dcA}{\dual{(\cA)}}
\newcommand{\dAb}{\dual{(\Ab)}}
\newcommand{\dAc}{\dual{(\Ac)}}

\newcommand{\dbAb}{\dual{(\gr{A}{}{B}{}{B})}}
\newcommand{\dcAc}{\dual{(\gr{A}{}{C}{}{C})}}

\newcommand{\comega}{{_{C}\omega}}
\newcommand{\bomega}{{_{B}\omega}}
\newcommand{\omegac}{\omega_{C}}
\newcommand{\omegab}{\omega_{B}}
\newcommand{\cupsilon}{{_{C}\upsilon}}
\newcommand{\bupsilon}{{_{B}\upsilon}}
\newcommand{\upsilonc}{\upsilon_{C}}
\newcommand{\upsilonb}{\upsilon_{B}}

\newcommand{\cphic}{\cphi{}_{C}}
\newcommand{\bpsib}{\bpsi{}_{B}}

\newcommand{\GstG}{G {_{s}\times_{t}} G}
\newcommand{\GssG}{G {_{s}\times_{s}} G}
\newcommand{\GttG}{G {_{t}\times_{t}} G}

\numberwithin{equation}{section}

\theoremstyle{definition} 

\swapnumbers
\newtheorem{remark}{Remark}[subsection]
\newtheorem{remarks}{Remarks}[subsection]
\newtheorem{example}[remark]{Example}
\newtheorem{notation}[remark]{Notation}
\newtheorem*{notation*}{Notation}

\theoremstyle{plain}
\newtheorem{definition}[remark]{Definition}
\newtheorem{theorem}[remark]{Theorem}
\newtheorem{proposition}[remark]{Proposition}
\newtheorem{corollary}[remark]{Corollary}
\newtheorem{lemma}[remark]{Lemma}

\newtheorem*{assumption*}{Assumption}

\newcommand{\C}{\mathds{C}}

\newcommand{\Z}{\mathds{Z}}

\newcommand{\oo}{\otimes}

\newcommand{\gr}[5]{\tensor*[^{#2}_{#3}]{#1}{^{#4}_{#5}}}

\newcommand{\id}{\iota}
\DeclareMathOperator{\Hom}{Hom}
\DeclareMathOperator{\End}{End}

\newcommand{\bB}{{_{B}B}}
\newcommand{\Bb}{B_{B}}
\newcommand{\cC}{{_{C}C}}
\newcommand{\Cc}{C_{C}}

\newcommand{\bsA}{\gr{A}{}{B}{}{}}
\newcommand{\btA}{\gr{A}{}{}{B}{}}
\newcommand{\bAs}{\gr{A}{}{}{}{B}}
\newcommand{\bAt}{\gr{A}{B}{}{}{}}

\newcommand{\csA}{\gr{A}{}{C}{}{}}
\newcommand{\ctA}{\gr{A}{}{}{C}{}}
\newcommand{\cAs}{\gr{A}{}{}{}{C}}
\newcommand{\cAt}{\gr{A}{C}{}{}{}}

\newcommand{\cAc}{\gr{A}{}{C}{}{C}}
\newcommand{\cCc}{\gr{C}{}{C}{}{C}}
\newcommand{\bAb}{\gr{A}{}{B}{}{B}}
\newcommand{\bBb}{\gr{B}{}{B}{}{B}}

\newcommand{\AltkA}{\bsA \overline{\times} \btA} 
\newcommand{\ArtkA}{\cAt \overline{\times} \cAs} 

\newcommand{\cAsA}{\cAs \otimes \csA}
\newcommand{\bAsA}{\bAs \otimes \bsA}

\newcommand{\AlA}{\bsA \otimes \btA}

\newcommand{\ArA}{\cAt \otimes \cAs}

\newcommand{\bATA}{\btA \otimes \bAt}
\newcommand{\cATA}{\ctA \otimes \cAt}

\newcommand{\AlAlA}{\bsA \otimes  \gr{A}{}{B'}{B}{} \otimes \gr{A}{}{}{B'}{}}

\newcommand{\bA}{{_{B}A}}
\newcommand{\Ab}{A_{B}}
\newcommand{\cA}{{_{C}A}}
\newcommand{\Ac}{A_{C}}
\newcommand{\sbA}{A^{B}}

\newcommand{\Asc}{\gr{A}{C}{}{}{}}

\newcommand{\AbA}{\Ab \oo \bA}

\newcommand{\ABA}{\bA \oo \Ab}

\newcommand{\AcA}{\Ac \oo \cA}

\newcommand{\ACA}{\cA \oo \Ac}

\newcommand{\op}{\mathrm{op}}
\newcommand{\co}{\mathrm{co}}

\newcommand{\Tl}{T_{\lambda}}
\newcommand{\Tr}{T_{\rho}}
\newcommand{\lT}{{_{\lambda}T}}
\newcommand{\rT}{{_{\rho}T}}

\newcommand{\beps}{{_{B}\varepsilon}}
\newcommand{\epsc}{\varepsilon_{C}}
\newcommand{\ceps}{{_{C}\varepsilon}}
\newcommand{\epsb}{\varepsilon_{B}}

\newcommand{\mult}{m}

\newcommand{\actleft}{\triangleleft}
\newcommand{\actright}{\triangleright}

\title[Integration on algebraic quantum groupoids]{Integration on algebraic quantum groupoids}

\author{Thomas Timmermann} 
 \address{FB Mathematik und Informatik, University of Muenster \\ Einsteinstr.\ 62, 48149
   Muenster, Germany}
 \email{timmermt@math.uni-muenster.de}

 \thanks{Supported by the SFB 878
     ``Groups, geometry and actions'' funded by the DFG}
\date{\today}


\begin{document}

\begin{abstract}
In this article, we develop a theory of integration on algebraic quantum groupoids in the form of regular multiplier Hopf algebroids, and establish the main properties of integrals  obtained by Van Daele for algebraic quantum groups before  --- faithfulness, uniqueness up to scaling, existence of a modular element and existence of a modular automorphism --- for algebraic quantum groupoids under reasonable assumptions. 
 The approach to integration developed in this article forms the basis for the extension of Pontrjagin duality to algebraic quantum groupoids, and for the passage from algebraic quantum groupoids to operator-algebraic completions, which both will be studied in separate articles.
\end{abstract}
\maketitle

\tableofcontents


 \section{Introduction}

In this article, we develop \emph{a theory of integration on algebraic quantum groupoids} in the form of regular multiplier Hopf algebroids \cite{timmermann:regular}, and establish the \emph{main properties of integrals} that were obtained by Van Daele in \cite{daele}  for algebraic quantum groups  --- faithfulness, uniqueness up to scaling, existence of a modular element and existence of a modular automorphism --- for algebraic quantum groupoids under reasonable assumptions. 
 The approach to integration developed in this article forms the basis for two important constructions.

 To every algebraic quantum groupoid equipped with  suitable integrals, we can associate a \emph{generalized Pontrjagin dual}, which is an
  algebraic quantum groupoid again. This construction generalizes  corresponding results of Van Daele for algebraic quantum groups
  \cite{daele} and of Enock and Lesieur for measured quantum groupoids  \cite{enock:action,lesieur}, and will be studied in a separate
  article \cite{timmermann:dual}, see also \cite{timmermann:integrals}.

 In the involutive case, we can construct \emph{operator-algebraic   completions} in the form of Hopf-von Neumann bimodules  \cite{vallin:1} and of Hopf $C^{*}$-bimodules \cite{timmermann:cpmu}, and thus link the algebraic approaches to  quantum groupoids to the operator-algebraic one.  This construction  generalizes corresponding results of Kustermans and Van Daele for  algebraic quantum groups \cite{kustermans:algebraic} and of the  author for dynamical quantum groups \cite{timmermann:dynamical}, and  is detailed in a forthcoming article \cite{timmermann:opalg}.

\smallskip

To explain the \emph{main ideas and results of our approach}, let us  first look at \emph{multiplier Hopf algebras} \cite{daele}. Given such a multiplier Hopf algebra $A$ with comultiplication $\Delta$, a non-zero linear functional $\phi$  on $A$ is called a \emph{left integral} if
\begin{align} \label{eq:integrals-mha}
  (\id \otimes \phi)(\Delta(a)(1 \otimes b)) = \phi(a)b 
\end{align} 
for all $a,b\in A$,  where the product $\Delta(a)(1 \otimes b)$  lies in $A\otimes A$ by assumption and $\id \otimes \phi$  is the ordinary slice maps from $A\otimes A$ to $A$.  Van Daele showed that every such integral
\begin{enumerate}
\item  is \emph{faithful}: if $b\neq 0$, then $\phi(ab)\neq 0$ and $\phi(bc) \neq 0$ for some $a,c\in A$;
\item is \emph{unique up to scaling}: every left integral has the form $\lambda \phi$ with $\lambda \in \C$;
\item admits a  \emph{modular automorphism} $\sigma$  such that $\phi(ab)=\phi(b\sigma(a))$ for all $a,b$.
\end{enumerate}
Corresponding results hold for every right integral $\psi$, which is a non-zero linear functional on $A$ satisfying
\begin{align} \label{eq:integrals-mha-right}
  (\psi\otimes \id)((a\otimes 1)\Delta(b)) = \psi(b)a
\end{align} 
for all $a,b\in A$. The last key result of Van Daele on integrals is
\begin{enumerate}\setcounter{enumi}{3}
\item \emph{existence of an (invertible) modular element} $\delta$ such that $\psi(a)=\phi(a\delta)$ for all $a\in A$.
\end{enumerate}

\smallskip

Our aim is to establish corresponding results for integrals on algebraic quantum groupoids and to provide the basis for the two applications outlined above. 

In the framework of \emph{weak multiplier Hopf algebras}, this will be done by Van Daele in a forthcoming paper. A weak multiplier Hopf algebra  consists of an algebra $A$ and a comultiplication $\Delta$ that is, in a sense (when extended to the multiplier algebras) no longer unital but still takes values in multipliers of $A\otimes A$. In that setting, the invariance conditions \eqref{eq:integrals-mha} still make sense for  functionals $\phi$ and $\psi$ on $A$, and the results  (1)--(4) above can be carried over from \cite{daele} with additional arguments.

In the present paper, we develop the theory in the considerably  more general and challenging framework of \emph{regular multiplier Hopf algebroids}. The latter were introduced by Van Daele and the author in \cite{timmermann:regular} and simultaneously generalize the regular weak multiplier Hopf algebras studied by Van Daele and Wang \cite{daele:weakmult,daele:relation} and Böhm \cite{boehm:weak}, and  Hopf algebroids studied by \cite{boehm:bijective,lu:hopf,xu}, see also \cite{boehm:hopf}.

A \emph{regular multiplier Hopf algebroid} consists of a total  algebra $A$, commuting subalgebras $B,C$ of the multiplier algebra of $A$ with anti-isomorphisms $S_{B}\colon B\to C$ and $S_{C} \colon C\to B$, and a left and a right comultiplication $\Delta_{B}$ and $\Delta_{C}$ which map $A$ to certain multiplier algebras such that one can form products of the form
\begin{align*}
&  \Delta_{B}(a)(1\otimes b), && \Delta_{B}(b)(a \otimes 1), &&
(a\otimes 1)\Delta_{C}(b), &&  (1\otimes b)\Delta_{C}(a). 
\end{align*}
These products do no longer lie in the tensor product $A\otimes A$ but rather in certain balanced tensor products $\AlA$ and $\ArA$, respectively, which are formed by considering $A$ as a  module over $B$ or $C$ in various ways. 
None of the algebras $A,B,C$ needs to be unital; if  all are, then one has a Hopf algebroid.

\smallskip 

What is the appropriate notion of a left or right integral for regular multiplier Hopf algebroids?

Unlike the case of (weak) multiplier Hopf algebras,  the key invariance relations \eqref{eq:integrals-mha} and \eqref{eq:integrals-mha-right} do no longer make sense for  functionals $\phi$ or $\psi$ on $A$.  The products $\Delta_{B}(a)(1 \otimes b)$ and $(a\otimes 1)\Delta_{C}(b)$ do not lie in the ordinary tensor product $A\otimes A$ but in the  balanced tensor products, and on these balanced tensor products, slice maps of the form $\id \otimes \phi$ or $\psi\otimes \id$ can only be defined if $\phi$ and $\psi$ are   maps from  $A$ to $B$ or $C$, respectively, that are compatible with certain module structures.
In the case of Hopf algebroids, such left- or right-invariant module maps from the total algebra $A$  to the base algebras $B$ and $C$ were studied already by Böhm \cite{boehm:integrals}; see also Böhm  and Szlachyani  \cite{boehm:bijective}.

\smallskip

The \emph{key idea} of our approach is to regard  not only such left- or right-invariant module maps from $A$ to $B$ or $C$, which correspond to \emph{partial, relative} or \emph{fiber-wise integration}, but \emph{total integrals} obtained by composition  with suitable functionals $\mu_{B}$ and $\mu_{C}$ on the base algebras $B$ and $C$. 

Why is it natural, necessary and useful to study  such scalar-valued total integrals?
\begin{enumerate}
\item The \emph{main results}  of Van Daele --- uniqueness and existence of a modular automorphism and of a modular element --- do  not hold for partial integrals or can not even be  formulated.  We shall prove that \emph{all of these results carry    over to scalar-valued total integrals}.

\item  The situation is similar for
  \emph{locally compact groupoids}
    \cite{renault}, where total integration of functions on a groupoid    is given by fiber-wise integration with respect to a left or right
    Haar system, followed by integration over the unit space with    respect to a quasi-invariant measure; and for   \emph{measured quantum groupoids} \cite{lesieur},    \cite{enock:action}, which are given by a Hopf-von Neumann  bimodule, the operator-algebraic counterpart to a multiplier Hopf    algebroid, together with a left- and a right-invariant partial  integral and a suitable weight on the base algebra. Again, the    interplay of the partial integrals and the weight on the base  is crucial for the whole theory.
\item To construct  a \emph{generalized Pontrjagin dual} of a (multiplier) Hopf algebroid,  one first has to define a dual algebra with a convolution product. If one regards the total algebra $A$ as a module over the base algebras $B$ and $C$, one obtains four dual modules with natural convolution products, two dual to  the left and two dual to the right comultiplication. Our approach yields an embedding of these four modules into the dual vector space of $A$ and one subspace  of the intersection where the four products coincide. In \cite{timmermann:dual}, we will show that this subspace can be equipped with the structure of a multiplier Hopf algebroid again; see also \cite{timmermann:integrals}.
\item Total integrals form the key to relate the algebraic  approach to quantum groupoids to the operator-algebraic one, as we  shall show in a forthcoming paper \cite{timmermann:opalg}. Given a multiplier Hopf
  $*$-algebroid with positive total integrals, one can define a  natural Hilbert space of ``square-integrable functions on the  quantum groupoid'' and construct a $*$-representation of the total  algebra which gives rise to a Hopf-von Neumann bimodule.
\end{enumerate}

\smallskip

How should the functionals $\mu_{B}$ and $\mu_{C}$ on the base algebras $B$ and $C$ then be chosen?

 Obviously, they should be faithful. Next, we demand that they are \emph{antipodal} in the sense that
\begin{align} \label{eq:intro-base-weights}
  \mu_{C} &= \mu_{B} \circ S_{C}  &&\text{and} & \mu_{B} &= \mu_{C} \circ S_{B}.
\end{align}
Our third assumption involves the left and the right counit $\beps$ and $\epsc$ of the multiplier Hopf algebroid, which  map $A$ to $B$ and $C$, respectively, and reads
\begin{align}
  \label{eq:intro-counit-condition}
  \mu_{B}\circ \beps = \mu_{C}\circ \epsc.
\end{align}
 This \emph{counitality} condition appeared already in \cite{daele:relation} and has strong implications, for example, that the two equations in \eqref{eq:intro-base-weights} are equivalent and that the anti-isomorphisms $S_{B}$ and $S_{C}$ combine to  \emph{Nakayama automorphisms} or \emph{modular automorphisms} for $\mu_{B}$ and $\mu_{C}$, that is,
\begin{align} 
  \label{eq:intro-modular-automorphism}
  \mu_{B}(xx') &= \mu_{B}(S_{C}S_{B}(x')x) &&\text{and} & \mu_{C}(yy') &= \mu_{C}(y'S_{B}S_{C}(y))
\end{align}
for all $x,x' \in B$, $y,y' \in C$.  It also implies that on a natural subspace of functionals on $A$, the two convolution products induced by the left and by the right comultiplication coincide, which is crucial for the construction of the generalized Pontrjagin dual.

Finally, we demand that $\mu_{B}$ and $\mu_{C}$ are \emph{quasi-invariant} with respect to the partial integrals, which map $A$ to $B$ or $C$, respectively, in a natural sense. This condition is easily seen to be necessary for the existence of a modular element, and similar conditions are used in the theories of locally compact quantum groupoids and of measured quantum groupoids.

\smallskip

What can we say about \emph{existence and uniqueness} of such functionals $\mu_{B}$ and $\mu_{C}$? 

We shall give simple examples which show that neither existence nor uniqueness can be expected in general. This may seem disappointing but is quite natural. Indeed, the situation is  similar to  the question whether an action of a non-compact group on a non-compact space admits an invariant or quasi-invariant measure.

 We also give examples where condition \eqref{eq:intro-counit-condition} can not be satisfied directly, but where the left and  the right comultiplication $\Delta_{B}$ and $\Delta_{C}$ can be \emph{modified}  so that    condition \eqref{eq:intro-counit-condition}  can be satisfied for the new left and right counits. The basic idea is that for every pair of automorphisms $(\Theta_{\lambda},\Theta_{\rho})$ of the underlying algebra $A$  which fix $B$ and $C$ and satisfy
 \begin{align*}
   (\Theta_{\lambda} \bar\times \id) \circ \Delta_{B} = (\id \bar\times \Theta_{\rho}) \circ \Delta_{B},
 \end{align*}
this composition forms a left comultiplication and one obtains a regular multiplier Hopf algebroid again. The right comultiplication can be modified similarly and \emph{independently}. This \emph{modification procedure} considerably generalizes a construction of Van Daele \cite{daele:modified}, and is of interest on its own because it illustrates how loosely the left and the right comultiplication of a (multiplier) Hopf algebroid are related. 

\subsection*{Plan} This article is organized as follows.

 First, we  recall the definition and main properties of regular multiplier Hopf algebroids from \cite{timmermann:regular} (\emph{Section 2}), and introduce the examples  that will be used throughout this article.

Then,  we introduce the  partial integrals, the functionals on the base algebras mentioned above,  the quasi-invariance condition relating the two, and the total integrals obtained by composition  (\emph{Section 3}).

Next, we prove uniqueness of total integrals relative to fixed base functionals $\mu_{B}$ and $\mu_{C}$ up to rescaling (\emph{Section 4}).

We then turn to condition \eqref{eq:intro-counit-condition} which is the last missing ingredient for  our  definition of \emph{measured multiplier Hopf algebroids} (\emph{Section 5}).

Next, we prove the remaining key results on integrals, which are existence of a modular automorphism and modular element, and faithfulness  (\emph{Section 6}). Along the way, we study various convolution  operators and obtain a dual algebra.

Finally, we present the modification procedure mentioned above (\emph{Section 7}) and consider further examples (\emph{Section 8}).

\subsection*{Preliminaries}

We shall use the following conventions and terminology.

All algebras and modules will be complex vector spaces and all homomorphisms will be linear maps, but much of the theory developed in this article applies in wider generality.

The identity map on a set $X$ will be denoted by $\iota_{X}$ or simply
$\iota$.

Let $B$ be an algebra, not necessarily unital. We denote by $B^{\op}$
the \emph{opposite algebra}, which has the same underlying vector space as $B$, but the reversed multiplication.

Given a right module $M$ over $B$, we write $M_{B}$ if we want to emphasize that $M$ is regarded as a right $B$-module. We call $M_{B}$
\emph{faithful} if for each non-zero $b\in B$ there exists an $m\in M$
such that $mb$ is non-zero, \emph{non-degenerate} if for each non-zero
$m\in M$ there exists a $b \in B$ such that $mb$ is non-zero,
\emph{idempotent} if $MB=M$, and we say that $M_{B}$ \emph{has local   units in $B$} if for every finite subset $F\subset M$ there exists a
$b\in B$ with $mb=m$ for all $m\in F$. Note that the last property implies the preceding two. We denote by $\dual{(M_{B})}
:=\Hom(M_{B},B_{B})$ the dual module, and by $\dual{f} \colon
\dual{(N_{B})} \to \dual{(M_{B})}$ the dual of a morphism $f\colon M\to N$
of right $B$-modules, given by $\dual{f}(\chi)= \chi \circ f$.  We use the same notation for duals of vector spaces and of linear maps. We furthermore denote by $L(M_{B}) := \Hom(\Bb,M_{B})$ the space of
\emph{left multipliers} of the module $M_{B}$.

For left modules, we obtain the corresponding notation and terminology by identifying left $B$-modules with right $B^{\op}$-modules.  We denote by $R(_{B}M):=\Hom(\bB,{_{B}M})$ the space of \emph{right   multipliers} of a left $B$-module $_{B}M$.

We write $B_{B}$ or ${_{B}B}$ when we regard $B$ as a right or left module over itself with respect to right or left multiplication. We say that the algebra $B$ is \emph{non-degenerate}, \emph{idempotent},
or \emph{has local units} if the modules ${_{B}B}$ and $B_{B}$ both are non-degenerate, idempotent or both have local units in $B$,
respectively. Note that the last property again implies the preceding two.

We denote by $L(B)=\End(B_{B})$ and
$R(B)=\End({_{B}B})^{\op}$ the algebras of left or right multipliers of $B$, respectively, where the multiplication in the latter algebra is given by $(fg)(b):=g(f(b))$. Note that $B_{B}$ or
${_{B}B}$ is non-degenerate if and only if the natural map from $B$ to
$L(B)$ or $R(B)$, respectively, is injective. If $B_{B}$ is non-degenerate, we define the multiplier algebra of $B$ to be the subalgebra $M(B) :=\{ t\in L(B) : Bt\subseteq B\} \subseteq L(B)$,
where we identify $B$ with its image in $L(B)$. Likewise we could define $M(B) =\{ t\in R(B) : tB \subseteq B\}$ if ${_{B}B}$ is non-degenerate.  If both definitions make sense, that is, if $B$ is non-degenerate, then they evidently coincide up to a natural identification, and a multiplier is given by a pair of maps
$t_{R},t_{L}\colon B\to B$  satisfying $t_{R}(a)b=at_{L}(b)$ for all
$a,b\in B$.

Given a left or right $B$-module $M$ and a space $N$, we regard the space of linear maps from $M$ to $N$ as a right or left $B$-module, where
$(f \cdot b)(m)=f(bm)$ or $(b\cdot f)(m)=f(mb)$ for all maps $f$ and all elements $b\in B$ and $m\in M$, respectively. 

In particular, we regard the dual space $\dual{B}$ of a non-degenerate, idempotent algebra $B$ as a bimodule over $M(B)$,
where $(a \cdot \omega \cdot b)(c)=\omega(bca)$, and call a functional
$\omega \in \dual{B}$ \emph{faithful} if the maps $B \to \dual{B}$
given by $d \mapsto d\cdot\omega$ and $d \mapsto\omega \cdot d$ are injective, that is, $\omega(dB)\neq 0$ and $\omega(Bd) \neq 0$
whenever $d\neq 0$.

We say that a functional $\omega \in \dual{B}$ \emph{admits a modular   automorphism} if there exists an automorphism $\sigma$ of $B$ such that $\omega(ab)=\omega(b\sigma(a))$ for all $a,b\in B$. One easily verifies that this condition holds if and only if $B\cdot \omega =
\omega \cdot B$, and that then $\sigma$ is characterised by the relation $\sigma(b) \cdot \omega = \omega \cdot b$ for all $b\in B$.

We equip the dual space $\dual{B}$ with a preorder $\lesssim$, where
\begin{align} \label{eq:lesssim}
  \upsilon \lesssim \omega \quad :\Leftrightarrow \quad B\upsilon \subseteq B\omega \text{ and } \upsilon B \subseteq \omega B.
\end{align}
The following result is straightforward:
\begin{lemma} \label{lemma:lesssim}
  Suppose that $\omega\in \dB$ is faithful and that $\upsilon \in \dB$.
  \begin{enumerate}
  \item Then $\upsilon \lesssim \omega$ if and only if there exist
    $\delta \in R(B)$ and $\delta' \in L(B)$ such that $
    \omega(x\delta) = \upsilon(x) = \omega(\delta'x)$ for all $x\in
    B$.
  \item If the conditions in (1) hold and $\omega$ admits a
    modular automorphism $\sigma$, then $\delta,\delta'$ lie in $ M(B)$ and
    $\delta=\sigma(\delta')$.
  \end{enumerate}
\end{lemma}
\begin{proof}
  (1) We have $\upsilon \lesssim \omega$ if and only if there exist maps $\delta,\delta'\colon B \to B$  such that
  $x\cdot \upsilon = \delta(x) \cdot \omega$ and $\upsilon \cdot x = \omega \cdot \delta'(x) $ for all $x\in B$. If $\omega$ is faithful, then necessarily $\delta \in R(B)$ and $\delta'\in L(B)$.

  (2) In this case,  we find that for all $x,x'\in B$,
  \begin{align*}
    \omega((x \delta) \sigma(x')) = \omega(x'x\delta) =\upsilon(x'x) = \omega(\delta'x'x) = \omega(x\sigma(\delta' x')).
  \end{align*}
Since $\omega$ is faithful and $x\in B$ was arbitrary, we can conclude that  $(x\delta)\sigma(x') = x\sigma(\delta'x')$, whence the assertion follows.
\end{proof}
Assume that $B$ is a $*$-algebra. We call a functional $\omega \in
\dual{B}$ \emph{self-adjoint} if it coincides with $\omega^{*}=\ast
\circ \omega \circ \ast$, that is, $\omega(a^{*})=\omega(a)^{*}$ for all $a\in B$, and \emph{positive} if additionally $\omega(a^{*}a)\geq
0$ for all $a\in A$.

\section{Regular multiplier Hopf algebroids}

\label{section:multiplier-bialgebroids}

 Regular multiplier Hopf algebroids were introduced in
\cite{timmermann:regular} as non-unital generalizations of Hopf algebroids and are special multiplier bialgebroids. 
As such, they consist of a left and a right multiplier bialgebroid with comultiplications related by a mixed co-associativity condition.

\subsection{Left multiplier bialgebroids}

 Let $A$ be an algebra, not necessarily unital, such that the right module $A_{A}$ is idempotent and non-degenerate. Then we can form the (left) multiplier algebras $L(A)$ and $M(A)\subseteq L(A)$ as explained  above.

Let  $B$ be an algebra, not necessarily unital, with a homomorphism $ s\colon B \to M(A)$ and an anti-homomorphism $t\colon B \to M(A)$ such that $s(B)$ and $t(B)$ commute.  

We denote elements of $B$ by $x,x',y,y',\ldots$ and reserve $a,b,c,\ldots$ for elements of $A$.

  We write $\bsA$ and $\btA$ when we regard $A$ as a left or right   $B$-module via left multiplication along $s$ or $t$, respectively,  that is, $x\cdot a = s(x)a$ and $a\cdot x = t(x)a$. Similarly, we  write $\bAs$ and $\bAt$ when we regard $A$ as a right or left  $B$-module via right multiplication along $s$ or $t$,  respectively. Without further notice, we also regard $\bsA,\bAt$ and  $\bAs,\btA$ as right or left $B^{\op}$-modules, respectively.

Regard the tensor product $\bsA \otimes \btA$ of $B^{\op}$-modules as a right module over $A\otimes 1$ or $1 \otimes A$ in the obvious way and denote by
\begin{align*} 
\AltkA \subseteq \End(\bsA \otimes \btA)
\end{align*} 
the subspace  formed by all endomorphisms $T$ of $\bsA\otimes \btA$ satisfying the following condition: 
 for every $a,b \in A$, there exist elements
  \begin{align*} T(a \otimes 1) \in \bsA \otimes \btA \quad \text{and}
    \quad T(1 \otimes b) \in \bsA \otimes \btA
  \end{align*}
  such that 
  \begin{align*}
T(a \otimes b) = (T(a \otimes 1))(1 \otimes b) = (T(1
    \otimes b))(a \otimes 1)    
  \end{align*}
This subspace is a subalgebra and commutes with the right $A\otimes
A$-module action.

\begin{definition} 
  A \emph{left multiplier bialgebroid} is a tuple $(A,B,s,t,\Delta)$ consisting of
  \begin{enumerate}
  \item algebras $A$ and $B$, where $A$ is non-degenerate and idempotent as a right $A$-module;
  \item a homomorphism $s\colon B \to M(A)$ and an anti-homo\-morphism   $t \colon B \to M(A)$ such that the images of $s$ and $t$ commute,     the $B$-modules $\bsA$ and $\btA$ are faithful and idempotent, and     $\bsA \otimes \btA$ is non-degenerate as a right module over $A
    \otimes 1$ and over $1 \otimes A$;
  \item a homomorphism $\Delta \colon A \to \AltkA$, called the     \emph{left comultiplication},  satisfying 
      \begin{align} \label{eq:left-delta-bimodule}
        \Delta(s(x)t(y)as(x')t(y')) &= (t(y) \otimes
        s(x))\Delta(a)(t(y') \otimes s(x')), \\
 \label{eq:delta-coassociative} (\Delta \otimes \iota)(\Delta(b)(1 \otimes c)) (a\otimes 1 \otimes 1) &=
(\iota\otimes \Delta)(\Delta(b)(a \otimes 1))(1 \otimes 1 \otimes c)
    \end{align} 
\end{enumerate}
A \emph{left counit}  for such a left multiplier bialgebroid  is a map $\eps \colon A \to B$ satisfying
  \begin{gather}
    \begin{aligned}
      \label{eq:left-counit-bimodule}
      \varepsilon(s(x)a ) &= x\varepsilon(a), &
      \varepsilon(t(y)a) &= \varepsilon(a)y,
    \end{aligned} \\
        \label{eq:left-counit}
        (\varepsilon \otimes\id)(\Delta(a)(1 \otimes b)) = ab = (\id
        \otimes \varepsilon)(\Delta(a)(b\otimes 1))
  \end{gather}
for all $a,b\in A$ and $x,y\in B$. 
\end{definition}
Note that \eqref{eq:left-counit-bimodule} implies that the  slice maps 
\begin{align*}
  \varepsilon \otimes \id &\colon \AlA \to A,  \ c\otimes d \mapsto t(\varepsilon(c))d, \\
 \id\otimes \varepsilon&\colon \AlA \to A, \ c\otimes d \mapsto s(\varepsilon(d))c
\end{align*}
occuring in \eqref{eq:left-counit} are well-defined.

\begin{notation} \label{notation:tensor-products}
 We will need to consider iterated tensor products of vector spaces or modules over $B$
  or $B^{\op}$, and if several module structures are used in an iterated tensor product, we mark the
  module structures that go together by primes.  For example, we denote by
  \begin{align*}
    \bsA \otimes A \otimes \btA \quad \text{and}    \quad  \bsA \otimes \gr{A}{}{B'}{B}{} \otimes \gr{A}{}{}{B'}{}
  \end{align*}
  the quotients of $A\otimes A \otimes A$ by the subspaces spanned by all elements of the form
  $s(x)a\otimes b\otimes c - a\otimes b\otimes t(x)a$, where $x\in B, a,b,c\in A$, in the case of 
$\bsA \otimes A \otimes \btA$, or  of the form
  $s(x)a\otimes b\otimes c - a\otimes t(x)b\otimes c$ or $a\otimes s(x')b\otimes c - a\otimes b
  \otimes t(x')c$ in the case of $ \bsA \otimes \gr{A}{}{B'}{B}{} \otimes \gr{A}{}{}{B'}{}$.
\end{notation}

Let $(A,B,s,t,\Delta)$ be a left multiplier bialgebroid. Then the maps
\begin{align*}
  \widetilde{\Tl} \colon A\otimes A &\to \AlA,  \ a \oo b\mapsto  \Delta_{B}(b)(a \oo 1), \\
  \widetilde{\Tr} \colon A\otimes A &\to \AlA,   \  a \oo b\mapsto  \Delta_{B}(a) (1\oo b), 
\end{align*}
 are  well-defined because of 
 the non-degeneracy assumption on $\bsA \otimes \btA$.  By definition and  \eqref{eq:delta-coassociative}, they
make the following diagrams commute,
\begin{gather} \label{dg:left-galois-1}
\xymatrix@R=15pt{
A\otimes A\otimes A \ar[r]^{\id \oo \widetilde{\Tr}} \ar[d]_{\widetilde{\Tl} \oo \id}
      & A \otimes \AlA \ar[d]^{\mult\Sigma \oo \id} \\
\AlA\otimes A \ar[r]^(0.55){\id \oo \mult} & \AlA,
    }
\quad
  \xymatrix@R=15pt{
A\otimes A\otimes A \ar[r]^{\id \oo \widetilde{\Tr}} \ar[d]_{\widetilde{\Tl} \oo \id}
    &
A\otimes \AlA  \ar[d]^{\widetilde{\Tl} \oo \id} \\
\AlA\otimes A  \ar[r]^{\id \oo \widetilde{\Tr}}  & \AlAlA,
}
\end{gather}
where $\Sigma$ denotes the flip map and $m$  the multiplication, 
and by \eqref{eq:left-delta-bimodule}, they factorize to
maps
\begin{align} \label{eq:left-galois-maps}
    \Tl \colon \bATA &\to \AlA,  &
    \Tr \colon \bAsA &\to \AlA,  
\end{align}
which we call the \emph{canonical maps}  of the left multiplier bialgebroid.

\subsection{Right multiplier bialgebroids} 
The notion of a right   multiplier bialgebroid is opposite to the notion of a left
  multiplier bialgebroid in the sense that in all assumptions, left   and right multiplication are reversed.

Let $A$ be an algebra, not necessarily unital, such that the left module $_{A}A$ is non-degenerate and idempotent. Then we can form the (right) multiplier algebras $R(A)$ and  $M(A) \subseteq R(A)$. 

Let $C$ be an algebra with a homomorphism $s\colon C\to M(A) \subseteq R(A)$ and an
anti-homomorphism $t\colon C \to M(A) \subseteq R(A)$ such that the images of $s$ and $t$ commute.

We write $\cAs$ and $\cAt$ if we regard $A$ as a right or left $C$-module such that $a\cdot y = as(y)$
or $y\cdot a = at(y)$ for all $a\in A$ and $y\in C$. We also regard $\cAs$ and $\cAt$ as a left or
right $C^{\op}$-module, and similarly use the notation $\csA$ and $\ctA$ when we use multiplication on
the left hand side instead of the right hand side.

 We consider the opposite algebra  $\End(\cAt \otimes \cAs)^{\op}$ and write $(a \otimes b)T$ for the
image of an element $a\otimes b$ under an element $T \in \End(\cAt \otimes \cAs)^{\op}$, so that $(a\otimes b)(ST) = ((a\otimes b) S)T$ for all $a,b\in A$ and $S,T\in \End(\cAt \otimes \cAs)^{\op}$.
Denote by
\begin{align*}
  \ArtkA \subseteq \End(\cAt \otimes \cAs)^{\op}
\end{align*}
the subspace formed by all endomorphisms $T$ such that for all $a,b\in
A$, there exist elements $(a\otimes 1)T \in \cAt \otimes \cAs$ and $(1
\otimes b)T \in \cAt \otimes \cAs$ such that
\begin{align*}
  (a\otimes b)T = (1\otimes b)((a\otimes 1)T) = (a\otimes 1)((1 \otimes b)T).
\end{align*}
\begin{definition}
  A \emph{right multiplier bialgebroid} is a tuple  $(A,C,s,t,\Delta)$ consisting of
  \begin{enumerate}
  \item  algebras $A$ and $C$, where $A$ is non-degenerate and idempotent as a left $A$-module;
  \item a homomorphism $s\colon C \to M(A) \subseteq R(A)$ and an anti-homo\-morphism $t \colon C
    \to M(A) \subseteq R(A)$ such that the images of $s$ and $t$ commute, the $C$-modules $\cAs$ and
    $\cAt$ are faithful and idempotent, and $\cAt \otimes \cAs$ is non-degenerate as a left module
    over $A \otimes 1$ and over $1 \otimes A$;
\item a homomorphism $\Delta \colon A \to \ArtkA$, called the \emph{right
    comultiplication},  satisfying
  \begin{align} \label{eq:right-delta-bimodule}
    \Delta(s(y)t(x)as(y')t(x')) &= (s(y) \otimes t(x))\Delta(a) (s(y') \otimes t(x')), \\
    (a \otimes 1 \otimes 1) ((\Delta \otimes \id)((1 \otimes
    c)\Delta(b))) &= (1 
\otimes 1 \otimes c)((\id \otimes \Delta)((a \otimes 1)\Delta(b))) \label{eq:right-delta-co-associative}
  \end{align}
for all $a,b,c\in A$ and $x,y\in C$.
\end{enumerate}
A \emph{right counit} for such a right multiplier bialgebroid is a map $\varepsilon \colon A\to C$ satisfying
\begin{gather} 
    \varepsilon(as(y)) = ay, \quad
  \varepsilon(at(x)) = xa,  \label{eq:rt-counit-bimodule}\\
(\varepsilon \otimes \id)((1\otimes b)\Delta(a))=ba =  (\id \otimes
\varepsilon)((b\otimes 1)\Delta(a))  \label{eq:right-counit}
\end{gather}
for all $a,b\in A$, $x,y\in C$.
\end{definition}
Again, \eqref{eq:rt-counit-bimodule} ensures that the slice maps
\begin{align*}
\varepsilon \otimes \id &\colon \ArA \to A,  \ c\otimes d \mapsto ds(\varepsilon(c)), \\
\id \otimes \varepsilon &\colon \ArA \to A, \ c\otimes d \mapsto ct(\varepsilon(d))
\end{align*}
are well-defined.

  Associated to a right multiplier bialgebroid as above are the
 \emph{canonical maps}
  \begin{align*}
      \widetilde{\lT} \colon A\oo A &\to \ArA,  \   a\oo b \mapsto (a \oo 1)\Delta_{C}(b),  \\
\widetilde{\rT}  \colon A\oo A &\to \ArA, \ a\oo b \mapsto  (1 \oo b)\Delta_{C}(a).
  \end{align*}
They     make diagrams similar to those in \eqref{dg:left-galois-1} commute and factorize to maps
  \begin{align}
    \label{eq:right-galois-maps}
    \lT \colon \cAsA &\to \ArA, & \rT \colon \cATA &\to \ArA.
  \end{align}

\subsection{Regular multiplier Hopf algebroids}
\label{subsection:hopf-algebroids}
We now combine the two structures.

\begin{definition} \label{definition:mult-hopf-algebroid} A \emph{multiplier bialgebroid}
  $\mathcal{A}=(A,B,C,S_{B},S_{C},\Delta_{B},\Delta_{C})$ consists of
\begin{enumerate}
\item a non-degenerate, idempotent algebra $A$,
\item subalgebras $B,C \subseteq M(A)$ with anti-isomorphisms $S_{B} \colon
  B\to C$ and $S_{C} \colon C\to B$, 
\item maps $\Delta_{B}\colon A \to \AltkA$ and $\Delta_{C} \colon A \to \ArtkA$
\end{enumerate}
  such that  $\mathcal{A}_{B}=(A,B,\id_{B},S_{B},\Delta_{B})$ is a left multiplier bialgebroid,
  $
\mathcal{A}_{C}= (A,C,\id_{C},S_{C},\Delta_{C})$
is a right multiplier bialgebroid, and
   the following mixed co-associativity conditions hold:
  \begin{align} \label{eq:compatible}
     \begin{aligned} ((\Delta_{B} \otimes \id)((1 \otimes c)\Delta_{C}(b)))(a \otimes 1 \otimes 1) &= (1
\otimes 1 \otimes c)((\id \otimes \Delta_{C})(\Delta_{B}(b)(a \otimes 1))), \\ (a \otimes 1 \otimes 1)((\Delta_{C}
\otimes \id)(\Delta_{B}(b)(1 \otimes c))) &= ((\id \otimes \Delta_{B})((a \otimes 1)\Delta_{C}(b)))(1 \oo
1 \otimes c)
     \end{aligned}
\end{align} 
for all $a,b,c \in A$.
\end{definition}
We call left counits of $\mathcal{A}_{B}$ and right counits of $\mathcal{A}_{C}$ just left and right
counits, respectively, of $\mathcal{A}$. Likewise, we call the canonical maps $\Tl,\Tr$ of
$\mathcal{A}_{B}$ and $\lT,\rT$ of $\mathcal{A}_{C}$ just the canonical maps of $\mathcal{A}$.

Given a  multiplier bialgebroid $(A,B,C,S_{B},S_{C},\Delta_{B},\Delta_{C})$, consider the subspaces
\begin{align*}
  \gr{I}{}{B}{}{} &:= \langle \omega(a) : \omega \in \Hom(\bA,\bB), a \in A\rangle, &
  \gr{I}{}{}{B}{} &:= \langle \omega(a) : \omega \in \Hom(\btA,\Bb), a \in A\rangle, \\
  \gr{I}{}{}{}{C} &:= \langle \omega(a) : \omega \in \Hom(\Ac,\Cc), a \in A\rangle, &
  \gr{I}{C}{}{}{} &:= \langle \omega(a) : \omega \in \Hom(\cAt,\cC), a \in A\rangle
\end{align*}
of $B$ and $C$, respectively.
\begin{definition}
  We call  a multiplier bialgebroid  $(A,B,C,S_{B},S_{C},\Delta_{B},\Delta_{C})$ a \emph{regular multiplier Hopf
    algebroid} if the following conditions hold:
  \begin{enumerate}
  \item the subspaces $S_{B}(\gr{I}{}{B}{}{})\cdot A$, $\gr{I}{}{}{B}{}  \cdot A$, $A\cdot S_{C}(\gr{I}{}{}{}{C})$ and
    $A \cdot \gr{I}{C}{}{}{}$ are equal to $A$;
  \item the canonical maps $\Tl,\Tr,\lT,\rT$ are
  bijective.
  \end{enumerate}
\end{definition}
If $\mathcal{A}$ is a multiplier Hopf algebroid as above, then the maps $\Delta_{B}$ and $\Delta_{C}$ can be extended to homomorphisms
from $M(A)$ to $\End(\AlA)$ and $\End(\ArA)^{\op}$, respectively, such that
\begin{align*}
\Delta_{B}(T)\Delta_{B}(a)(b\otimes c) &= 
\Delta_{B}(Ta)(b\otimes c), &
(a\otimes b)\Delta_{C}(c)\Delta_{C}(T) &= (a\otimes b)\Delta_{C}(cT)
\end{align*}
for all $T\in M(A)$ and $a,b,c\in A$, and then \eqref{eq:left-delta-bimodule} and \eqref{eq:right-delta-bimodule} take the form
\begin{align} \label{eq:delta-bimodule-extended}
  \Delta_{B}(xy) &= y\otimes x, & \Delta_{C}(xy) &= y\otimes x,
\end{align}
where $y\otimes x$ is regarded as an element of  $\End(\AlA)$ and $\End(\ArA)^{\op}$, respectively, via left or right multiplication.

The main result in \cite{timmermann:regular} is the characterization of regular multiplier Hopf algebroids in terms of an invertible antipode:
\begin{theorem}[{\cite[Theorem 5.6]{timmermann:regular}}] \label{tm:hopf-characterization} 
Let  $\mathcal{A}=(A,B,C,S_{B},S_{C},\Delta_{B},\Delta_{C})$ be a multiplier bialgebroid. Then $\mathcal{A}$ is a regular multiplier
  Hopf algebroid if and only if there exists an 
  anti-automorphism $S$ of $A$ satisfying the following conditions:
  \begin{enumerate}
  \item $S(xyax'y')=S_{C}(y')S_{B}(x')S(a)S_{C}(y)S_{B}(x)$     for all $x,x'\in B, y,y' \in C, a \in A$;
  \item there exist a left counit $\beps$ and a right counit $\epsc$    for $\mathcal{A}$ such that the following diagrams commute,
    where $m$ denotes the multiplication maps:
    \begin{align} \label{dg:antipode}
      \xymatrix@C=25pt@R=15pt{\Ab\otimes \bA \ar[d]_{ T_{\rho}}
        \ar[r]^(0.55){S_{C}\epsc \otimes \id} & A, & \Ac\otimes \cA
        \ar[d]_{_{\lambda} T} \ar[r]^(0.55){\id \otimes S_{B}\beps}
        &A. \\
        \bA \otimes \sbA \ar[r]_{S \otimes \id} & \ar[u]_{m} \Ac
        \otimes \cA & \Asc \otimes \Ac \ar[r]_{\id \otimes S} & \ar[u]_{m}
        \Ab \otimes \bA}
    \end{align}
  \end{enumerate}
 In that case, the map $S$, the left counit $\beps$ and
  the right counit $\epsc$ are
  uniquely determined.
\end{theorem} 

Let $\mathcal{A}$ be a regular multiplier Hopf algebroid. Then the map $S$ above is called the \emph{antipode} of $\mathcal{A}$, and the following diagrams commute,
\begin{gather} \label{dg:galois-inverse}
        \xymatrix@R=10pt@C=15pt{\ABA \ar[d]_{\rT} \ar[r]^{\id \oo S} & \AlA,
          \\
          \ArA \ar[r]_{\id \oo S} & \AbA \ar[u]_{\Tr} } \quad
        \xymatrix@R=10pt@C=15pt{\ACA \ar[d]_{\Tl} \ar[r]^{S \oo \id} & \ArA,
          \\
          \AlA \ar[r]_{S \oo \id} & \AcA \ar[u]_{\lT} } \\ \label{dg:galois-antipode}
    \xymatrix@C=25pt@R=10pt{
      \ACA \ar[d]_{\Tl} \ar[r]^{\Sigma(S \oo S)} &       \ABA
      \ar[d]^{\rT} \\
      \AlA \ar[r]_{\Sigma (S \oo S)} & \ArA, }
  \quad
    \xymatrix@C=25pt@R=10pt{
      \AcA \ar[d]_{\lT} \ar[r]^{\Sigma (S \oo S)} &       \AbA
      \ar[d]^{\Tr} \\
      \ArA \ar[r]_{\Sigma (S \oo S)} & \AlA,  } 
\end{gather}
where $\Sigma$ denotes the flip maps on varying tensor products; see Theorem 6.8, Proposition 6.11 and Proposition 6.12 in \cite{timmermann:regular}. Furthermore, by Corollary 5.12 in \cite{timmermann:regular},
\begin{gather} \label{eq:antipode-counits}
  S_{B}\circ \beps = \epsc\circ S \quad\text{and}\quad  S_{C}\circ \epsc =
  \beps \circ S.
\end{gather}
We shall also use the following multiplicativity of the  counits, see (3.5) and (4.9) in \cite{timmermann:regular}:
\begin{align} \label{eq:counits-multiplicative}
  \begin{aligned}
    \beps(ab)&=\beps(a\beps(b))=\beps(aS_{B}(\beps(b))), \\
    \epsc(ab)&=\epsc(\epsc(a)b) = \epsc(S_{C}(\epsc(a))b)
  \end{aligned}
\end{align}
for all $a,b\in A$.

Let us finally consider involutions.
\begin{definition}  \label{definition:involution}
A \emph{multiplier Hopf  $*$-algebroid} is a regular multiplier Hopf algebroid $\mathcal{A}=(A,B,C,S_{B},S_{C},\Delta_{B},\Delta_{C})$ with an
  involution on the underlying algebra $A$   such that
  \begin{enumerate}
  \item $B$ and $C$ are $^*$-subalgebras of $M(A)$;
  \item $  S_{B}\circ \ast \circ S_{C} \circ \ast =\id_{C}$ and $S_{C}\circ \ast \circ S_{B}
\circ \ast =\id_{B}$;
\item $\Delta_{B}(a^{*})(b^{*}
    \otimes c^{*}) = ((b\otimes c)\Delta_{C}(a))^{(-)^{*} \otimes (-)^{*}}$ for all $a,b,c\in A$.
\end{enumerate}
\end{definition}
Here, condition (2) ensures that the map
\begin{align*}
(-)^{*} \otimes (-)^{*}\colon  \AlA \to \ArA, \ a\otimes b \mapsto a^{*} \otimes b^{*}
\end{align*}
is well-defined.
If
  $\mathcal{A}$ is a multiplier Hopf
  $^{*}$-algebroid as above, then its left and right counits $\beps,\epsc$ and its antipode $S$ satisfy
  \begin{align} \label{eq:counit-antipode-involution}
   \epsc\circ * &= *\circ
  S_{B}\circ \beps, & \beps\circ *&=*\circ S_{C}\circ \epsc, & 
  S\circ *\circ S \circ *&=\id_{A};
  \end{align}
see Proposition 6.2 in \cite{timmermann:regular}.

\subsection{Examples of  multiplier Hopf algebroids}
\label{subsection:hopf-examples}

 The following examples  from \cite{timmermann:regular} will be used throughout this article.

\begin{example}[Unital case] \label{example:hopf-unital} We call a multiplier bialgebroid $\mathcal{A}$ \emph{unital} if the algebras $A,B,C$ and the maps $S_{B},S_{C},\Delta_{B},\Delta_{C}$ are unital. In that case, it is easy to see that also the antipode and the left and the right counit are unital.
Such unital multiplier bialgebroids correspond to usual bialgebroids  as defined, for example, in \cite{boehm:bijective,boehm:hopf}, and regular multiplier Hopf algebroids correspond with Hopf algebroids whose antipode is invertible, see \cite[Propositions 3.2, 5.13]{timmermann:regular}.
\end{example} 
\begin{example}[Weak multiplier Hopf algebras] \label{example:wmha}
  Let $(A,\Delta)$ be regular weak multiplier Hopf algebra with counit $\varepsilon$ and antipode $S$; see \cite{daele:weakmult}. Then one can define maps $\varepsilon_{s},\varepsilon_{t} \colon A \to M(A)$ such that
  \begin{align} \label{eq:wmha-source-target}
    \varepsilon_{s}(a)b &= \sum S(a_{(1)})a_{(2)}b, & b\varepsilon_{t}(a) &=  \sum ba_{(1)}S(a_{(2)})
  \end{align}
for all $a,b\in A$. Let $B=\varepsilon_{s}(A)$ and $C=\varepsilon_{t}(A)$. Then the extension of the antipode $S$ to $M(A)$ restricts to anti-isomorphisms $S_{B}\colon B\to C$ and $S_{C} \colon C\to B$.
 Denote by 
$\pi_{B} \colon A\otimes A \to \AlA$ and $\pi_{C} \colon A\otimes A \to \ArA$ the  quotient maps. Then the formulas
\begin{align*}
  \Delta_{B}(a)(b\otimes c) &:=\pi_{B}(\Delta(a)(b \otimes c)), &
  (a\otimes b)\Delta_{C}(c) &= \pi_{C}((a\otimes b)\Delta(c))
\end{align*}
define a left comultiplication $\Delta_{B}$ and a right comultiplication $\Delta_{C}$ such that
\begin{align*}
  \mathcal{A} = (A,B,C,S_{B},S_{C},\Delta_{B},\Delta_{C})
\end{align*}
becomes a regular multiplier Hopf algebroid, see \cite[Theorem 4.8]{daele:relation}. Its antipode coincides with $S$, and its left counit and right counit are given by $\beps = S^{-1}\circ \varepsilon_{t}$ and  $\epsc = S^{-1}\circ \varepsilon_{s}$, respectively.
\end{example}

\begin{example}[Function algebra of an \'etale groupoid] \label{example:hopf-groupoid-functions}
Let $G$ be a locally compact,  Hausdorff groupoid which is \emph{\'etale} in the sense that the source and the target maps $s,t\colon G\to G^{0}$ are open and local homeomorphisms \cite{renault}.
Denote by $C_{c}(G)$ and $C_{c}(G^{0})$ the algebras of compactly supported continuous functions on
$G$ and on $G^{0}$, respectively, by $s^{*},t^{*} \colon C_{c}(G^{0}) \to M(C_{c}(G))$
the pull-back of functions along $s$ and $t$, respectively, let
$A=C_{c}(G)$, $B=s^{*}(C_{c}(G^{0}))$ and $C =t^{*}(C_{c}(G^{0}))$,
and denote by $S_{B},S_{C}$ the isomorphisms $B \rightleftarrows C$ mapping $s^{*}(f)$ to $t^{*}(f)$
and vice versa.  Since $G$ is \'etale, the natural map $A \otimes A \to C_{c}(G\times G)$ factorizes
to an isomorphism
$\bA \otimes \sbA =
 \Asc\otimes \Ac  \cong C_{c}(\GstG)$,
where $\GstG$ denotes the composable pairs of elements of $G$.
Denote by $\Delta_{B}, \Delta_{C}  \colon C_{c}(G) \to M(C_{c}(\GstG))$ the pull-back
of functions along the groupoid multiplication, that is,
  \begin{align*} (\Delta_{B}(f)(g \otimes h))(\gamma,\gamma')
=f(\gamma\gamma')g(\gamma)h(\gamma') = ((g\otimes h)\Delta_{C}(f))(\gamma,\gamma')
  \end{align*}
  for all $f,g,h\in A, \gamma,\gamma' \in G$.  Then  $(A,B,C,S_{B},S_{C},\Delta_{B},\Delta_{C})$ is a
multiplier Hopf $^*$-algebroid with counits and antipode given by
$\beps(f) = s^{*}(f|_{G^{0}})$, $
\epsc(f) = t^{*}(f|_{G^{0}})$, $
  (S(f))(\gamma) =f(\gamma^{-1})$
for all $f\in C_{c}(G)$.
\end{example}

\begin{example}[Convolution algebra of an \'etale groupoid] \label{example:hopf-groupoid-algebra}
Let $G$ be a locally compact, \'etale, Hausdorff groupoid again.
Then the space $C_{c}(G)$ can also be regarded as a $^*$-algebra with
respect to the convolution product and involution given by
\begin{align*}
  (f\ast g)(\gamma) &= \sum_{\gamma=\gamma'\gamma''}
  f(\gamma')g(\gamma''), &
  f^{*}(\gamma) &= \overline{f(\gamma^{-1})}.
\end{align*}
Since $G$ is \'etale, $G^{0}$ is closed and open in $G$, and the function algebra $C_{c}(G^{0})$
embeds into the convolution algebra $C_{c}(G)$. Denote by $\hat A$ this convolution algebra, let
$\hat B=\hat C=C_{c}(G^{0}) \subseteq \hat A$ and let $\hat S_{\hat B}=\hat S_{\hat
  C}=\id_{C_{c}(G^{0})}$. Then the natural map $A \otimes A \to C_{c}(G\times G)$ factorizes to
isomorphisms
\begin{align} \label{eq:groupoid-isomorphism-convolution}
  {_{\hat B}\hat A} \otimes  \gr{\hat A}{}{}{\hat B}{} &\cong
C_{c}(\GttG), & \gr{\hat A}{\hat C}{}{}{} \otimes \hat A_{\hat C}
&\cong C_{c}(\GssG),
\end{align}
and we obtain a   multiplier
Hopf $^*$-algebroid  $(\hat A,\hat B,\hat C,\hat S_{\hat B},\hat
S_{\hat C},\hat \Delta_{\hat B},\hat \Delta_{\hat C})$, where
\begin{align*}
  (\hat \Delta_{\hat B}(f)(g\otimes h))(\gamma',\gamma'') &=
  \sum_{t(\gamma)=t(\gamma')} f(\gamma)g(\gamma^{-1}\gamma')h(\gamma^{-1}\gamma''), \\
  ((g\otimes h)\hat \Delta_{\hat C}(f))(\gamma',\gamma'') &=
  \sum_{s(\gamma)=s(\gamma')} g(\gamma'\gamma^{-1})h(\gamma''\gamma^{-1})f(\gamma).
\end{align*}
Its counits and antipode are given by
\begin{align*}
  (_{\hat B}\hat \varepsilon(f))(u) &= \sum_{t(\gamma')=u} f(\gamma'), 
&
(\hat \varepsilon_{\hat C}(f))(u) &= \sum_{s(\gamma')=u}f(\gamma'),
  &
  (\hat S(f))(\gamma)&= f(\gamma^{-1}).
\end{align*}
\end{example}

\begin{example}[Tensor product] \label{example:hopf-cb}
  Let $B$ and $C$ be non-degenerate,
idempotent algebras with anti-isomorphisms $S_{B} \colon B\to C$ and
$S_{C}\colon C\to B$, form the the tensor
product $A=C\otimes B$, and  identify $B$ and $C$ with their images in
$M(A)$ under the canonical inclusions.
Then we obtain a regular
multiplier Hopf algebroid $(A,B,C,S_{B},S_{C},\Delta_{B},\Delta_{C})$ with  comultiplication, counits and antipode given by
\begin{gather*}
  \begin{aligned}
    \Delta_{B}(y \otimes x)(a \otimes a') &= ya \otimes  xa', & (a
    \otimes a')\Delta_{C}(y \otimes x) &= ay\otimes a' x, &
  \end{aligned} \\
  \begin{aligned}
    \beps(y\otimes x) &= xS_{B}^{-1}(y), & \epsc(y\otimes x)
    &=S_{C}^{-1}(x)y, & S(y\otimes x) &=S_{B}(x)\otimes S_{C}(y)
  \end{aligned}
\end{gather*}
for all $x\in B$, $y\in C$, $a,a'\in A$.
\end{example}

The following example is a special case of the extension of scalars considered in \cite[\S 4.1.5]{boehm:hopf}.
\begin{example}[Symmetric crossed product] \label{example:hopf-ch}
Let $C$ be a non-degenerate, idempotent, commutative algebra with a unital left action of a regular multiplier Hopf algebra $(H,\Delta_{H})$, that is,
\begin{align*}
    h\actright (yy') &= (h_{(1)} \actright y)(h_{(2)} \actright y')
\end{align*}
for all $y,y'\in C$ and $h\in H$ \cite{daele:actions}, and assume that the action is \emph{symmetric} in the sense that
\begin{align} \label{eq:ch-symmetric}
  h_{(1)} \otimes h_{(2)} \actright y = h_{(2)} \otimes h_{(1)} \actright y
\end{align}
for all $h\in H$ and $y\in C$.  If the action is faithful, then symmetry follows easily from commutativity of $C$, but in general, it is an extra assumption and equivalent to the Yetter-Drinfeld condition for the given action and the trivial coaction of $H$ on $C$. As an example, in the case that $H$
is the function algebra of a discrete group $\Gamma$, a symmetric action of $C_{c}(\Gamma)$ corresponds to a grading by  the center of $\Gamma$.

Denote by $A=C\# H$ the  usual smash product or crossed product, that is, the vector space $C\otimes H$ with multiplication given by
\begin{align*}
(y\otimes h)(y' \otimes h') = y (h_{(1)} \actright y')\otimes h_{(2)}h'.
\end{align*}
Then $C$ and $H$ can naturally be identified with  subalgebras of $M(A)$.

We obtain a regular multiplier Hopf algebroid $(A,B,C,S_{B},S_{C},\Delta_{B},\Delta_{C})$, where
 $B=C$, $S_{B}=S_{C}=\id_{C}$ and
\begin{align*}
  \Delta_{B}(yh)(a \otimes a') &= yh_{(1)}a \otimes h_{(2)}a' = h_{(1)}a \otimes yh_{(2)}a', \\
(a\otimes a')  \Delta_{C}(hy) &= ah_{(1)}y \otimes a'h_{(2)} = ah_{(1)}\otimes  a'h_{(2)}y
\end{align*}
for all $y\in C$, $h\in H$ and $a,a'\in A$. Its antipode and counits are  given by
\begin{align} \label{eq:ch-counits-antipode}
  S(yh) &= S_{H}(h)y, & \beps(yh) &= y\varepsilon_{H}(h) = \epsc(hy)
\end{align}
for all $y\in C$ and $h\in H$, where $S_{H}$ and $\varepsilon_{H}$ denote the antipode and counit of $(H,\Delta_{H})$. The verification is straightforward.
  \end{example}

\begin{example}[A two-sided crossed product] \label{example:hopf-chb}
  Let $B$, $C$, $S_{B}$ and $S_{C}$ be as above and let $H$ be a regular multiplier Hopf algebra with a unital
left action on $C$ and a unital right action on $B$ such that  for all $h\in H$, $x,x'\in B$, $y,y'\in C$,
  \begin{align} \label{eq:chb-action-algebra}
  (xx') \actleft h &= (x \actleft h_{(1)})(x'
  \actleft h_{(2)}), 
&  h\actright (yy') &= (h_{(1)} \actright y)(h_{(2)} \actright y'), \\ \label{eq:chb-action-antipode}
 S_{B}(x \actleft h) &= S_{H}(h) \actright S_{B}(x), & S_{C}(h \actright
    y) &= S_{C}(y) \actleft S_{H}(h),
  \end{align}
where  $S_{H}$ denotes the antipode of $H$.
    Then the space $A=C \otimes H \otimes B$ is a non-degenerate, idempotent algebra with
    respect to the product
      \begin{align*} (y \otimes h \otimes x)(y' \otimes h' \otimes x') &=
y(h_{(1)} \actright y') \otimes h_{(2)}h'_{(1)} \otimes (x \actleft h'_{(2)})x'.
    \end{align*}
  The algebras $C,H,B$
  embed naturally into $M(A)$, and  we identify them with their images in $M(A)$. Then the products $yhx,yxh,hyx$
    lie in $A\subseteq M(A)$ for all $x\in B$, $y\in C$ and $h\in H$,
    and we obtain a     regular multiplier Hopf algebroid
    $(A,B,C,S_{B},S_{C},\Delta_{B},\Delta_{C})$, where
    \begin{align*}
      \Delta_{B}(yhx)(a \otimes a') &= yh_{(1)}a \otimes h_{(2)}xa', &
      (a\otimes a')\Delta_{C}(yhx) &= ayh_{(1)} \otimes a'h_{(2)}x
    \end{align*} 
    for all $x\in B$, $y\in C$, $h\in H$, $a,a'\in A$.  Its counits
    and antipode are given by
    \begin{align*}
       \beps(xhy) &= xS_{B}^{-1}(h \actright y), & \epsc(xhy)
&=S_{C}^{-1}(x\actleft h)y, &
S(yhx) &= S_{B}(x)S_{H}(h)S_{C}(y).
    \end{align*} 
\end{example}

\section{Partial integrals and quasi-invariant base weights}
\label{section:integration}

This section  introduces the basic ingredients for integration on a regular multiplier Hopf algebroid $\mathcal{A}=(A,B,C,S_{B},S_{C},\Delta_{B},\Delta_{C})$. These are partial  left and partial right integrals, which are maps
$\cphi\colon A \to C$ and $\bpsib\colon A\to B$  satisfying suitable invariance conditons with respect to the comultiplications, and base weights $(\mu_{B},\mu_{C})$, which are functionals on $B$ and $C$, respectively, that are quasi-invariant with respect to $\cphic$ and $\bpsib$.
We first formulate the appropriate left- and right-invariance for maps from $A$ to $B$ and $C$ (subsection
\ref{subsection:invariant-elements}). In the case of Hopf algebroids, such invariant maps were studied already in
\cite{boehm:integrals}, and yield conditional expectations onto the orbit algebra of $\mathcal{A}$ (subsection \ref{subsection:expectation}).   We then discuss the quasi-invariance assumption  on the functionals $\mu_{B}$ and $\mu_{C}$,
(subsection \ref{subsection:base-weights}) and study the algebraic implications thereof (subsection \ref{subsection:factorizable}). 
Along the way, we keep an eye on the examples of multiplier Hopf algebroids introduced in subsection \ref{subsection:hopf-examples}.

\subsection{Partial integrals}
\label{subsection:invariant-elements}

 We use the notation introduced in Section \ref{section:multiplier-bialgebroids}.
\begin{proposition} \label{proposition:partial-integrals}
Let $\mathcal{A}=(A,B,C,S_{B},S_{C},\Delta_{B},\Delta_{C})$ be a regular multiplier Hopf algebroid.
  \begin{enumerate}
  \item For every linear map $\bpsib\colon A \to B$, the following conditions are equivalent:
    \begin{enumerate}
    \item $\bpsib \in \Hom(\bA,\bB)$ and for all $a,b\in A$,
      \begin{align}\label{eq:partial-left-deltab}
         (\bpsib \otimes \iota)(\Delta_{B}(a)(1 \otimes     b)) = \bpsib(a)b; 
      \end{align}
    \item $\bpsib \in \Hom(\Ab,\Bb)$ and for all $a,b\in A$,
      \begin{align} \label{eq:partial-left-deltac}
(S_{C}^{-1}\circ\bpsib \otimes \iota)((1 \otimes     b)\Delta_{C}(a)) = b\bpsib(a);        
      \end{align}
    \item $\bpsib\in \Hom(\bAb,\bBb)$ and the following diagram commutes:
      \begin{align} \label{dg:strong-invariance-right}
        \xymatrix@R=12pt@C=35 pt{ \AcA \ar[r]^{\lT} \ar[d]_{\Tl\Sigma}
          & \ArA
          \ar[d]^{S\circ (S_{C}^{-1}\circ\bpsib \otimes\iota)}\\
          \AlA \ar[r]_(0.55){\bpsib\otimes \iota} &
          A. }        
      \end{align}
  \end{enumerate}
  \item For every linear map $\cphic \colon A \to C$, the following conditions are equivalent:
    \begin{enumerate}
    \item $\cphic \in \Hom(\cA,\cC)$ and for all $a,b\in A$,
      \begin{align} \label{eq:partial-right-deltab}
        (\iota \otimes     S_{B}^{-1}\circ\cphic)(\Delta_{B}(b)(a\otimes 1))=\cphic(b)a;
      \end{align}
    \item $\cphic \in \Hom(\Ac,\Cc)$ and  for all $a,b\in A$,
      \begin{align} \label{eq:partial-right-deltac}
   (\iota \otimes \cphic)((a\otimes   1)\Delta_{C}(b)) = a\cphic(b);     
      \end{align}
    \item $\cphic \in \Hom(\cAc,\cCc)$ and the following diagram commutes:
      \begin{align} \label{dg:strong-invariance-left}
        \xymatrix@R=12pt@C=35pt{  \AbA \ar[r]^{\Tr}
          \ar[d]_(0.55){\rT\Sigma} & \AlA \ar[d]^(0.55){S\circ (\iota
            \otimes S_{B}^{-1}\circ\cphic)}\\
      \ArA \ar[r]_(0.55){\iota\otimes \cphic} & A. }       
      \end{align}
    \end{enumerate}
    \end{enumerate}
\end{proposition}
\begin{proof}
  We only prove (1) because (2) is similar. If (a) holds, then 
  \begin{align*}
    \bpsib(ax)b = (\bpsib \otimes \id)(\Delta_{B}(ax)(1 \otimes b)) =  (\bpsib \otimes \id)(\Delta_{B}(a)(1 \otimes xb)) = \bpsib(a)xb
  \end{align*}
for all $x\in B$ and $a,b\in A$ by
\eqref{eq:left-delta-bimodule} and hence $\bpsib \in \Hom(\bAb,\bBb)$. A similar application of \eqref{eq:right-delta-bimodule} shows that the same conclusion holds if (b) is satisfied.

Suppose now that $\bpsib\in \Hom(\bAb,\bBb)$. 

We show that (1a) and (1b) are equivalent.  The equations
\eqref{eq:partial-left-deltac} and \eqref{eq:partial-left-deltab} are
equivalent to commutativity of the triangle on the left hand side or
on the right hand side, respectively, in the following diagram:
    \begin{align*}
        \xymatrix@R=10pt@C=45pt{\ABA \ar[dd]_{\rT} \ar[rrr]^{\id
            \oo S}
          \ar[rd]^(0.65){\bpsib \otimes  \iota}
          &&& \AlA           \ar[ld]_(0.65){\bpsib \otimes  \iota}
          \\ &  A \ar[r]^{S} & A &  \\
          \ArA \ar[rrr]_{\id \oo S}
          \ar[ru]_(0.65){S_{C}^{-1} \circ \bpsib \otimes  \iota}
 && &\AbA \ar[uu]_{\Tr} \ar[lu]^(0.65){\bpsib \otimes \iota}}      
\end{align*}
The large rectangle commutes by \eqref{dg:galois-inverse}, and the upper and lower cells commute by inspection. Since all horizontal and vertical arrows are bijections, the triangle on the left hand side commutes if and only if the triangle on the right hand side does.

To see that (1a) and (1c) are equivalent, consider the following diagram:
   \begin{align*}
     \xymatrix@R=15pt@C=45 pt{ \ACA \ar[rr]^{\Tl} \ar[d]_{\lT\Sigma}
       && \AlA
       \ar[d]^{\bpsib\otimes \iota} \\
       \ArA \ar[r]^{\iota \otimes S} \ar[rd]_{S_{C}^{-1} \circ \bpsib
         \otimes\iota}       & \AbA \ar[r]^{\bpsib \oo \id} \ar[ur]^{\Tr} & A \\
 & A \ar[ru]_{S} } 
   \end{align*}
 The upper left cell commutes by \cite[Proposition 5.8]{timmermann:regular}, and the lower  cell commutes by inspection. Therefore, the outer cell commutes if and only if the triangle on  the right commutes. 
\end{proof}
 \begin{definition}\label{definition:invariant-elements}
   Let $\mathcal{A}$ be a regular multiplier Hopf algebroid. A
 \emph{partial right integral} on $\mathcal{A}$ is a map $\bpsib\colon A \to B$ satisfying the equivalent conditions in  Proposition \ref{proposition:partial-integrals} (1), and a
 \emph{partial left integral} on $\mathcal{A}$ is a map $\cphic\colon A \to C$ satisfying the equivalent conditions in  Proposition \ref{proposition:partial-integrals} (2).\end{definition}

Regard $\Hom(\cAc,\cCc)$ as an $M(B)$-bimodule and $\Hom(\bAb,\bBb)$ as an
$M(C)$-bimodule, where $(x \cdot \cphic \cdot x')(a) = \cphic(x' a x)$
and $(y \cdot \bpsib \cdot y') (a) = \bpsib(y' a y)$.
\begin{proposition} \label{proposition:partial-integrals-bimodule-antipode}
  Let $\mathcal{A}$ be a regular multiplier Hopf algebroid.
  \begin{enumerate}
  \item All partial left integrals form an $M(B)$-sub-bimodule of $\Hom(\cAc,\cCc)$;
  \item all partial right integrals form an $M(C)$-sub-bimodule of $\Hom(\bAb,\bBb)$;
  \item the maps $\cphic \mapsto S^{\pm 1} \circ \cphic \circ S^{\mp 1}$   are bijections between all partial left and all partial right integrals.
  \end{enumerate}
\end{proposition} 
\begin{proof}
This follows easily from \eqref{eq:left-delta-bimodule}, \eqref{eq:right-delta-bimodule} and 
\eqref{dg:galois-antipode}.
\end{proof}

Thus, a regular multiplier Hopf algebroid has a surjective partial left integral if and only if it has a  surjective partial right integral.

In the following result, we use the extension of $\Delta_{B}$ and $\Delta_{C}$ to multipliers as described in Section 2 and explained after  \eqref{eq:delta-bimodule-extended}, and  identify $M(B)$ and $M(C)$ with subalgebras of $M(A)$.
\begin{proposition} \label{proposition:ergodic}
Let $\mathcal{A}$ be a regular multiplier Hopf algebroid which has a surjective partial integral.
 Then
  \begin{align*}
    M(B) &= \{ z\in M(A) \cap C' : \Delta_{B}(z) = 1\otimes z\} 
    = \{ z\in M(A) \cap C' : \Delta_{C}(z) = 1\otimes z\}, 
 \\
    M(C) &= \{z\in M(A) \cap B' : \Delta_{B}(z) = z\otimes 1\} = \{z\in M(A) \cap B' : \Delta_{C}(z) =  z \otimes 1\}.
  \end{align*}
\end{proposition}
\begin{proof}
We only prove the first equality. The inclusion $\subseteq$  follows from \eqref{eq:left-delta-bimodule}. To prove the reverse inclusion, suppose that $\bpsib$ is a surjective right integral and that $z \in M(A) \cap C'$ satisfies $ \Delta_{B}(z) = 1\otimes z$. Then for all $a,b\in A$,
\begin{align*}
  \bpsib(za)b &= (\bpsib \otimes \id)(\Delta_{B}(za)(1\otimes b)) = 
  z(\bpsib \otimes id)(\Delta_{B}(a)(1 \otimes b)) = z\bpsib(a)b, \\
  \bpsib(az)b &= (\bpsib \otimes \id)(\Delta_{B}(az)(1\otimes b)) = 
  (\bpsib \otimes id)(\Delta_{B}(a)(1 \otimes zb)) = \bpsib(a)zb, 
\end{align*}
and hence $z\bpsib(a)=\bpsib(za) \in B$ and $\bpsib(a)z=\bpsib(az) \in B$. Since $\bpsib$ is surjective, we can conclude that $z\in M(B)$.
\end{proof}

Let us return to the examples introduced in subsection \ref{subsection:hopf-examples}.
\begin{example} \label{example:integrals-groupoid-functions}
Let $G$ be a locally compact, \'etale Hausdorff groupoid and regard the function algebra $C_{c}(G)$ as a multiplier Hopf $*$-algebroid
as in Example \ref{example:hopf-groupoid-functions}. Then for each  $h\in C(G^{0})$, the maps
$\cphic^{(h)} \colon C_{c}(G) \to t^{*}(C_{c}(G^{0}))$ and $
\bpsib^{(h)} \colon C_{c}(G) \to s^{*}(C_{c}(G^{0}))$
given by
\begin{align*}
  (\cphic^{(h)}(f))(\gamma) &= \sum_{\substack{\gamma' \in G\\ t(\gamma')=t(\gamma)}} f(\gamma')h(s(\gamma')), &
  (\bpsib^{(h)}(f))(\gamma) &= \sum_{\substack{\gamma' \in G \\ s(\gamma')=s(\gamma)}} f(\gamma')h(t(\gamma'))
\end{align*}
are a partial left and a partial right integral, respectively.
\end{example}

\begin{example} \label{example:integrals-groupoid-algebra}
Let $G$ be as above and regard the convolution algebra $C_{c}(G)$ as a multiplier Hopf $*$-algebroid
as in Example \ref{example:hopf-groupoid-algebra}. Then for each $h\in C(G^{0})$, the map
$\cphic^{(h)} \colon C_{c}(G) \to C_{c}(G^{0})$
given by\begin{align*}
  (\cphic^{(h)}(f))(u) &=  f(u)h(u)
\end{align*}
is easily seen to be a partial left and a partial right integral.
\end{example}

\begin{example} \label{example:integrals-cb} 
Consider  the tensor product $A=C\otimes B$ discussed in Example \ref{example:hopf-cb}. For all $\upsilon \in \dual{B}$ and $\omega \in \dual{C}$, the maps 
  \begin{align*}
     \cphic^{(\upsilon)} &:= \id \otimes \upsilon \colon A\to C, &
     \bpsib^{(\omega)} &:=\omega \otimes \id \colon A \to B
  \end{align*}
are a partial left and a partial right integral, respectively, as one can easily check again.  
\end{example}

\begin{example} \label{example:integrals-ch}
  Consider the symmetric crossed product $A=C\# H$ constructed in Example \ref{example:hopf-ch}. If $\phi_{H}$ is a left integral and $\psi_{H}$ is a right integral for the multiplier Hopf algebra $H$,  then the maps
  \begin{align} \label{eq:integrals-ch}
    \cphic &\colon C\# H \to C, \ y h \mapsto y\phi_{H}(h),  &
    \bpsib&\colon C\# H\to C=B, \ yh\mapsto y\psi_{H}(h)
  \end{align}
are a partial left and a partial right integral, respectively. Note that 
\begin{gather} \label{eq:integrals-ch-symmetric}
  \begin{aligned}
    \cphic(hy) &= (h_{(1)}\actright y)\phi_{H}(h_{(2)}) = y\phi_{H}(h), \\
    \bpsib(hy) &= (h_{(2)} \actright y)\psi_{H}(h_{(1)}) =
    y\psi_{H}(h).
  \end{aligned}
\end{gather}
\end{example}

\begin{example}
  \label{example:integrals-chb} Consider the two-sided crossed product $A=C\# H \# B$ discussed in Example \ref{example:hopf-chb}.
Suppose that the multiplier Hopf algebra $H$ has a left integral $\phi_{H}$, and let $\upsilon \in \dual{B}$. Then the map
  \begin{align*}
    \cphic^{(\upsilon)}  \colon A \to C, \quad yhx \mapsto y\phi_{H}(h)\upsilon(x),
  \end{align*}
is a partial left integral.  Indeed,  clearly $\cphic^{(\upsilon)} \in  \Hom(\cA,\cC)$, and
  \begin{multline*}
    (\id \otimes S_{B}^{-1}\cphic^{(\upsilon)})(\Delta_{B}(yhx)(y'h'y' \otimes 1)) = 
      y(h_{(1)} \actright y')h_{(2)}h'x' \cdot \phi_{H}(h_{(3)})\upsilon(x) \\
      = yy'h'x' \phi_{H}(h)\upsilon(x) = \cphic^{(\upsilon)}(yhx)y'h'x'
  \end{multline*}
for all $y,y'\in C$, $x,x'\in B$ and $h,h' \in H$  by left-invariance of $\phi_{H}$. 

 Similarly, if $\psi_{H}$ is a right integral of $H$ and if  $\omega\in \dual{C}$, then the map
  \begin{align*}
    \bpsib^{(\omega)} \colon A \to B, \quad yhx \mapsto \omega(y)\psi_{H}(h)x,
  \end{align*}
is right-invariant. 
\end{example}

\subsection{Expectations onto the orbit algebra in the proper case}
\label{subsection:expectation}

We show that for a proper regular multiplier Hopf algebroid,  partial integrals restrict to conditional expectations  to the orbit algebra and are completely determined by these restrictions.  Let us first define these terms.
 \begin{definition} \label{definition:proper}
We call  a regular multiplier Hopf algebroid $\mathcal{A}$ as above 
 \emph{proper} if
   \begin{align*}
   BC\subseteq A \text{ in } M(A).
   \end{align*}
  The \emph{orbit algebra} of $\mathcal{A}$  is the subalgebra
  \begin{align*}
   \mathcal{O}:=M(B) \cap M(C) \subseteq M(A).
  \end{align*}
   We call $\mathcal{A}$ \emph{ergodic} if $\mathcal{O}=\C 1$.
 \end{definition}
Clearly, the orbit algebra $\mathcal{O}$ is central in $M(B)$ and in $M(C)$. 
\begin{proposition} \label{proposition:orbit-antipode}
The antipode $S$ of a regular multiplier Hopf algebroid $\mathcal{A}$ acts trivially on its orbit algebra.
\end{proposition}
\begin{proof}
Let $z\in \mathcal{O}$. By Proposition \ref{proposition:ergodic}, we have for all $a,b\in A$
\begin{align*}
a\otimes S_{B}(z)b = za\otimes b = \Delta_{B}(z)(a\otimes b) = a\otimes zb
\end{align*}
in $\AlA$. We apply $\beps \otimes \id$, use the relations $\beps(A)A=A$ and $z\in C' \cap B' \cap M(A)$, and find $S_{B}(z)c=zc$ for all $c\in A$.
\end{proof}
In the proper case, partial integrals extend to conditional expectations on the base algebras and are completely determined by these extensions.
\begin{proposition} \label{proposition:proper-partial-integrals}
  Let $\mathcal{A}$ be  a proper regular multiplier Hopf algebroid with a partial left integral $\cphic$ and a partial right integral $\bpsib$.  Then the extensions $\bpsib|_{C} \colon C \to ZM(B)$ and $\cphic|_{B} \colon B\to ZM(C)$ 
defined by \begin{align*}
   \bpsib|_{C}(y)x =\bpsib(yx) \quad\text{and} \quad  \cphic|_{B} (x)y = \cphic(xy)
\end{align*}
 take values in the orbit algebra $\mathcal{O}$ and we have
  \begin{align} \label{eq:proper-composed-partial-integrals}
    \cphic|_{B} \circ \bpsib = \bpsib|_{C}\circ \cphic
  \end{align}
as maps from $A$ to $\mathcal{O}$.
\end{proposition}
\begin{proof}
To see that the extension  $\cphic|_{B}$ takes values in $\mathcal{O}$, let $a\in A$,  $x,x'\in B$, $y\in C$. Then
\begin{align*}
  \cphic|_{B}(x)yx'a=\cphic(xy)x'a&= (\id \otimes S_{B}^{-1}\cphic)(\Delta_{B}(xy)(x'a \otimes 1)) \\ &= 
(\id \otimes S_{B}^{-1}\cphic)(ya \otimes xS_{B}(x')) \\ &  = S_{B}^{-1}(\cphic(xS_{B}(x')))ya =
S_{B}^{-1}(\cphic|_{B}(x))x'ya,
\end{align*}
  and hence $\cphic|_{B}(x) \in \mathcal{O}$.
  A similar argument shows that $\bpsib|_{C}(y) \in \mathcal{O}$ for all $y\in C$.

To prove \eqref{eq:proper-composed-partial-integrals},   let $a\in A$, $x,x',x''\in B$ and $y\in C$.  Then the expression
  \begin{align*}
        (\bpsib \otimes S_{B}^{-1}\cphic)(\Delta_{B}(a)(yx' \otimes S_{B}(x'')x)
  \end{align*}
is equal to
  \begin{align*}
    \bpsib(\cphic(ax)x''yx') =    \bpsib|_{C}(\cphic(axy))x''x'
\end{align*}
by left-invariance of $\cphic$, and to
\begin{align*}
 S_{B}^{-1}\cphic(\bpsib(ay)S_{B}(x''x')x)  =
S_{B}^{-1}(\cphic|_{B}(\bpsib(axy)))x''x'
  \end{align*}
by right-invariance of $\bpsib$. Hence,  $\bpsib|_{C}\circ \cphic = S^{-1}_{B} \circ \cphic|_{B}\circ \bpsi$. With Proposition \ref{proposition:orbit-antipode}, the claim follows.
\end{proof}
\begin{remark} \label{remark:proper-ergodic}
  \begin{enumerate}
  \item If there are sufficiently many partial integrals in the sense that  (extensions of) the partial  left integrals separate the points of $B$, which by Proposition \ref{proposition:partial-integrals-bimodule-antipode} is equivalent to partial right integrals separating the points of $C$, then \eqref{eq:proper-composed-partial-integrals} implies that every partial left integral $\cphic$  and every partial right integral $\bpsib$ is uniquely determined by the extension $\cphic|_{B}$ or  $\bpsib|_{C}$, respectively.
  \item In the proper and ergodic case, every partial left integral    $\cphic$  determines a functional $\mu_{B}$ on $B$ such that
    $\cphic|_{B}(x)=\mu_{B}(x)1 \in M(A)$, and every partial right integral    $\bpsib$ determines a functional $\mu_{C}$ on $C$ such that
    $\bpsib|_{C}(y)=\mu_{C}(y)1 \in M(A)$. Under the assumption in (1),  $\cphic$ and $\bpsib$
    are uniquely determined by these functionals.
  \end{enumerate}
\end{remark}

\subsection{Quasi-invariant base weights}
\label{subsection:base-weights}

  Let $\mathcal{A}$ be a regular multiplier Hopf algebroid, not necessarily proper. To obtain  total integrals  on $\mathcal{A}$, we  compose partial integrals  with suitable functionals $\mu_{B}$ and $\mu_{C}$  on the base algebras $B$ and $C$, respectively. On these functionals, we impose several conditions. 
\begin{definition}
 A \emph{base weight}   for $\mathcal{A}$ is  a pair of faithful functionals $(\mu_{B},\mu_{C})$ on  the algebras $B$ and $C$, respectively. We call such a base weight
 \begin{enumerate}
 \item \emph{antipodal} if
$ \mu_{B} \circ S_{C} = \mu_{C}$ and $\mu_{C} \circ     S_{B} = \mu_{B}$,
\item 
 \emph{modular} if $\sigma_{B}:=S_{B}^{-1}S_{C}^{-1}$ and $\sigma_{C}:=S_{B}S_{C}$ are modular automorphisms of $\mu_{B}$ and $\mu_{C}$, respectively;
\item \emph{positive} if $\mathcal{A}$ is a multiplier Hopf  $*$-algebroid and $\mu_{B}$ and $\mu_{C}$ are positive.
 \end{enumerate}
\end{definition}
 Conditon (1)  is quite natural. Like (2), it implies
\begin{align}
  \label{eq:base-weight-invariance}
  \mu_{B} &= \mu_{B} \circ S_{C} \circ S_{B} &&\text{and} & \mu_{C} &= \mu_{C} \circ S_{B} \circ S_{C}.
\end{align}
We shall later introduce another condition, counitality, which implies condition (2) and in the unital case also (1).  
\begin{definition}
We call a base weight $(\mu_{B},\mu_{C})$ \emph{quasi-invariant} with respect to
\begin{enumerate}
\item a partial left integral $\cphic$ if the functional $\phi:=\mu_{C} \circ \cphic$ can be written
  \begin{align} \label{eq:left-integral} \phi =
    \mu_{B} \circ \bphi = \mu_{B} \circ \phib
  \end{align}
with maps $\bphi 
  \in \Hom(\bA,\bB)$ and $\phib \in \Hom(\Ab,\Bb)$;
\item  a partial right integral $\bpsib$
 if the functional $\psi:=\mu_{B}\circ \bpsib$ can be written
    \begin{align} \label{eq:right-integral}
\psi = \mu_{C} \circ \cpsi = \mu_{C} \circ   \psic
    \end{align}
with maps  $\cpsi \in \Hom(\cA,\cC)$ and $\psic
    \in \Hom(\Ac,\Cc)$.
\end{enumerate}
We call a functional $\phi$ or $\psi$ of the form above a \emph{total left} or \emph{total right integral}.
\end{definition}

We next consider some special cases and   examples. 
In these examples,  we shall frequently deduce quasi-invariance from the existence of certain modular multipliers:
\begin{lemma} \label{lemma:quasi-modular}
Let $\cphic$ be a partial left  and $\bpsib$   a partial right integral for $\mathcal{A}$.
Suppose that  $(\mu_{B},\mu_{C})$ is a base weight and that there exist invertible multipliers $\delta \in R(A)$, $\delta'\in L(A)$ such that the functionals  $\phi:=\mu_{C}\circ \cphic$ and $\psi:=\mu_{B}\circ \bpsib$ satisfy
    \begin{align*}
\psi(a\delta)=      \phi(a)=\psi(\delta'a)
    \end{align*}
 for all $a\in A$. 
 Then    $(\mu_{B},\mu_{C})$ is quasi-invariant with respect to $\cphic$ and  $\bpsib$.
\end{lemma}
\begin{proof}
 The maps $\bphi,\phib,\cpsi,\psic$ defined by the formulas $\bphi(a):=\bpsib(a\delta)$, $\phib:=\bpsib(\delta'a)$, $\cpsi(a):=\cphic(a\delta^{-1})$ and $\psic(a):=\cphic(\delta'{}^{-1}a)$ satisfy  \eqref{eq:left-integral} and \eqref{eq:right-integral}.
\end{proof}
Under suitable assumptions, a converse holds; see Corollary \ref{corollary:modular-element} and Theorem \ref{theorem:modular-element-second}.

In the proper  case, quasi-invariant base weights can be constructed by an analogue of \cite[Proposition 3.6]{renault} from functionals on the orbit algebra $\mathcal{O}=M(B)\cap M(C)$.  If $\mathcal{A}$ is also ergodic, this choice picks just a scalar; see Remark \ref{remark:proper-ergodic}.
\begin{lemma} \label{lemma:quasi-proper}
Let $\mathcal{A}$ be a proper regular multiplier Hopf algebroid with a partial left integral
 $\cphic$, a partial right integral $\bpsib$ and a faithful functional $\tau$ on its orbit algebra.   Suppose that
  the extensions $\cphic\big|_{B}$ and $\bpsib\big|_{C}$ are faithful. 
\begin{enumerate}
\item
  The functionals $\mu_{B}:=\tau\circ\cphic\big|_{B}$ and
    $\mu_{C}:=\tau\circ \bpsib\big |_{C}$ form a base weight and
    \begin{align*}
      \mu_{B} \circ \bpsib = \mu_{C} \circ \cphic.
    \end{align*}
In particular, $(\mu_{B},\mu_{C})$ is
     quasi-invariant with respect to $\cphic$ and $\bpsib$.
   \item This base weight is  antipodal if and only if   
     $S^{-1}\circ \cphic \circ S = \bpsib = S\circ \cphic \circ S^{-1}$.
  \end{enumerate}
\end{lemma}
\begin{proof}
(1)
First, observe that the functionals $\mu_{B}$ and $\mu_{C}$ are faithful. For example, if $\mu_{B}(xx')=0$ for some $x\in B$ and all $x'\in B$, then $\tau(\cphic|_{B}(xx')z) = \mu_{B}(xx'z)=0$ for all $z\in \mathcal{O}$ and by faithfulness of $\tau$ and $\cphic$, we first conclude that $\cphic|_{B}(xx')=0$ for all $x'\in B$ and then that $x=0$. 

Next, Proposition \ref{proposition:proper-partial-integrals} (2)  implies
\begin{align*}
\mu_{C}   \circ \cphic  &= \tau \circ \bpsib\big|_{C} \circ \cphic = \tau \circ \cphic\big|_{B} \circ \bpsib = \mu_{B} \circ \bpsib.  \end{align*}

(2)  Suppose that $(\mu_{B},\mu_{C})$ is antipodal.  Then  by Proposition \ref{proposition:orbit-antipode}, 
\begin{align*}
  \tau(z\cphic(S_{C}(y))) = \mu_{B} (S_{C}(zy)) = \mu_{C}(zy) = \tau(z\bpsib(y))
\end{align*}
for all $z\in \mathcal{O}$ and $y\in C$, whence $\cphic\big|_{B} \circ S_{C} = \bpsib\big|_{C}$. By Proposition \ref{proposition:partial-integrals-bimodule-antipode} (3), \ref{proposition:orbit-antipode} and   Remark \ref{remark:proper-ergodic} (1), we can conclude that $S^{-1}\circ \cphic \circ S=\bpsib$. Similarly, the relation $\mu_{B} \circ S_{C}^{-1} = \mu_{C}$ implies that $S^{-1}\circ \cphic \circ S=\bpsib$.

The converse implication follows easily from the definitions and  Proposition \ref{proposition:orbit-antipode}.
\end{proof}

In the case of  regular multiplier Hopf algebroids arising from weak multiplier Hopf algebras as in  Example \ref{example:wmha}, there exists a canonical base weight which satisfies all of our conditions.
\begin{lemma} \label{lemma:quasi-wmha}
Suppose that $\mathcal{A}$ is a  regular multiplier Hopf algebroid associated to a
 regular weak multiplier Hopf algebra $(A,\Delta)$. 
\begin{enumerate}
\item The functionals on $B=\varepsilon_{s}(A)$ and $C=\varepsilon_{t}(A)$ defined by
\begin{align} \label{eq:weak-base-weight}
  \mu_{B}(\varepsilon_{s}(a)) &= \varepsilon(a) & &\text{and} &
  \mu_{C}(\varepsilon_{t}(a)) &= \varepsilon(a)
\end{align}
 form an antipodal, modular base weight that is quasi-invariant with respect to every partial   integral.
\item The assignment $\cphic\mapsto \mu_{C}\circ \cphic$ defines a bijection between all partial left integrals $\cphic$ and all functionals $\phi$ on $A$ satisfying
\begin{align*}
  (\id \otimes \phi)((b\otimes 1)\Delta(a)) &= (\id \otimes \phi)((b\otimes a)E)  \quad \text{for all } a,b\in A,
\end{align*}
where $E=\Delta(1) \in M(B\otimes C)$ is the canonical separability idempotent of $(A,\Delta)$. 
\end{enumerate}
\end{lemma}
Assertion (1) is a particular case of a more general result, see Example \ref{example:factorization-frobenius} and the comments there.   Of course,  (2) has an analogue for right integrals.
\begin{proof}
(1)  The results in  \cite[Propositions 1.7, 2.1 and 4.8]{daele:separability} and the remarks    following  Definition 2.2 in \cite{daele:separability} show that   the functionals $\mu_{B}$ and $\mu_{C}$ form an antipodal, modular base weight and that the  separability idempotent $E$  satisfies
\begin{align} \label{eq:wmha-separability-two}
  (\mu_{B} \otimes \id)(E) &= \id_{C}, &
  (\id \otimes \mu_{C})(E) &=\id_{B}, \\
 \label{eq:wmha-separability-one}
  (\mu_{B} \otimes \id)(E(x\otimes 1)) &= S_{B}(x), &
  (\id \otimes \mu_{C})((1\otimes y)E) &= S_{C}(y).
\end{align}

Let $\cphic$ be a partial left integral, write $\phi=\mu_{C} \circ \cphic$ and define $\bphi,\phib\colon A\to M(B)$ by
\begin{align*}
  \phib(a) &:= (\phi \otimes S_{C})((a \otimes 1)E), & \bphi(a) &:= (\phi \otimes S_{B}^{-1})(E(a\otimes 1)).
\end{align*}
These maps take values in $B$ because $(a\otimes 1)E$ and $E(a\otimes 1)$ lie in $A \otimes C$. By antipodality of $(\mu_{B},\mu_{C})$ and by \eqref{eq:wmha-separability-two},
\begin{align*}
  \mu_{B} (\phib(a)) = (\phi \otimes \mu_{C})((a\otimes 1)E) = \phi(a) = (\phi \otimes \mu_{C})(E(a\otimes 1)) = \mu_{B}(\bphi(a)).
\end{align*}
Moreover, by \eqref{eq:wmha-separability-one},
\begin{align*}
  \phib(aS_{C}(y)) &= (\phi \otimes S_{C})((aS_{C}(y)\otimes 1)E) \\ & = (\phi\otimes S_{C})((a\otimes y)E) = (\phi \otimes S_{C})((a\otimes 1)E)S_{C}(y) = \phib(a)S_{C}(y),
\end{align*}
and a similar calculation shows that $\bphi(xa)=x\bphi(a)$ for all $a\in A$, $x\in B$. 

A similar argument shows that $(\mu_{B},\mu_{C})$ is quasi-invariant with respect to every partial right integral.

(2) Let $\phi$ be a functional on $A$. A similar argument as above shows that we can  write $\phi=\mu_{C} \circ \phic$ with some $\phic \in \Hom(\Ac,\Cc)$.     Let  now $a,b\in A$ and write
\begin{align*}
  (a \otimes 1)\Delta(b)=\sum_{i} c_{i} \otimes d_{i}
\end{align*}
with $c_{i},d_{i} \in A$.  
Then $(a\otimes 1)\Delta_{C}(b) = \sum_{i} c_{i} \otimes d_{i} \in \ArA$.
Since $\Delta(b)=\Delta(b)E$ and $\phi(ey)=\mu_{C}(\phic(e)y)$ for all $e\in A$ and $y\in C$,
\begin{align*}
  (\id \otimes \phi)((a \otimes 1)\Delta(b)) &= \sum_{i} (\id \otimes \phi)((c_{i} \otimes d_{i})E) 
 = \sum_{i} c_{i}(\id \otimes \mu_{C})((1 \otimes \phic(d_{i}))E).
\end{align*}
But   \eqref{eq:wmha-separability-one} implies that 
$(\id \otimes \mu_{C})((1 \otimes \phic(d_{i}))E) = S_{C}(\phic(d_{i}))$ and hence
\begin{align*}
  (\id \otimes \phi)((a \otimes 1)\Delta(b))  &=\sum_{i} c_{i}S_{C}(\phic(d_{i})) = (\id \otimes \phic)((a\otimes 1)\Delta_{C}(b)). \end{align*}
On the other hand, a similar calculation shows that
$(\id \otimes \phi)((a\otimes b)E) = aS_{C}(\phic(b))$. The assertion follows.
\end{proof}

Let us next consider the examples introduced in Subsection  \ref{subsection:hopf-examples}.
\begin{example}  \label{example:quasi-groupoid-functions}
Consider the function algebra of a  locally compact, \'etale Hausdorff groupoid $G$; see Example \ref{example:hopf-groupoid-functions}. 
Suppose that $\mu$ is a Radon measure on the space of units $G^{0}$ with full support.
 Then the functionals $\mu_{B}$ and $\mu_{C}$ given on  $ B=s^{*}(C_{c}(G^{0}))$ and $C=t^{*}(C_{c}(G^{0}))$ by
\begin{align} \label{eq:base-weight-groupoid-functions}
  \mu_{B}(s^{*}(f)) = \int_{G^{0}} f \intd \mu = \mu_{C}(t^{*}(f))
\end{align}
for all $f\in C_{c}(G^{0})$  form  an antipodal, modular, positive base weight. 

  Consider the partial left  integral $\cphic$ and the partial right integral $\bpsib$
given by
\begin{align} \label{eq:integrals-groupoid-functions}
  (\cphic(f))(\gamma) &= \sum_{\substack{\gamma' \in G\\ t(\gamma')=t(\gamma)}} f(\gamma'), &
  (\bpsib(f))(\gamma) &= \sum_{\substack{\gamma' \in G \\ s(\gamma')=s(\gamma)}} f(\gamma').
\end{align}
The compositions $\phi:=\mu_{C} \circ \cphic$ and $\psi:=\mu_{B} \circ \bpsib$ correspond to integration with respect to the measures $\nu$ and $\nu^{-1}$ on $G$ defined by
   \begin{align} \label{eq:dnu}
  \int_{G} f \intd \nu  &=  \phi(f)= \int_{G^{0}} \sum_{r(\gamma)=u} f(\gamma) \intd\mu(u), &
\int_{G} f \intd \nu^{-1} &= \psi(f)=\int_{G^{0}} \sum_{s(\gamma) = u} f(\gamma) \intd\mu(u).
\end{align}

Assume that the  measure $\mu$ on $G^{0}$  is \emph{continuously quasi-invariant} in the sense that the measures $\nu$ and $\nu^{-1}$ on $G$ are related by a  continuous Radon-Nikodym derivative, 
\begin{align*}
\nu  =  D\nu^{-1} \text{ for some } D \in C(G).
\end{align*}
Then $\phi(f)=\psi(fD)$ for all $f \in C_{c}(G)$ and hence $(\mu_{B},\mu_{C})$ is quasi-invariant with respect to $\cphic$ and $\bpsib$.
Conversely, we shall see in Example \ref{example:modular-groupoid-functions} that if $(\mu_{B},\mu_{C})$  is quasi-invariant  with respect to $\cphic$ (or $\bpsib$), then $\mu$ is continuously quasi-invariant.
\end{example}

\begin{example} \label{example:quasi-groupoid-algebra}
Let $G$  be as above and consider the convolution algebra $C_{c}(G)$ as in Example \ref{example:hopf-groupoid-algebra}. Suppose that 
 $\mu$ is a Radon measure on the space of units $G^{0}$ with full support and consider the functional
\begin{align} \label{eq:base-weight-groupoid-algebra}
\hat  \mu\colon C_{c}(G^{0}) \to \C, \quad f\mapsto \int_{G^{0}} f \intd \mu.
\end{align}
Since both base algebras $\hat B$ and $\hat C$ coincide with $C_{c}(G^{0})$, the pair $(\hat \mu,\hat \mu)$ is an antipodal, modular, positive base weight, and quasi-invariant with respect to every partial integral.
\end{example}

In the following example, we use the preorder $\lesssim$ on $\dual{B}$ defined  in \eqref{eq:lesssim}.
\begin{example} \label{example:quasi-cb}
Consider  a tensor product $A=C\otimes B$ as discussed in Example \ref{example:hopf-cb}, and
the partial integrals
  \begin{align*}
     \cphic^{(\upsilon)} &:= \id \otimes \upsilon \colon A\to C, &
     \bpsib^{(\omega)} &:=\omega \otimes \id \colon A \to B
  \end{align*}
associated to arbitrary functionals $\upsilon$ and $\omega$  on $B$ and $C$, respectively, as in Example \ref{example:integrals-cb}.
Then a base weight $(\mu_{B},\mu_{C})$ is quasi-invariant with respect to $\cphic^{(\upsilon)}$  if and only if $\upsilon\lesssim \mu_{B}$. Indeed, if $\upsilon \lesssim \mu_{B}$ and $\delta\in R(B)$ and $\delta' \in L(B)$ are as in Lemma \ref{lemma:lesssim}, then 
\begin{align*}
  \mu_{C}(\cphic^{(\upsilon)}(y\otimes x)) = \mu_{C}(y)\upsilon(x) = \mu_{B}(x\mu_{C}(y)\delta) = \mu_{B}(\mu_{C}(y)\delta'x),
\end{align*}
and the maps $\bphi \colon y\otimes x\mapsto \mu_{C}(y)x\delta$ and $\phib\colon y\otimes x\mapsto \mu_{C}(y)\delta'x$ satisfy \eqref{eq:base-weight-invariance}. Conversely, if $(\mu_{B},\mu_{C})$ is quasi-invariant with respect to $\cphic^{(\upsilon)}$ and if $y\in C$ is chosen such that $\mu_{C}(y)=1$, then the formulas
\begin{align*}
x  \delta &:= \bphi(y\otimes x) &&\text{and} & \delta'x:= \phib(y\otimes x)
\end{align*}
define a right and a left   multiplier of $B$ such that $\mu_{B}(x\delta) =  \upsilon(x) = \mu_{B}(\delta'x)$.

Similarly, $(\mu_{B},\mu_{C})$ is quasi-invariant with respect to $\bpsib^{(\omega)}$ if and only if $\omega \lesssim \mu_{C}$.
\end{example}

\begin{example} \label{example:quasi-ch}
  Consider a symmetric crossed product $A=C\# H$ as discussed in Example \ref{example:hopf-ch}.  Suppose that  $\mu$  is a faithful functional on $C$ and that $\phi_{H}$ is a left  and $\psi_{H}$ is a right integral on $(H,\Delta_{H})$. Then $(\mu,\mu)$ is an antipodal, modular base weight and as such quasi-invariant with respect to the partial integrals $\cphic$ and $\bpsib$ defined  in Example \ref{example:integrals-ch}.
Indeed  by \cite[Theorem 3.7, Propositions 3.10, 3.12]{daele}, there exist invertible multipliers $\delta_{H},\delta'_{H}$ such that $\psi_{H}=\delta_{H} \cdot \phi_{H} = \phi_{H} \cdot \delta'_{H}$,  the compositions $\phi=\mu \circ \cphic$ and $\psi=\mu \circ \bpsib$ then satisfy
    \begin{align*}
      \psi(yh \delta_{H}) & = \mu(y)\psi_{H}(h\delta_{H}) = \mu(y)\phi_{H}(h) = \phi(yh)
    \end{align*}
    for all $y\in C$ and $h\in H$,    and a similar calculation using \eqref{eq:integrals-ch-symmetric} shows that $\phi \cdot \delta'_{H} = \psi$. Now, quasi-invariance follows from Lemma \ref{lemma:quasi-modular}.
\end{example}

\begin{example} \label{example:quasi-chb}
  Consider a two-sided crossed product $A=C\# H \# B$ as discussed in Example \ref{example:hopf-chb}.
Suppose that $\mu_{B}$ and $\mu_{C}$ are faithful functionals on $B$ and $C$, respectively, which are $H$-invariant in the sense that
\begin{align} \label{eq:base-invariant}
  \mu_{C}(h\actright y) &= \varepsilon_{H}(h) \mu_{C}(y),  &
  \mu_{B}(x\actleft h)&=\varepsilon_{H}(h)\mu_{B}(x)
\end{align}
for all $y\in C$, $x\in B$ and $h\in H$. 
Then the base weight  $(\mu_{B},\mu_{C})$ is quasi-invariant with respect to the partial integrals $\cphic$ and $\bpsib$ defined by
\begin{align*}
  \cphic(yhx) &:= y\phi_{H}(h) \mu_{B}(x) &&\text{and} &
  \bpsib(yhx) &:=\mu_{C}(y)\phi_{H}(h)x;
\end{align*}
see also Example \ref{example:integrals-chb}. Indeed
if $\delta_{H},\delta'_{H} \in M(H)$ are as in the preceding example, then \eqref{eq:base-invariant} implies that the compositions $\phi=\mu_{C} \circ \cphic$ and $\psi = \mu_{B} \circ \bpsib$ satisfy
$\phi(a)=\psi(a\delta_{H})=\psi(\delta'_{H}a)$ for all $a\in A$, and quasi-invariance follows from Lemma \ref{lemma:quasi-modular}.
\end{example}

\subsection{Factorizable functionals}
\label{subsection:factorizable}
We now focus on the algebraic aspects of the quasi-invariance condition introduced above. The proper context for this are functionals on bimodules. For the moment,  $B$ and $C$ denote arbitrary non-degenerate, idempotent algebras, not necessarily coming from a regular multiplier Hopf algebroid.

\begin{definition}
  Let $B$ and $C$ be non-degen\-erate, idempotent algebras with faithful functionals $\mu_{B}$ and $\mu_{C}$, respectively, and  let $M$   be an idempotent $B$-$C$-bimodule. We call a functional $\omega \in \dual{M}$ \emph{factorizable} (with respect to $\mu_{B}$ and $\mu_{C}$) if there 
exist   maps $\bomega  \in \Hom(_{B}M,{_{B}B})$ and $\omegac \in \Hom(M_{C},C_{C})$  such that
\begin{align}
  \label{eq:1}
  \mu_{B} \circ \bomega = \omega = \mu_{C} \circ \omegac.
\end{align}
We denote by $M^{\sqcup} \subseteq \dual{M}$ the subspace of all such factorizable functionals.
\end{definition}
Using the fact that $\mu_{B}$ and $\mu_{C}$ are faithful and that the bimodule $M$ is idempotent,  one can reformulate the condition above as follows.
A functional $\omega\in \dual{M}$ is factorizable if and only if there exist maps $\bomega \colon M\to B$ and $\omegac \colon M \to C$ such that
\begin{align*}
  \omega(xm) &= \mu_{B}(x\bomega(m)) &&\text{and} &
\omega(my) &= \mu_{C}(\omegac(m)y)
\end{align*}
for all $x\in B$, $y\in C$ and $m\in M$.
Note that such maps, if they exist, are  uniquely  determined by $\omega$.

The assignment $M \mapsto M^{\sqcup}$ is  functorial. Indeed,
if $T\colon M\to N$ is a morphism of idempotent $B$-$C$-bimodules, then the dual map
$ \dual{T} \colon \dual{N} \to \dual{M}$   restricts to a map
\begin{align*}
  T^{\sqcup} \colon N^{\sqcup} \to M^{\sqcup}, \quad \omega \mapsto \omega\circ T.
\end{align*}

The key property of factorizable functionals is that one can form \emph{slice maps} and \emph{tensor products} for such functionals   as follows.

\begin{lemma} \label{lemma:factorizable-slice-tensor}
Let $_{B}M_{C}$ be an idempotent $B$-$C$-bimodule and $_{C}N_{D}$  an idempotent $C$-$D$-bimodule, where $B,C,D$ are non-degenerate, idempotent algebras with faithful functionals $\mu_{B},\mu_{C},\mu_{D}$, respectively,  and consider the balanced tensor product   $_{B}M_{C}\otimes {_{C}N_{D}}$.
Suppose that $\upsilon \in \dual{M}$ and $\omega\in \dual{N}$ are factorizable. 
  \begin{enumerate}
  \item The formulas
    \begin{align*}
      (\upsilon \underset{\mu_{C}}{\otimes} \id)(m \otimes n) &:=\upsilonc(m)n, &
(\id \underset{\mu_{C}}{\otimes} \omega)(m\otimes n) &:=m\comega(n)
    \end{align*}
define  morphisms of modules
\begin{align*}
  \upsilon \underset{\mu_{C}}{\otimes} \id & \colon M_{C}\otimes {_{C}N_{D}} \to N_{D}, &
\id \underset{\mu_{C}}{\otimes} \omega  &\colon {_{B}M_{C}} \otimes {_{C}N} \to {_{B}M}.
\end{align*}
  \item  The formula
    \begin{align*}
      (\upsilon \underset{\mu_{C}}{\otimes} \omega)(m \otimes
      n)=\mu_{C}(\upsilonc(m)\comega(n))
    \end{align*}
    defines a factorizable functional $\upsilon \underset{\mu_{C}}{\otimes}
    \omega$ on $M_{C}\otimes {_{C}N} \to N$, and
    \begin{align*}
      _{B}(\upsilon \underset{\mu_{C}}{\otimes} \omega) &= {_{B}\upsilon}
      \circ (\id \underset{\mu_{C}}{\otimes} \omega), & (\upsilon
      \underset{\mu_{C}}{\otimes} \omega)_{D} &= \omega_{D} \circ (\upsilon
      \underset{\mu_{C}}{\otimes} \id).
    \end{align*}
  \end{enumerate}
\end{lemma}
 The proof is straightforward and left to the reader. Note that
 \begin{align*}
\upsilon \circ (\id \underset{\mu_{C}}{\otimes} \omega)=   (\upsilon \underset{\mu_{C}}{\otimes} \omega) = \omega \circ (\upsilon  \underset{\mu_{C}}{\otimes} \id).
 \end{align*}

Clearly, the product $(\upsilon,\omega) \mapsto \upsilon \underset{\mu_{C}}{\otimes} \omega$ is associative and  unital in the sense that if $\upsilon,\omega$ are as above and $\theta$ is a factorizable functional on an idempotent $D$-$E$-bimodule $O$, where $E$ is a non-degenerate, idempotent algebra with a fixed faithful functional, then
\begin{gather*}
  ((\upsilon \underset{\mu_{C}}{\otimes} \omega) \underset{\mu_{D}}{\otimes}\theta)((m\otimes n)\otimes o) = \omega(\upsilon_{C}(m) \cdot n \cdot {_{D}\theta}(o)) =  (\upsilon \underset{\mu_{C}}{\otimes} (\omega \underset{\mu_{D}}{\otimes}\theta))(m\otimes (n\otimes o)), \\
  (\mu_{B} \underset{\mu_{B}}{\otimes} \upsilon)(b\otimes m) = \upsilon(bm),  \quad
  (\upsilon \underset{\mu_{C}}{\otimes} \mu_{C})(m \otimes c) = \upsilon(mc)
\end{gather*}
for all $m\in M$, $n\in N$ and $o\in O$.

\begin{example}\label{example:factorization-frobenius}
  Assume that $B$ and $C$ are separable Frobenius algebras with separating linear functionals $\mu_{B}$ and $\mu_{C}$, respectively, as defined in \cite{daele:separability}; see  also Section 26 in \cite{brzezinski:book}.  Then similar arguments as in the proof of Lemma \ref{lemma:quasi-wmha} show that every  linear functional on an indempotent, non-degenerate $B$-$C$-bimodule is factorizable; see also \cite[Proposition 3.3]{boehm:comodules}. 
\end{example}

Of particular interest for us is the case where the fixed functionals $\mu_{B}$ and $\mu_{C}$ admit modular automorphisms. Regard $\dual{M}$ as an $M(C)$-$M(B)$-bimodule.
\begin{lemma} \label{lemma:factorizable-bimodule}
  Assume that $\mu_{B}$   and $\mu_{C}$ admit modular automorphisms $\sigma_{B}$ and $\sigma_{C}$, respectively. Then $M^{\sqcup}$ is an  $M(C)$-$M(B)$-sub-bimodule  of $\dual{M}$.
\end{lemma}
\begin{proof}
  Let $\omega \in M^{\sqcup}$ and  $x_{0}\in M(B)$. Then 
  $\omega(x_{0}my) = \mu_{C}(\omega_{C}(x_{0}m)y)$ and  \begin{align*}
    \omega(x_{0}xm) &= \omega(x_{0}x\bomega(m)) = \mu_{B}(x\bomega(m)\sigma_{B}(x_{0}))
  \end{align*}
 for all $x\in B$, $m\in M$, $y\in C$, showing that the functional $\omega x_{0}$ is factorizable and
 \begin{align*}
   _{B}(\omega x_{0})(m) &= \bomega(m)\sigma_{B}(x_{0}), & ( \omega  x_{0})_{C}(m) = \omega_{C} (mx_{0}).
 \end{align*} 
 A similar argument shows that $\omega y_{0}$ is factorizable for every $y_{0}\in M(C)$ and that
 \begin{align*}
   _{B}(y_{0 }\omega)(m) &= {_{B}\omega(my_{0})}, &
   ( y_{0} \omega)_{C}(m) &= \sigma_{C}^{-1}(y)\omega_{C}(m).  \qedhere
 \end{align*}
\end{proof}
\begin{remark} \label{remark:bphi-phib}
In the situation above, suppose that $y_{0}  \omega = \omega x_{0}$ for some $\omega \in M^{\sqcup}$, $x_{0}\in M(B)$ and $y_{0} \in M(C)$. Then the equations above imply that for all $m\in M$,
\begin{align*} 
  \bomega(my_{0}) &= \bomega(m)\sigma_{B}(x_{0}), & \omegac(x_{0}m) = \sigma_{C}^{-1}(y_{0}) \omega_{C}(m). 
\end{align*}
\end{remark}
Under the assumptions above, the assignment $M \mapsto M^{\sqcup}$ becomes a functor from $B$-$C$-bimodules to $C$-$B$-bimodules or, equivalently, to $B^{\op}$-$C^{\op}$-bimodules. Moreover, if we assume  in the situation of Lemma
\ref{lemma:factorizable-slice-tensor} that the functionals $\mu_{C},\mu_{D}$ and $\mu_{E}$ admit modular automorphisms, then the assignment $\upsilon \otimes \omega \mapsto \upsilon \underset{\mu_{C}}{\otimes} \omega$ factorizes to a morphism of $B^{\op}$-$D^{\op}$-bimodules
\begin{align*}
 {_{B^{\op}}(M^{\sqcup})_{C^{\op}}} \otimes 
 {_{C^{\op}}(N^{\sqcup})_{D^{\op}}}  \to 
 {_{B^{\op}}((M_{C} \otimes {_{C}N})^{\sqcup})_{D^{\op}}}.
\end{align*}
Finally, let us look at the factorizable functionals on $_{B}B_{B}$ in case where $\mu_{B}$ admits a modular automorphism.
\begin{example}  \label{example:factorizable-base}
  Regard $B$ as a $B$-bimodule and assume that $\mu_{B}$ admits a modular automorphism $\sigma_{B}$. Then for every multiplier $T \in M(B)$, the functional
  $\mu_{B} T \colon b\to \mu_{B}(Tb)$, is
 factorizable, and the assignment $T \mapsto \mu_{B} T$  defines a bijection between $M(B)$ and $B^{\sqcup}$.
This follows easily from  Lemma \ref{lemma:factorizable-bimodule} and \ref{lemma:lesssim} (2).
\end{example}

We us now return to the discussion of base weights and integrals.
Let
\begin{align*}
  \mathcal{A}=(A,B,C,S_{B},S_{C},\Delta_{B},\Delta_{C})
\end{align*}
 be a regular multiplier Hopf algebroid
with an antipodal base weight $(\mu_{B},\mu_{C})$. 
\begin{lemma} \label{lemma:factorizable}
   A functional $\omega \in
  \dA$   is  factorizable as a functional on the $B$-bimodule $_{B}A_{B}$ and on the $C$-bimodule $_{C}A_{C}$ if and only if
it is factorizable as a functional on the $B$-bimodule $_{B}A^{B}$ and on the $C$-bimodule $^{C}A_{C}$.
\end{lemma}
\begin{proof}
This is   straightforward.
\end{proof}
\begin{definition}
 We call a functional $\omega \in
  \dA$  \emph{factorizable (with respect to $(\mu_{B},\mu_{C})$)} if it satisfies the equivalent conditions in Lemma \ref{lemma:factorizable}. We denote by $A^{\sqcup} \subseteq \dA$ the subspace of all such factorizable functionals.
\end{definition}

\begin{remarks} \label{remark:factorizable}
  \begin{enumerate}
  \item The definition of quasi-invariance can now be reformulated as
    follows. The base weight $(\mu_{B},\mu_{C})$ is quasi-invariant
    with respect to a partial left integral $\cphic$ (or a partial
    right integral $\bpsib$) if and only if the functional
    $\phi:=\mu_{C} \circ \cphic$ (or $\psi:=\mu_{B} \circ \bpsib$) is
    factorizable.
  \item If $\mathcal{A}$ arises from a regular weak multiplier Hopf algebra $(A,\Delta)$ as in Example \ref{example:wmha}, then the base algebras $B$ and $C$ are Frobenius separable in the sense of \cite{daele:separability} and every functional on $A$ is factorizable; see Example  \ref{example:factorization-frobenius} and Lemma \ref{lemma:quasi-wmha}.
  \end{enumerate}
\end{remarks}
The following auxiliary result will be used later.  Denote by $\ast$ the involution on $\C$.
\begin{lemma} \label{lemma:base}
Let $(\mu_{B},\mu_{C})$ be an antipodal base weight for $\mathcal{A}$ and let $\omega \in A^{\sqcup}$.
 \begin{enumerate}
 \item Suppose that the base weight is modular.
 Let $x,x' \in M(B)$, $y,y' \in M(C)$ and $\omega':= xy \cdot
   \omega \cdot x'y'$. Then $\omega' \in A^{\sqcup}$ and
   \begin{align*}
     \bomega'(a) &=\bomega(y'axy)\sigma_{B}(x'), & \omegab'(a)
     &=
     (\sigma_{B})^{-1}(x)\omegab(x'y'ay), \\
     \comega'(a) &= \comega(x'axy)\sigma_{C}(y'), & \omegac'(a)
     &= (\sigma_{C})^{-1}(y) \omegac(x'y'ax).
   \end{align*}
 \item  Let $k\in \Z$ be odd and $\omega':=\omega  \circ    S^{k}$. Then $\omega' \in A^{\sqcup}$ and
  \begin{align*}
    _{B}(\omega') &= S^{-k} \circ \upsilonc \circ S^{k}, & (\omega')_{B}&= S^{-k} \circ \cupsilon  \circ  S^{k}, \\
    _{C}(\omega') &= S^{-k} \circ \upsilonb  \circ S^{k}, & (\omega')_{C} &= S^{-k} \circ \bupsilon  \circ S^{k}.
  \end{align*}
\item Assume that $\mathcal{A}$ is a multiplier Hopf $*$-algebroid and that the base weight is positive. Then
  $\omega^{*}:= \ast \circ \omega \circ \ast$ lies in $A^{\sqcup}$ and
  \begin{align*}
    _{B}(\omega^{*}) &= \ast \circ \upsilonb \circ \ast, & (\omega^{*})_{B} &=
\ast \circ     \bupsilon \circ \ast, & _{C}(\omega^{*}) &=  \ast \circ \upsilonc \circ \ast,
& (\omega^{*})_{C} &=  \ast \circ \cupsilon \circ \ast.
  \end{align*}
 \end{enumerate}
\end{lemma}
\begin{proof}
Assertion (1) follows from Lemma \ref{lemma:factorizable-bimodule},  (2) from functoriality of the assignment $M \mapsto M^{\sqcup}$,  and (3)  is  straightforward and left to the reader.
\end{proof}

\begin{corollary} \label{corollary:integrals-bimodule-antipode}
  Let $\mathcal{A}$ be a regular multiplier Hopf algebroid
  $\mathcal{A}$ with an antipodal base weight $(\mu_{B},\mu_{C})$.
  \begin{enumerate}
  \item If the base weight is modular, then all left (right) integrals form an $M(B)$-bimodule
    ($M(C)$-bimodule).
  \item  The maps $\phi \mapsto \phi\circ S^{\pm 1}$ are bijections
    between all left and all right integrals.
  \end{enumerate}
\end{corollary}
\begin{proof}
This follows easily from    Lemma \ref{lemma:base} and  Proposition \ref{proposition:partial-integrals-bimodule-antipode}.
\end{proof}

\section{Uniqueness of integrals relative to a base weight}

Let $\mathcal{A}$ be a regular multiplier Hopf algebroid with a modular base weight $(\mu_{B},\mu_{C})$ and a left integral $\phi$. Then for every  $x\in M(B)$, the rescaled functionals $a\mapsto \phi(ax)$ and $a\mapsto \phi(xa)$ are left integrals again by Corollary~\ref{corollary:integrals-bimodule-antipode}. We now show that under a certain non-degeneracy assumption on $\phi$ and local projectivity of $A$ as a module over $B$ and $C$, every left integral is of this form. Of course, a similar statement holds for right integrals.

\subsection{The unital case}
Let us first consider the  case of a    unital regular multiplier Hopf algebroid
  $\mathcal{A}$   with a base weight $(\mu_{B},\mu_{C})$ satisfying
  $\mu_{B}|_{\mathcal{O}} = \mu_{C}|_{\mathcal{O}}$,
where $\mathcal{O}=B\cap C$ denotes the orbit algebra. For example, this relation holds for antipodal base weights by Proposition \ref{proposition:orbit-antipode},  and for the base weights constructed in  Lemma \ref{lemma:quasi-proper}.
\begin{lemma}\label{lemma:unital-uniqueness-1}
Let $\phi$ be a left and $\psi$ a right integral for $(\mathcal{A},\mu_{B},\mu_{C})$. Then for all $a\in A$,
\begin{align} \label{eq:unital-fullness}
\psi(a\bphi(1)) = \psi(\phib(1)a) = \phi(a\cpsi(1))= \phi(\psic(1)a).
\end{align}
\end{lemma}
\begin{proof}
  By definition of $\bphi$ and of $\phib$, we can write $\phi(\bpsib(a))$ in the form
  \begin{align*}
\mu_{B}(\bpsib(a)\bphi(1)) = \psi(a\bphi(1)) \quad\text{or} \quad
\mu_{B}(\phib(1)\bpsib(a)) = \psi(\phib(1)a).
  \end{align*}
  Similarly, we can write $\psi(\cphic(a))$ in the form
  $\mu_{C}(\psic(1)\cphic(a))$ or $\mu_{C}(\cphic(a)\cpsi(a))$.
Now,  \eqref{eq:unital-fullness} follows because  by Propositon \ref{proposition:proper-partial-integrals},
  \begin{align}
    \phi \circ \bpsib &= \mu_{C} \circ \cphic \circ \bpsib = \mu_{B} \circ \bpsib \circ \cphic = \psi \circ \cphic. \qedhere
  \end{align}
\end{proof}
We can now conclude the following equivalence:
\begin{lemma} \label{lemma:unital-uniqueness-2}
Suppose that the base weight $(\mu_{B},\mu_{C})$ is antipodal. Then
for every functional $h$  on $A$, the following conditions are equivalent:
\begin{enumerate}
\item $h$ is a left integral and $h|_{B} = \mu_{B}$;
\item $h$ is a right integral and $h|_{C} = \mu_{C}$.
\end{enumerate}
If these conditions hold, then $h=h\circ S$, and with $\tau:=\mu_{B}|_{\mathcal{O}} = \mu_{C}|_{\mathcal{O}}$, we have
\begin{align*}
  _{C}h_{C}(1) &= 1, & _{B}h_{B}(1)&=1, & \mu_{B} &= \tau\circ {_{C}h_{C}}|_{B}, & \mu_{C} &= \tau\circ {_{B}h_{B}}|_{C}.
\end{align*}
\end{lemma}
\begin{proof}
  Suppose that  (1) holds. Since $\mu_{B}$ is faithful and $\mu_{B}(x) = h(x)=\mu_{B}(h_{B}(1)x)=\mu_{B}(x{_{B}h}(1))$ for all $x\in B$, we must have $h_{B}(1)=1={_{B}h}(1)$. By Corollary \ref{corollary:integrals-bimodule-antipode}, $\psi:=h\circ S$ is a right integral and by Lemma 
\ref{lemma:base}, $\cpsi(1) = 1=\psic(1)$. In particular, $\psi|_{C}=\mu_{C}$. Now, $\psi=h$ by Lemma \ref{lemma:unital-uniqueness-2} and (2) follows. The reverse implication follows similarly.
\end{proof}
\begin{definition}
  Let $\mathcal{A}$ be a unital regular multiplier Hopf algebroid with an antipodal base weight $(\mu_{B},\mu_{C})$. We call a functional $h$ on $A$ a \emph{Haar integral} for $(\mathcal{A},\mu_{B},\mu_{C})$ if it satisfies the equivalent conditions in Lemma \ref{lemma:unital-uniqueness-2}.
\end{definition}
\begin{example}
Suppose that $\cphic$ is a partial left integral on  a unital regular multiplier Hopf algebroid $\mathcal{A}$ such that $\cphic(1)=1$, $\cphic\circ S^{2} = S^{2} \circ \cphic$  and $\cphic|_{B}$ is faithful. 
Let $\bpsib=S\circ \cphic \circ S^{-1}$ and choose a faithful functional $\tau$ on the orbit algebra $\mathcal{O}$. Then $\mu_{B} := \tau\circ \cphic|_{B}$ and $\mu_{C} := \tau\circ \bpsib|_{C}$ form an antipodal base weight by Lemma \ref{lemma:quasi-proper} and  $\phi=\mu_{C}\circ \cphic$ is a Haar integral. 
\end{example}

The preceding results immediately imply the following uniqueness result:
\begin{corollary}
  Let $\mathcal{A}$ be a unital regular multiplier Hopf algebroid with an antipodal base weight $(\mu_{B},\mu_{C})$. If a Haar integral $h$ exists, then it is unique, and then 
  \begin{align*}
    h(a\bphi(1)) =     \phi(a) = h(\phib(1)a) \quad \text{and} \quad h(a\cpsi(1)) = h(a) = h(\psic(1)a)
  \end{align*}
for every left integral $\phi$, every right integral $\psi$ and all $a\in A$.
\end{corollary}
\begin{proof}
  Combine Lemmas \ref{lemma:unital-uniqueness-1} and  \ref{lemma:unital-uniqueness-2}.
\end{proof}
\subsection{Uniqueness of integrals}

\label{subsection:uniqueness}
Let us now consider the general case. The first step towards the proof of uniqueness is the following result.
\begin{lemma} \label{lemma:uniqueness-phi-psi} Let $\phi$
  be a left and $\psi$ a right integral on a regular   multiplier Hopf algebroid $\mathcal{A}$ with antipodal base weight. Then 
  \begin{align*}
    (A \bphi(A)) \cdot \psi = (A \cpsi(A)) \cdot \phi \quad \text{and}\quad
\psi \cdot
  (\phib(A)A) = \phi \cdot (\psic(A)A).
  \end{align*}
\end{lemma}
\begin{proof}
  We only prove the first equation.  Let $a\in A$ and $b\otimes c\in
  \AlA$ and write
  \begin{align*}
    b\otimes c=\sum_{i} \Delta_{B}(d_{i})(1 \otimes e_{i}) = \sum_{j}
    \Delta_{B}(f_{j}) (g_{j} \otimes 1)
  \end{align*}
  with $d_{i},e_{i},f_{j},g_{j} \in A$.  Then  $(\psi \underset{\mu_B}{\otimes} \phi)(\Delta_{B}(a)(b\otimes c))$ is equal to
  \begin{align*}
    \sum_{i} (\psi \underset{\mu_B}{\otimes} \phi)(\Delta_{B}(ad_{i})(1 \otimes e_{i})))
    &= \sum_{i} \phi(\bpsib(ad_{i}) e_{i}) =
\sum_{i}    \psi(ad_{i}\bphi(e_{i}))
  \end{align*}
  because $\psi$ is a right integral, and to
  \begin{align*}
    \sum_{j} (\psi \underset{\mu_B}{\otimes} \phi)(\Delta_{B}(af_{j})(g_{j}\otimes 1))) &=
    \sum_{j} \psi(\cphic(af_{j})g_{j}) =
\sum_{j}    \phi(af_{j}\cpsi(g_{j}))
  \end{align*}
because $\phi$ is a left integral. Since the maps $\Tl$ and $\Tr$ are bijective, we can conclude that $(A \bphi(A)) \cdot \psi = (A
  \cpsi(A)) \cdot \phi$.
\end{proof}
The preceding result suggests to consider the following non-degeneracy condition.
\begin{definition} \label{definition:full}
  Let $\mathcal{A}$ be a regular   multiplier Hopf algebroid  with an antipodal base weight $(\mu_{B},\mu_{C})$.
We call  a left integral $\phi$ (right integral $\psi$) for 
$(\mathcal{A},\mu_{B},\mu_{C})$
  \emph{full} if $\bphi$ and $\phib$ ($\cpsi$ and $\psic$,
  respectively) are surjective.  
\end{definition}
\begin{remark}\label{remark:full}
  Note that $\omega$ is a full left or right integral, then the right
  or left integrals $\omega \circ S$ and $\omega \circ S^{-1}$ are
  full again by Lemma \ref{lemma:base} (2).
\end{remark}
We obtain the following corollaries, which involve the relation \eqref{eq:lesssim} on $\dA$:
\begin{proposition} \label{proposition:uniqueness-full}
  Let $\omega$ and $\omega'$ be   left or right integrals  on a regular multiplier Hopf algebroid $\mathcal{A}$ with a fixed antipodal base weight.
   If $\omega$ is full, then $\omega' \lesssim \omega$.
 \end{proposition}
 \begin{proof}
   This follows immediately from the preceding result.
 \end{proof}
\begin{proposition} \label{proposition:full}
For every full left integral  $\phi$ and every full right integral
$\psi$ on a regular multiplier Hopf algebroid $\mathcal{A}$ with a fixed antipodal base weight, the partial integrals $\cphic$ and $\bpsib$ are surjective.
\end{proposition}
\begin{proof}
Suppose that $\phi$ is a full left integral. Then $\psi':=\phi
   \circ S$ is a full right integral and  $A \phi = A \psi'$ by Proposition
   \ref{proposition:uniqueness-full}. Let $a,b\in A$ and choose $c\in A$ such that $c\cdot \phi=b\cdot \psi'$. Then $_{C}\psi'(ab)=\cphic(ac)$. Therefore, $C={_{C}\psi'(A)}=\cphic(A)$.  A similar argument shows that
   $\bpsib(A)=B$ for every right integral $\psi$.
\end{proof}

To show that  integrals are unique up to scaling, we need some further preparations.
\begin{lemma}\label{lemma:integrals-modular-base}
  Let $\phi$ be a left and $\psi$ a right integral for a  regular multiplier Hopf algebroid   $\mathcal{A}$ with a modular base weight. Then for all $x\in B$ and $y\in C$,
  \begin{align*}
    \phi  \cdot y&=  S^{2}(y)\cdot \phi  &&\text{and} & x\cdot \psi &=\psi\cdot S^{2}(x).
  \end{align*}
\end{lemma}
\begin{proof}
  We only prove the first assertion. If $(\mu_{B},\mu_{C})$ denotes the base weight and $a\in A$, $y\in C$, then
$ \phi(ya) = \mu_{C}(y\cphic(a)) = \mu_{C}(\cphic(a)S^{2}(y)) = \phi(aS^{2}(y))$.
\end{proof}

To prove  the desired uniqueness result for integrals, we need a further assumption. 

Let $M$ be a right module over an algebra $D$. Recall that $M$ is called \emph{firm}  if the multiplication map induces an isomorphism  $M_{D} \otimes {_{D}D} \to M$, and  \emph{locally projective} \cite{zimmermann-huisgen} if for every finite number of elements $m_{1},\ldots,m_{k} \in M$, there exist finitely many $\upsilon_{i} \in\Hom(M_{D},D_{D})$ and $e_{i} \in \Hom(D_{D},M_{D})$ such that $m_{j} = \sum_{i} e_{i}(\upsilon_{i}(m_{j}))$ for all $j=1,\ldots,k$. The corresponding definition for left modules is obvious.  The algebra $D$ is firm if it is so as a right (or, equivalently, as a left) module over itself.
The following results may be well-known, but we could not find a reference and leave the straightforward proof to the reader.
\begin{lemma} \label{lemma:projective}
  Let $D$ be a firm algebra.
  \begin{enumerate}
  \item  Every idempotent, locally projective right $D$-module is firm.
  \item Let $M_{D}$ be a locally projective right $D$-module, $_{D}N$  a firm left $D$-module,
 and suppose that the set of maps $P \subseteq
    \Hom({_{D}N},{_{D}D})$ separates the points of $N$. Then the slice
    maps
    \begin{align*}
      \id \otimes \omega \colon M_{D} \otimes {_{D}N} \to M, \quad
      m\otimes n \mapsto m\omega(n),
    \end{align*}
    where $\omega \in P$, separate the points of $M_{D} \otimes
    {_{D}N}$.
  \end{enumerate}
\end{lemma}
\begin{definition} \label{definition:projective}
  We call a regular multiplier Hopf algebroid $(A,B,C,S_{B},S_{C},\Delta_{B},\Delta_{C})$ \emph{locally projective} if 
  the algebras $B$ and $C$ are firm and the modules $\bA,\Ab,\cA,\Ac$ are locally projective.
\end{definition}
\begin{theorem} \label{theorem:integrals-uniqueness} 
Let $\mathcal{A}$ be a locally projective regular multiplier Hopf algebroid with modular base weight $(\mu_{B},\mu_{C})$.
Then for every full   and faithful left  integral $\phi$ for $(\mathcal{A},\mu_{B},\mu_{C})$,
    \begin{gather*}
      M(B) \cdot \phi = \{ \text{left integrals } \phi' \text{ for }
(\mathcal{A},\mu_{B},\mu_{C}) \} = \phi \cdot M(B),
\end{gather*}
 and for every full and faithful right   integral $\psi$ for $(\mathcal{A},\mu_{B},\mu_{C})$,
\begin{gather*}
      M(C) \cdot \psi = \{ \text{right integrals } \psi' \text{ for }
(\mathcal{A},\mu_{B},\mu_{C}) \} = \psi \cdot M(C).
    \end{gather*}
\end{theorem}
\begin{proof}
  We only prove the assertion concerning a full and faithful left integral $\phi$.

  Every element of $M(B) \cdot \phi$ and of $\phi \cdot M(B)$ is a  left integral by Corollary  \ref{corollary:integrals-bimodule-antipode}.

  Conversely, assume that $\phi'$ is a left integral on  $\mathcal{A}$. By Proposition \ref{proposition:uniqueness-full} and Lemma \ref{lemma:lesssim},  there exist  a unique  multipliers $\alpha \in L(A)$ and  $\beta \in R(A)$ such that $\alpha \cdot \phi = \phi' = \phi \cdot
  \beta$.  The multiplier $\beta$ commutes with $C$ because $\phi$ is faithful and by Lemma \ref{lemma:integrals-modular-base},
  \begin{align*}
\phi\cdot \beta y = \phi' \cdot y = S^{2}(y) \cdot \phi' = S^{2}(y) \cdot \phi \cdot \beta  = \phi \cdot y\beta
\end{align*}
for all $y\in C$, and similarly, $\alpha$ commutes with $C$.

  Choose a full left integral $\psi$, for example, $\phi \circ S$.  We will show that for all $a,b\in A$,
  \begin{align} \label{eq:integrals-uniqueness-2}
a \bpsib(b)\alpha &= a\bpsib(b\alpha), &    \beta \bpsib(b) a &= \bpsib(\beta b)a.
  \end{align}
  These equations imply $B\alpha \subseteq B$ and $\beta B \subseteq   B$. We then conclude, similarly as in the proof of Lemma \ref{lemma:lesssim} (2), that
  \begin{align*}
    \mu_{B}(x\alpha S^{-2}(\phib(a))) =
    \mu_{B}(\phib(a)x\alpha) &= \phi(ax\alpha) \\ &= \phi(\beta ax) =
    \mu_{B}(\phib(\beta a)x) = \mu_{B}(xS^{-2}(\phi_{B}(\beta a)))
  \end{align*}
 for all $a\in A$, $x\in B$. Since $\mu_{B}$ is faithful and $\phi$ is  full, this relation implies $\alpha B \subseteq B$, that is, $\alpha
 \subseteq M(B)$. A similar argument shows that also $\beta \in M(B)$.

 Therefore, we only need to prove
 \eqref{eq:integrals-uniqueness-2}. We focus on the second equation; the first equation follows  similarly.  Let $a,b\in A$.  Since $\cphic' = \cphic \cdot \beta$ and 
 $\cphic$ are partial left inegrals,
  \begin{align*}
    (\id \otimes S_{B}^{-1} \circ \cphic \circ \beta)(\Delta_{B}(b)(a \otimes 1)) & =
    \cphic(\beta b)a = (\id \otimes S_{B}^{-1} \circ    \cphic)(\Delta_{C}(\beta b)(a \otimes 1)).
  \end{align*}
  Since $\Tl$ is surjective, we can conclude
  \begin{align*}
    (\id \otimes S_{B}^{-1} \circ \cphic \circ \beta)(\Delta_{B}(b)(a\otimes cd)) & =
    (\id \otimes S_{B}^{-1} \circ \cphic)(\Delta_{B}(\beta b)(a \otimes cd))
  \end{align*}
  for all $a,b,c,d \in A$. Since $\phi$ is faithful and $A$ non-degenerate, maps of the form
  $d \cdot \cphic$ separate the points of $A$. By assumption and Lemma \ref{lemma:projective}, slice maps of the form $\id \otimes S_{B}^{-1} \circ  (d\cdot \cphic)$ separate the points of
  $\AlA$. Consequently,
  \begin{align*}
    (\iota \otimes \beta)(\Delta_{B}(b)(1 \otimes c)) =
    \Delta_{B}(\beta b) (1\otimes c) \quad \text{ for all }a,b,d\in A.
  \end{align*}
Proposition \ref{proposition:ergodic} now implies $\beta \in M(B)$.
\end{proof}
For every full and faithful left integral, we thus obtain bijections between $M(B)$ and the space of left integrals, and between invertible multipliers of $B$ and full and faithful
left integrals. Of course, a similar remark applies right  integrals.

 \section{Counital base weights and measured multiplier Hopf algebroids}

We now introduce the missing last 
 assumption on our base weights, which is existence of a factorizable counit functional. This condition  appeared in a related context in \cite{daele:relation} and   
implies a much closer relation between the left and the right comultiplication of a regular multiplier Hopf algebroid than the mixed co-associativity condition alone. After a discussion of this condition, we  finally define the notion of a measured regular multiplier Hopf algebroid and look at  examples.

\subsection{Counital base weights} \label{subsection:counital}
  Let $\mathcal{A}=(A,B,C,S_{B},S_{C},\Delta_{B},\Delta_{C})$ be a regular multiplier Hopf algebroid with antipode $S$, left counit $\beps$ and right counit $\epsc$.
\begin{definition}
 A base weight $(\mu_{B},\mu_{C})$  for $\mathcal{A}$ is \emph{counital} if it is antipodal and satisfies
\begin{align}
  \label{eq:base-weight-counital}
  \mu_{B} \circ \beps = \mu_{C} \circ \epsc. 
\end{align}
 In this case, we call this composition the associated \emph{counit functional} and denote it by $\varepsilon$.
\end{definition}
We shall give examples in the next subsection and first discuss the condition above.
\begin{remarks}  \label{remarks:counital}
Let $(\mu_{B},\mu_{C})$ be an antipodal base weight.
  \begin{enumerate}
  \item Relation
 \eqref{eq:antipode-counits}  implies
  \begin{align} \label{eq:counit-antipode}
    (\mu_{B} \circ \beps) \circ S  &= \mu_{C} \circ \epsc, & (\mu_{C} \circ \epsc) \circ S &= \mu_{B} \circ \beps,
  \end{align}
so that \eqref{eq:base-weight-counital} is equivalent to invariance of either sides under the antipode. Thus, the counit functional associated to a counital base weight is invariant under the antipode.
\item Conversely, suppose that $\mathcal{A}$ is unital and \eqref{eq:base-weight-counital} holds. Then $(\mu_{B},\mu_{C})$ is antipodal because then $\beps(1_{A})=1_{B}$, $\epsc(1_{A})=1_{C}$  and hence
  \begin{align*}
  \mu_{C}(S_{C}^{-1}(x)) = \mu_{C}(\epsc(x)) = \mu_{B}(\beps(x)) =\mu_{B}(x)
\end{align*}
for all $x\in B$ and similarly $\mu_{B}(S_{B}^{-1}(y))=\mu_{C}(y)$ for all $y\in C$.
\item Suppose that the base weight $(\mu_{B},\mu_{C})$ is counital. Then the associated counit functional $\varepsilon$ is factorizable and the associated module maps are
  \begin{align} \label{eq:counit-functional-factorizations}
    \beps, \quad \epsb = S_{C} \circ \epsc, \quad \ceps = S_{B}\circ \beps, \quad \epsc.
  \end{align}
 Indeed, 
 \eqref{eq:left-counit-bimodule} implies 
$\eps(xa) = \mu_{B}(x\beps(a))$ and
  \begin{align*}
    \eps(ax) &=\mu_{C}(\epsc(ax))=\mu_{C}(S_{C}^{-1}(x)\epsc(a)) =
    \mu_{B}(S_{C}(\epsc(a))x)
  \end{align*}
  for all $a\in A$ and $x\in B$, and \eqref{eq:rt-counit-bimodule} implies $\eps(ay)=\mu_{C}(\epsc(a)y)$ and
  $\eps(ya)=\mu_{C}(yS_{B}(\epsb(a)))$ for all $a\in A$ and $y\in C$.  

 We can therefore form the relative tensor products $\eps \underset{\mu_{B}}{\otimes} \eps$ and $\eps \underset{\mu_{C}}{\otimes} \eps$, which are functionals on $\bAsA$ and  $\cAsA$, respectively, and
 \eqref{eq:counit-antipode} and \eqref{eq:counits-multiplicative} imply 
\begin{align*}
  (\eps \underset{\mu_{B}}{\otimes} \eps)(a\otimes b) =\varepsilon(ab) = (\eps \underset{\mu_{C}}{\otimes} \eps)(a\otimes b)
\end{align*}
for all $a,b\in A$.
\end{enumerate}
\end{remarks}
In the involutive case, \eqref{eq:base-weight-counital} is equivalent to several natural conditions.
\begin{remarks}
 Suppose that $\mathcal{A}$ is a multiplier Hopf $*$-algebroid and that $(\mu_{B},\mu_{C})$ is positive and antipodal.
\begin{enumerate}
\item Equation  \eqref{eq:base-weight-counital} is equivalent to self-adjointness of either sides  because by
  \eqref{eq:counit-antipode-involution},
  \begin{align*}
    \mu_{B} \circ \beps \circ * = \mu_{B} \circ * \circ S_{C} \circ
    \epsc = * \circ \mu_{C}\circ \epsc.
  \end{align*}
\item Equip $B$  with the inner product $\langle x|x'\rangle:=\mu_{B}(x^{*}x')$ and  consider the map
  \begin{align*}
    \pi_{B} \colon A \to \End(B), \quad
\pi_{B}(a)x &:= \beps(ax).
  \end{align*}
This is a homomorphism because of \eqref{eq:counits-multiplicative}.   Now,  \eqref{eq:base-weight-counital} holds if and only if
\begin{align*}
  \langle x|\pi_{B}(a)x'\rangle &=   \langle \pi_{B}(a^{*})x|x'\rangle
\end{align*}
for all $x,x'\in B$ and $a\in A$ because the inner products above are given by
 \begin{align*}
  \mu_{B}(x^{*}\beps(ax)) = (\mu_{B} \circ \beps)(x^{*}ax')
\end{align*}
and
\begin{align*}
 \mu_{B}(\beps(a^{*}x)^{*}x') =
 \mu_B(x'{}^{*}\beps(a^{*}x))^{*} = (\ast \circ \mu_{B}\circ \beps \circ \ast)(x^{*}ax'),
\end{align*}
respectively. Similarly, condition \eqref{eq:base-weight-counital}  can be reformulated in terms of the map $\pi_{C} \colon A^{\op} \to \End(C)$ given by $\pi_{C}(a^{\op})y := \epsc(ya)$  and the inner product on $C$ induced by $\mu_{C}$.
  \end{enumerate}
\end{remarks}

Recall that without loss of generality, one can assume the left and the right counit of  a regular multiplier Hopf algebroid to be surjective; see \cite[Lemma 3.6 and 3.7]{timmermann:regular}.
\begin{proposition} \label{proposition:counit-kms} 
Suppose that $\mathcal{A}$ is a regular multiplier Hopf algebroid with surjective left and surjective right counit. Then every counital base weight for $\mathcal{A}$ is modular.
\end{proposition}
\begin{proof}
  We only show that $S_{B}^{-1}S_{C}^{-1}$ is a modular automorphism of $\mu_{B}$. Let
  $x \in B$ and $a\in A$. Then \eqref{eq:counits-multiplicative} implies
  \begin{align*}
    \mu_{B}(x\beps(a)) = \mu_{B}(\beps(xa)) &= \mu_{C}(\epsc(xa)) \\ &=
    \mu_{C}(\epsc(S_{C}^{-1}(x)a)) \\ &= \mu_{B}(\beps(S_{C}^{-1}(x)a)) =
    \mu_{B}(\beps(a)S_{B}^{-1}(S_{C}^{-1}(x))).  \qedhere
  \end{align*}
\end{proof}
The preceding result fits with the theory of measured
    quantum groupoids, where the square of the antipode generates the
    scaling group and the latter restricts to the modular automorphism
    groups on the base algebras, that is, in the notation of
    \cite{enock:action}, $S^{2} = \tau_{i}$ and $\tau_{t} \circ \alpha
    = \alpha \circ \sigma^{\nu}_{t}$, $\tau_{t} \circ \beta = \beta
    \circ \sigma^{\nu}$.

\subsection{Measured regular multiplier Hopf algebroids}
\label{subsection:measured}
We have now gathered all ingredients and assumptions to define the main objects of interest of this article.
\begin{definition}
  A \emph{measured regular multiplier Hopf algebroid} consists of a regular multiplier Hopf algebroid $\mathcal{A}$, a  base weight $(\mu_{B},\mu_{C})$, a faithful partial right integral $\bpsib$ and a faithful partial left integral $\cphic$ such that 
  \begin{enumerate}
  \item the base weight is counital, and quasi-invariant with respect to $\bpsib$ and $\cphic$,
  \item the right integral $\psi:=\mu_{B} \circ \bpsib$ and the left integral $\phi:=\mu_{C} \circ \cphic$ are full.
  \end{enumerate}
We call it a \emph{measured multiplier Hopf $*$-algebroid} 
if additionally $\mathcal{A}$ is a multiplier Hopf $*$-algebroid and the functionals $\mu_{B},\mu_{C},\phi,\psi$ are positive.
\end{definition}
  Faithfulness of $\psi$ and  $\phi$, and hence also of $\bpsib$ and $\cphic$,  follows from the other assumptions if  $A$ is locally projective as a module over $B$ and $C$, as we shall see  in Theorem  \ref{theorem:integrals-faithful}.

Let us  consider our list of examples.
\begin{example} \label{example:base-weights-wmha}
Consider the  regular multiplier Hopf algebroid  associated to a regular weak multiplier Hopf  algebra $(A,\Delta)$ as  in Example \ref{example:wmha}.  Then the  base weight  in \eqref{eq:weak-base-weight} is counital because it is antipodal and the left and the right counit of $\mathcal{A}$ are given by $\beps =S^{-1} \circ \varepsilon_{t}$ and $\epsc = S^{-1} \circ \varepsilon_{s}$, respectively.  The associated counit functional  is just the counit of $(A,\Delta)$.
\end{example}

\begin{example} \label{example:measured-groupoid-functions}
Let $G$ be a locally compact, \'etale Hausdorff groupoid with a Radon measure  $\mu$
 on the space of units $G^{0}$ that has full support and is continuously quasi-invariant as  in Example \ref{example:quasi-groupoid-functions}. 
Then the multiplier Hopf $*$-algebroid $\mathcal{A}$  of functions on $G$ defined in Example \ref{example:hopf-groupoid-functions} together with the base weight $(\mu_{B},\mu_{C})$ and the partial integrals $\cphic$ and $\bpsib$  defined in \eqref{eq:integrals-groupoid-functions} forms a measured  multiplier Hopf $*$-algebroid. Indeed, the base weight $(\mu_{B},\mu_{C})$ is  counital   and its the associated counit functional is given by
\begin{align*}
  \varepsilon(f) = \int_{G^{0}} f|_{G^{0}} \intd \mu,
\end{align*}
and the integrals $\phi$ and $\psi$ are given easily seen to be full and faithful.
\end{example}

\begin{example} \label{example:measured-groupoid-algebra}
Let $G$ and $\mu$ be as above and consider the multiplier Hopf $*$-algebroid $\hat{\mathcal{A}}$
associated to the convolution algebra $\hat A=C_{c}(G)$ as in Example \ref{example:hopf-groupoid-algebra}. Then the base weight $(\hat\mu_{\hat B},\hat\mu_{\hat C})$ associated to $\mu$ as in Example \ref{example:quasi-groupoid-algebra}
is counital if and only if for every $f\in C_{c}(G)$, the integrals
\begin{align*}
  \hat\mu(_{\hat B}\hat\varepsilon(f))&=  \int_{G^{0}} \sum_{t(\gamma)=u} f(\gamma) \intd\mu(u) &&\text{and} &
  \hat\mu(\hat\varepsilon_{\hat C}(f)) &=  \int_{G^{0}} \sum_{s(\gamma) = u} f(\gamma)\intd\mu(u)
\end{align*}
coincide, that is, if and only if
 $\mu$ is  \emph{invariant} in  sense of \cite[Definition 3.12]{renault}.  If this condition holds, then together with the partial integral 
 \begin{align*}
   _{\hat B}\hat \psi_{\hat B}={_{\hat C}\hat \phi_{\hat C}}\colon C_{c}(G) \to C_{c}(G^{0}), \quad f\mapsto  f|_{G^{0}},
 \end{align*}
we obtain a measured multiplier Hopf $*$-algebroid again; see also Example \ref{example:integrals-groupoid-algebra} and \ref{example:quasi-groupoid-algebra}.

Invariance of the measure $\mu$ on $G^{0}$  is a rather strong condition, and in subsection \ref{subsection:modification-convolution}, we shall see that after a slight modification of the multiplier Hopf $*$-algebroid, it suffices to assume $\mu$  to be   continuously quasi-invariant in the sense defined in Example \ref{example:quasi-groupoid-functions}.
\end{example}

\begin{example} \label{example:measured-cb}
Consider  the tensor product $A=C\otimes B$ discussed in Example \ref{example:hopf-cb}.
In this case, an antipodal base weight $(\mu_{B},\mu_{C})$ is counital if and only if the two expressions
  \begin{align*}
    \mu_{B}(\beps(y\otimes x)) &=\mu_{B}(xS_{B}^{-1}(y)) &&\text{and} &  \mu_{C}(\epsc(y\otimes x)) &= \mu_{C}(S_{C}^{-1}(x)y)
  \end{align*}
  coincide for all $x\in B$, $y\in C$, and therefore if and only if it is modular. Suppose that this condition holds.  The maps
  \begin{align*}
     \cphic &:= \id \otimes \mu_{B} \colon A\to C, &
     \bpsib &:=\mu_{C} \otimes \id \colon A \to B
  \end{align*}
are  left- and right-invariant, respectively, by Example \ref{example:integrals-cb}, and the resulting integral   $\phi=\psi=\mu_{C} \circ \mu_{B}$ is full and faithful.  We therefore  obtain a measured regular multiplier Hopf algebroid again.
\end{example}

\begin{example} \label{example:measured-ch} 
    Consider the symmetric crossed product $A=C\# H$ introduced in Example \ref{example:hopf-ch},
 and  let $\mu$  be a faithful functional on the algebra $C$. Then the base weight $(\mu,\mu)$ is counital if and only if the expressions
  \begin{align*}
 \mu(\beps(hy)) &= \mu(\beps((h_{(1)} \actright y)h_{(2)})) = \mu(h\actright y) & &\text{and} & \mu(\epsc(hy)) &= \mu(y)\varepsilon_{H}(h)
  \end{align*}
 coincide for all $y\in C$ and $h\in H$, that is, if and only if $\mu$ is invariant under the action of $H$. Suppose that this condition holds, and that $\phi_{H}$ is a left  and $\psi_{H}$  is a right integral on $(H,\Delta_{H})$. With the partial integrals $\cphic$ and $\bpsib$ defined as in Example \ref{example:integrals-ch}, we obtain a measured regular multiplier Hopf algebroid; see also Example \ref{example:quasi-ch}.

  Again, a weaker  quasi-invariance condition on $\mu$ turns out to be sufficient after a suitable modification of the multiplier Hopf algebroid, see subsection \ref{subsection:modification-crossed}.
\end{example}
\begin{example} \label{example:measured-chb}
Consider the two-sided crossed product $A=C\# H \# B$ discussed in Example \ref{example:hopf-chb}. In this case, an antipodal base weight $(\mu_{B},\mu_{C})$ is counital if and only if the two expressions
\begin{align} \label{eq:chb-base-weights}
    \mu_{B}(\beps(xh y)) &= \mu_{B}(xS^{-1}(h\actright y)) && \text{and} &
  \mu_{C}(\epsc(xhy)) &=  \mu_{C}(S^{-1}(x\actleft h)y) 
\end{align}
coincide for all $x\in B$, $h\in H$ and $y\in C$. Suppose that  $(\mu_{B},\mu_{C})$ is modular. Then
\begin{align*}
   \mu_{B}(xS^{-1}(h\actright y)) = \mu_{C}((h\actright y)S(x)) = \mu_{C}(S^{-1}(x)(h\actright y)),
\end{align*}
and then equality of the expressions in \eqref{eq:chb-base-weights}  is equivalent to invariance of $\mu_{C}$ and $\mu_{B}$ under $H$ in the sense of  \eqref{eq:base-invariant}. Suppose that also this condition holds, and 
 that $\phi_{H}$ is a left and $\psi_{H}$ is a right integral on $(H,\Delta_{H})$. Together with  the partial integrals
 \begin{align*}
  \cphic(yhx) &:= y\phi_{H}(h) \mu_{B}(x) &&\text{and} &
  \bpsib(yhx) &:=\mu_{C}(y)\phi_{H}(h)x
 \end{align*}
defined in Example \ref{example:integrals-chb}, we obtain a measured regular multiplier Hopf algebroid; see also Example \ref{example:quasi-chb}.

 In subsection \ref{subsection:modification-twosided}, we shall treat the case where $\mu_{C}$ and $\mu_{C}$ are only  quasi-invariant with respect to the action of $H$.
 \end{example}

\section{The key results on integration}

With all the assumptions in place, we now establish the key results on integration listed in the introduction --- existence of a modular automorphism (subsection \ref{subsection:modular-automorphism}), existence of a modular element (subsection \ref{subsection:modular-element}), and faithfulness of the integrals (subsection \ref{subsection:faithful}). Along the way, we use and study left and right 
convolution operators naturally associated to factorizable functionals (subsection \ref{subsection:convolution}). Here, the counitality assumption on the base weight comes into play, and ensures that the convolutions formed with respect to the left and with respect to the right comultiplication coincide; see Corollary \ref{corollary:convolution}.

\subsection{Convolution operators and the modular automorphism}

 \label{subsection:modular-automorphism}

Let $\mathcal{A}$ be a regular multiplier Hopf algebroid with a counital base weight $(\mu_{B},\mu_{C})$. 
We  show that   integrals for $(\mathcal{A},\mu_{B},\mu_{C})$ which are full and faithful automatically admit modular automorphisms.

  As a tool, we  use the following left   convolution operators   associated to elements $ \bupsilon \in \dbA = \Hom(\bA,\bB)$ and $\upsilon'_{B} \in \dAb=\Hom(\Ab,\Bb)$,
\begin{align*}
  \begin{aligned}
    \lambda(\bupsilon) &\colon A \to L(A), & \lambda(\bupsilon)(a)b &=
    (\bupsilon \oo \id)(\Delta_{B}(a)(1
    \oo     b)), \\
    \lambda(\upsilonb') &\colon A \to R(A), & b\lambda(\upsilonb')(a)
    &= (S_{C}^{-1}\circ \upsilonb' \oo \id)((1 \oo b)\Delta_{C}(a)),
  \end{aligned}
\end{align*}
and the following right   convolution operators   associated to elements $\comega \in \dcA = \Hom(\cA,\cC)$ and $\omega'_{C} \in \dAc=\Hom(\Ac,\Cc)$,
associated to maps
\begin{align*}
  \begin{aligned}
     \rho(\comega) &\colon A \to L(A), &
    \rho(\comega)(a)b &= (\id \oo S_{B}^{-1}\circ
    \comega)(\Delta_{B}(a)(b \oo 1)), \\
    \rho(\omegac') &\colon A \to R(A),   &
    b\rho(\omegac')(a) &= (\id \oo \omegac')((b \oo
    1)\Delta_{C}(a)).
  \end{aligned}
 \end{align*}

 The notation above can be  ambiguous for elements $\bupsilon_{B}\in \Hom(\bAb,\bBb)
$ and $\comega_{C} \in \Hom(\cAc,\cCc)$, and we shall always write  $\rho(\bupsilon)$, $\rho(\upsilonb)$, $\lambda(\comega)$ or  $\lambda(\omegac)$ to indicate which convolution operator we mean. This ambiguity will be resolved in Lemma \ref{lemma:convolution} (4) below.

 Let us collect a few easy observations.   

For all  $\bupsilon,\upsilonb',\comega,\omegac'$ as above and $a,c\in A$, the multipliers
\begin{align*}
   \lambda(c \cdot \bupsilon)(a), \quad
   \lambda( \upsilonb'\cdot c)(a), \quad
  \rho(c \cdot \comega)(a), \quad
  \rho(\omegac'\cdot c)(a)
\end{align*}
lie in $A$;  for example,
$\lambda(c \cdot \bupsilon)(a)  = (\bupsilon \oo
  \id)(\Delta_{B}(a)(c \oo 1))$.

By Proposition \ref{proposition:partial-integrals},  a  map $\bpsib\in \dbAb$ is a partial right integral if and only if the following equivalent conditions hold,
\begin{align} \label{eq:convolution-right-invariance}
    \lambda(\bpsi) &= \bpsi, &
    \lambda(\psib) &= \psib, &
\lambda(a\cdot \bpsi)(b) &= S(\lambda(\psib \cdot b)(a))  \text{ for     all } a,b\in A.
\end{align}
 Similarly, $\cphic \in \dcAc$ is a partial left integral if and only if the following equivalent conditions hold:
\begin{align}
  \label{eq:convolution-left-invariance}
  \rho(\cphi) &= \cphi, & \rho(\phic) &= \phic, &
    \rho(\phic \cdot a)(b) &=    S(\rho(b\cdot \cphi)(a))     \text{ for     all } a,b\in A.
\end{align}
 Finally, \eqref{eq:left-counit} and \eqref{eq:right-counit}  imply\begin{align} \label{eq:convolution-counit}
  \lambda(\beps) &= \rho(\ceps) =  \lambda(\epsb) =
\rho(\epsc) = \id_{A},
\end{align}
where $\ceps = S_{B} \circ \beps$ and $\epsb = S_{C} \circ \epsc$. As before, $\varepsilon$ denotes the counit functional.
\begin{lemma} \label{lemma:convolution} 
Let $\mathcal{A}$ be a regular multiplier Hopf algebroid with a countial base weight $(\mu_{B},\mu_{C})$,
  let $\bupsilon \in A\cdot
  \dbA$, $ \upsilonb' \in \dAb \cdot A$, $ \comega \in A \cdot \dcA$ and $ \omegac' \in \dAc \cdot A$, and   write
  $\upsilon:=\mu_{B} \circ \bupsilon$, $\omega:=\mu_{C}\circ \comega$,
  $\upsilon':=\mu_{B}\circ \upsilon'_{B}$, $\omega':=\mu_{C} \circ
  \omega'_{C}$. Then
  \begin{enumerate}
  \item  $\eps \circ \lambda(\bupsilon) =   \upsilon$,
    $\eps \circ \rho(\comega) =  \omega$, $\eps \circ
    \lambda(\upsilonb') = \upsilon'$, 
    $\eps \circ
    \rho(\omegac')  = \omega'$;
  \item $\lambda(\bupsilon)$ and $\lambda(\upsilonb')$ commute     with both $\rho(\comega)$ and $\rho(\omegac')$;
  \item $\upsilon \circ \rho(\comega) = \omega \circ
    \lambda(\bupsilon)$,  $\upsilon \circ \rho(\omegac') =
    \omega' \circ \lambda(\bupsilon)$, $\upsilon' \circ
    \rho(\comega) = \omega \circ \lambda(\upsilonb')$ and
    $\upsilon' \circ \rho(\omegac') = \omega' \circ
    \lambda(\upsilonb')$.
  \item Suppose that  factorizable functionals
        separate the points of $A$. Then     $\lambda(\bupsilon)=\lambda(\upsilonb')$ whenever $\upsilon=\upsilon'$, and      $\rho(\comega)=\rho(\omegac')$ whenever  $\omega=\omega'$.
  \end{enumerate}
\end{lemma}
\begin{proof}
  The relations in (1) and (2) follow immediately from the counit   property and the coassociativity conditions relating $\Delta_{B}$
  and $\Delta_{C}$, respectively. Combined, they imply
  \begin{align*}
    \upsilon \circ \rho(\comega) &= \eps \circ
    \lambda(\bupsilon) \circ \rho(\comega) = \eps \circ
    \rho(\comega) \circ \lambda(\bupsilon) = \omega \circ
    \lambda(\bupsilon)
  \end{align*}
  which is (3). Let us prove (4). Suppose that $\upsilon=\upsilon'$.  Then the assumption and  non-degeneracy of  $A$ imply that  functionals of the form like $\omega$ will
  separate the points of $A$, and by (3), $\omega \circ
  \lambda(\bupsilon) = \upsilon \circ \rho(\comega) = \upsilon' \circ
  \rho(\comega) = \omega \circ \lambda(\upsilonb')$, whence
  $\lambda(\bupsilon) = \lambda(\upsilonb')$. A similar argument   proves the assertion concerning $\rho(\comega)$ and
  $\rho(\omegac')$.
\end{proof}
We proceed with the study of convolution operators in the next subsection.

The next result is the key step towards the existence of modular automorphisms.
\begin{theorem}\label{theorem:modular}
  Let $\mathcal{A}$ be a regular multiplier Hopf algebroid with a counital base weight $(\mu_{B},\mu_{C})$. Then $A\cdot \phi=\phi\cdot A$ for every full left integral $\phi$ and $A\cdot \psi=\psi\cdot A$ for every full right integral $\psi$ for $(\mathcal{A},\mu_{B},\mu_{C})$.
\end{theorem}
\begin{proof}
  We only prove the assertion concerning a full left integral $\phi$.
 Let $\psi:=\phi \circ   S$. Then $\psi$ is a full right integral and $A \cdot \phi = A \cdot  \psi$ by Proposition \ref{proposition:uniqueness-full}. We show that $\phi
  \cdot A \subseteq A \cdot \psi$, and a similar argument   proves the reverse inclusion.

  Let $a,b,c \in A$.   By Lemma \ref{lemma:convolution} (3),
  \begin{align*}
    ((a \cdot \phi) \circ \lambda(\psi_{B} \cdot b))(c) &=
    ((\psi \cdot b) \circ \rho(a \cdot \cphi))(c)  =
    ((S(b) \cdot \phi) \circ S  \circ \rho(a \cdot \cphi))(c).
  \end{align*}
  Choose $b' \in A$ with $S(b) \cdot \phi = b' \cdot \psi$ and use \eqref{eq:convolution-right-invariance} to rewrite the expression above in the form
\begin{align*}
  ((b' \cdot \psi) \circ \rho(\phic \cdot c))(a) &=
  ((\phi \cdot c) \circ \lambda(b'\cdot \bpsi))(a) \\ &=
  ((\phi \cdot c) \circ S \circ \lambda(\psib \cdot a))(b') =
  ((S^{-1}(c) \cdot \psi') \circ \lambda(\psib \cdot a))(b').
\end{align*}
 Choose $c' \in A$
with $S^{-1}(c) \cdot \psi = c' \cdot \phi$ and use   Lemma \ref{lemma:convolution} (3) again to rewrite this expression in the form
\begin{align*}
  ((c' \cdot \phi)  \circ \lambda(\psib\cdot a))(b') &=
  ((\psi \cdot a) \circ \rho(c' \cdot \cphi))(b').
\end{align*}
We thus obtain
\begin{align*}
  \phi(\lambda(\psib \cdot b)(c) a) &=
      ((a \cdot \phi) \circ \lambda(\bpsi \cdot b))(c) \\ &=
  ((\psi \cdot a) \circ \rho(c' \cdot \cphi))(b') =
  \psi(a \rho(c' \cdot \cphi)(b')).
\end{align*}
Here, $b$ and $c \in A$ were arbitrary, and the linear span of all elements of the form
\begin{align*}
  \lambda(\psib\cdot b)(c) &= (S_{C}^{-1} \circ \psib\oo \id)((b
  \oo 1)\Delta_{C}(c))
\end{align*}
is equal to $AS_{C}^{-1}(\psib(A))=AC=A$ because $\lT$ and $\psib$ are surjective. Thus,  $\phi \cdot A \subseteq A
\cdot \psi=A \cdot \phi$ and consequently $A\cdot \phi = \phi \cdot A$.  
\end{proof}
\begin{theorem} \label{theorem:modular-automorphism}
  Let $(\mathcal{A},\mu_{B},\mu_{C},\bpsib,\cphic)$ be a measured regular multiplier Hopf algebroid. Then $\phi$ and  $\psi$ admit modular automorphisms $\sigma^{\phi}$ and  $\sigma^{\psi}$, respectively, and
  \begin{align*}
  \sigma^{\phi}|_{C} &=S^{2}|_{C}, &
     \Delta_{B} \circ \sigma^{\phi} &= (S^{2} \otimes \sigma^{\phi})
    \circ \Delta_{B}, &
    \Delta_{C} \circ \sigma^{\phi} &= (S^{2} \otimes \sigma^{\phi})
    \circ \Delta_{C}, \\
  \sigma^{\psi}|_{B} &=S^{-2}|_{B}, &
    \Delta_{B} \circ \sigma^{\psi} &= (\sigma^{\psi} \otimes S^{-2})
    \circ \Delta_{B},  &
    \Delta_{C} \circ \sigma^{\psi} &= (\sigma^{\psi} \otimes S^{-2})
    \circ \Delta_{C}.    
  \end{align*}
  If $\mathcal{A}$ is locally projective, then
  $\sigma^{\phi}(M(B))=M(B)$ and $\sigma^{\psi}(M(C))=M(C)$. If these  equations hold, then
for all $x\in M(B)$, $y\in M(C)$ and $a\in A$, 
 \begin{align} \label{eq:bphi-phib}
     \phib(xa) &= (S^{2} \circ \sigma^{\phi})(x)\phib(a), &
     \bphi(ax) &= \bphi(a)(\sigma^{\phi} \circ S^{2})^{-1}(x), \\ \label{eq:cpsi-psic}
  \psic(ya) &= (S^{-2}\circ \sigma^{\psi})(y)\psic(a), &
  \cpsi(ay) &= \cpsi(a)(\sigma^{\psi} \circ S^{-2})^{-1}(y).
\end{align}
\end{theorem}
\begin{proof}
From Theorem \ref{theorem:modular}, we conclude existence of a unique   bijection $\sigma^{\phi}\colon A \to A$ such that $\phi \cdot a =
  \sigma^{\phi}(a) \cdot \phi$ for all $a \in A$. This map is   easily seen to be an algebra automorphism.
 Lemma \ref{lemma:integrals-modular-base} implies that  $\sigma^{\phi}(y)=S^{2}(y)$  all $y\in C$. In particular, 
 the tensor product $S^{2} \otimes \sigma^{\phi}$ is well-defined on
  $\AlA$ and on $\ArA$.
  Two applications of  \eqref{eq:convolution-left-invariance}
  show that for all $a,b \in A$,
  \begin{align*}
    \rho(\phi \cdot a)(\sigma^{\phi}(b)) = S(\rho(\sigma^{\phi}(b)
    \cdot \phi)(a)) &=S(\rho(\phi \cdot b)(a)) \\
    &=S^{2}(\rho(a \cdot
    \phi)(b)) = S^{2}(\rho((\phi \cdot a) \circ \sigma^{\phi})(b)).
\end{align*}
Since $a \in A$ was arbitrary, we can conclude the desired formulas for   $\Delta_{B} \circ \sigma^{\phi}$ and $
    \Delta_{C} \circ \sigma^{\phi}$. 

If $\mathcal{A}$ is locally projective, then 
 the relation $\sigma^{\phi}(M(B))=M(B)$ follows     immediately from the relation $M(B) \cdot \phi = \phi \cdot M(B)$
    obtained in Theorem \ref{theorem:integrals-uniqueness}. 

The last equations follow easily from Remark \ref{remark:bphi-phib}.
\end{proof}
Let us look at our  examples again.
\begin{example} \label{example:modular-groupoid}
  Let $G$ be a locally compact, \'etale Hausdorff groupoid. Then  the function algebra $C_{c}(G)$ is commutative and hence every integral is tracial. The case of the convolution algebra will be considered in subsection \ref{subsection:modification-convolution}.
\end{example}
\begin{example} \label{example:modular-cb}
For the measured regular multiplier Hopf algebroid associated to  the tensor product $A=C\otimes B$ and suitable functionals $\mu_{B}$ and $\mu_{C}$ on $B$ and $C$ as in Example \ref{example:measured-cb}, we have $\phi = \psi = \mu_{C} \otimes \mu_{B}$ and  $\sigma^{\phi}=\sigma^{\psi}=S^{2} \otimes S^{-2}$.
\end{example}
\begin{example} \label{example:modular-ch}
For the measured regular multiplier Hopf algebroid associated to  a symmetric crossed  product $A=C\# H$, a faithful, $H$-invariant functional $\mu$ on $C$ and left and right integrals $\phi_{H}$ and $\psi_{H}$ of $(H,\Delta_{H})$ as in Example \ref{example:measured-ch}, the  integrals $\phi$ and $\psi$ are given by
\begin{align*}
  \phi(yh) &= \mu(y)\phi_{H}(h) = \phi(hy), &
  \psi(yh) &= \mu(y)\psi_{H}(h) = \psi(hy),
\end{align*}
for all $y\in C$ and $h\in H$,
see \eqref{eq:integrals-ch-symmetric}, and hence their modular automorphisms are given by
\begin{align*}
  \sigma^{\phi}(yh) &= y\sigma_{H}(h), & \sigma^{\psi}(yh) &= y\sigma'_{H}(h),
\end{align*}
where $\sigma_{H}$ and $\sigma'_{H}$ denote the modular automorphisms of $\phi_{H}$ and $\psi_{H}$.
\end{example}
\begin{example} 
Consider  the measured regular multiplier Hopf algebroid associated to the two-sided crossed product $A=C \# H\# B$, 
 suitable $H$-invariant functionals $\mu_{B}$ and $\mu_{C}$ on $B$ and $C$ and left and right integrals $\phi_{H}$ and $\psi_{H}$ of $(H,\Delta_{H})$  as in Example \ref{example:measured-chb}. The left integral $\phi$ is given by
 \begin{align*}
   \phi(yhx) &= \mu_{C}(y)\phi_{H}(h)\mu_{B}(x)
 \end{align*}
for all $y\in C$, $h\in H$, $x\in B$. Let us compute its modular automorphism $\sigma^{\phi}$.  By Lemma \ref{lemma:integrals-modular-base},  $\sigma^{\phi}(y)=S^{2}(y)$ for all $y\in C$. 
   Denote by  $\delta_{H}$ the modular element of $\phi_{H}$, see \cite[Proposition 3.8]{daele}. Then  $\phi_{H} \circ S_{H}=\delta_{H} \cdot \phi_{H}$ is right-invariant on $H$ and  hence  $\psi':=\delta_{H}\cdot \phi$ is right-invariant; see also  Example \ref{example:measured-chb}. The modular automorphism $\sigma^{\psi'}$ of $\psi'$ satisfies $\sigma^{\psi'}(x)=S^{-2}(x)$ for all $x\in B$ by  Lemma \ref{lemma:integrals-modular-base} again, and hence
 \begin{align*}
   \sigma^{\phi}(x)=\delta^{-1}_{H}\sigma^{\psi'}(x)\delta_{H}=S^{-2}(x) \actleft \delta_{H}.
 \end{align*}
 Finally, using $H$-invariance of $\mu_{C}$ and $\mu_{B}$, we find
\begin{align*}
  \phi(h'yhx) = \mu_{C}(y)\phi_{H}(h'h)\mu_{B}(x) = \mu_{C}(y)\phi_{H}(h\sigma_{H}(h'))\mu_{B}(x) = \phi(yhx\sigma_{H}(h')),
\end{align*}
where $\sigma_{H}$ denotes the modular automorphism of $\phi_{H}$, and hence $\sigma^{\phi}(h)=\sigma_{H}(h)$.
\end{example}

\subsection{Convolution operators and the dual algebra}
\label{subsection:convolution}
The results obtained so far immediately imply the existence of modular elements:
\begin{corollary} \label{corollary:modular-element} Let $(\mathcal{A},\mu_{B},\mu_{C},\bpsib,\cphic)$ be a measured regular multiplier Hopf algebroid. Then there exists a unique invertible multiplier $\delta \in   M(A)$ such that $\psi = \delta \cdot \phi$.
\end{corollary}
\begin{proof}
 By Proposition \ref{proposition:uniqueness-full}, $\psi \lesssim \phi$ and $\phi\lesssim \psi$, and by Theorem \ref{theorem:modular-automorphism} $\phi$ and $\psi$ admit modular automorphisms. Now, apply Lemma \ref{lemma:lesssim} (2).
\end{proof}
To determine the behaviour of the comultiplication,
counits and antipode on $\delta$, we  need a few more results on the convolution operators introduced above.

Let $(\mathcal{A},\mu_{B},\mu_{C},\bpsib,\cphic)$ be a measured regular multiplier Hopf algebroid. Then
Proposition \ref{proposition:uniqueness-full} and  Theorem \ref{theorem:modular-automorphism} imply that the subspaces
\begin{align*}
  A\cdot \phi, \quad \phi \cdot A, \quad A\cdot \psi, \quad \psi \cdot A
\end{align*}
of $\dA$ coincide, and Theorem \ref{theorem:integrals-uniqueness} implies that it does not depend on the choice of $\cphic$ and $\bpsib$.
 Since $\phi,\psi$ are factorizable, it follows that all functionals in the space
\begin{align} \label{eq:hata}
  \hat A:=A\cdot \phi = \phi \cdot A = A \cdot \psi = \psi \cdot A
\end{align}
are factorizable, that is, $\hat A \subseteq A^{\sqcup}$.   Functionals in $\hat A$ naturally extend to $M(A)$:
\begin{lemma} \label{lemma:extension-multipliers}   There exists a unique embedding
  $j\colon\hat A \to \dual{M(A)}$ such that
  \begin{align*}
    j(a \cdot \phi)(T) &= \phi(Ta), & j(\phi \cdot       a)(T) &= \phi(aT), & j(a \cdot \psi)(T) &= \psi(Ta),
    & j(\psi \cdot a)(T) &= \psi(aT)
  \end{align*}
  for every   $T\in M(A)$ and $a\in A$.
\end{lemma}
\begin{proof}
  The point is to show that the formulas given above are   compatible in the sense that for each $\upsilon \in \hat A$, the   extension $j(\upsilon)$ is well-defined, and this can easily  be done using Theorem \ref{theorem:modular-automorphism} and Corollary
  \ref{corollary:modular-element}. 
\end{proof}
We henceforth regard elements of $\hat A$ as functionals on $M(A)$
without mentioning the embedding $j$  explicitly.
\begin{proposition} \label{proposition:convolution} Let
 $(\mathcal{A},\mu_{B},\mu_{C},\bpsib,\cphic)$ be a measured regular multiplier Hopf algebroid.  For every $\upsilon
  \in A^{\sqcup}$ and $b\in A$,
 there exist      $\rho(\upsilon)(b),\lambda(\upsilon)(b) \in M(A)$ such that
    \begin{align*}
      \rho(\upsilon)(b) &= \rho(\upsilonb)(b) \text{ in } R(A), & 
\rho(\upsilon)(b) &=      \rho(\bupsilon)(b) \text{ in } L(A), \\
\lambda(\upsilon)(b) &= \lambda(\upsilonc)(b) \text{ in } R(A),  &
\lambda(\upsilon)(b) &= \lambda(\cupsilon)(b) \text{ in } L(A).
    \end{align*}
Moreover, the following relations hold   for all $\upsilon 
  \in A^{\sqcup}$ and $\omega \in \hat A$:
  \begin{align*}
    \upsilon \circ \rho(\omega)&=\omega \circ \lambda(\upsilon) \in A^{\sqcup}, &
    \rho(\upsilon \circ \rho(\omega)) &=
    \rho(\upsilon)\rho(\omega), \\
    \upsilon \circ \lambda(\omega) &= \omega \circ
    \rho(\upsilon) \in A^{\sqcup}, & \lambda(\upsilon \circ
    \lambda(\omega)) &= \lambda(\upsilon)\lambda(\omega).
  \end{align*}
\end{proposition}
\begin{proof}
Let
 $\upsilon
  \in A^{\sqcup}$, $b\in A$ and
 $\omega \in \hat A$. 

Since $\phi$ is faithful, elements of $\hat A \subseteq A^{\sqcup}$ separate the points of $A$, and  so Lemma \ref{lemma:convolution} (4)
  implies $\lambda(\bomega)(b) = \lambda(\omegab)(b)$ and $
  \rho(\comega)(b) = \rho(\omegac)(b)$. We therefore drop the   subscripts $B$ and $C$ and write $\lambda(\omega)$ and
  $\rho(\omega)$ from now on.  

Suppose $\omega=\phi\cdot a$ with $a\in A$. Then
  \begin{align*}
    \omega(\rho(\upsilonc)(b)) = \phi(a\rho(\upsilonc)(b)) &= (\phi
    \underset{\mu_{C}}{\otimes} \upsilon)((a \oo 1)\Delta_{C}(b)) = \upsilon(\lambda( \phi
    \cdot a)(b)) = \upsilon(\lambda(\omega)(b)).
  \end{align*}

  A similar calculation shows that $\omega(\rho_{C}(\upsilon)(b)) =
  \upsilon(\lambda(\omega)(b))$. For $\omega\in \hat A$ of the   form $\omega=c\cdot \phi \cdot a$ with $a,c\in A$, we obtain
  \begin{align*}
    \phi((a\rho(\upsilonc)(b))c) = \upsilon(\lambda(\omega)(b)) =
    \phi(a(\rho(\cupsilon)(b)c)).
  \end{align*}
  Since $a,c\in A$ were arbitrary and $\phi$ is faithful, we can   conclude that $(a\rho(\upsilonc)(b))c =
  a(\rho(\cupsilon)(b)c)$ for all $a,c\in A$ so that
  $\rho(\cupsilon)(b)$ and $\rho_{C}(\upsilon)(b)$ form a two-sided   multiplier $\rho(\upsilon)(b)$ as claimed.

  Along the way, we just showed that $\omega \circ \rho(\upsilon) =
  \upsilon \circ \lambda(\omega)$. One easily verifies that this   composition belongs to $A^{\sqcup}$, for example, $_{B}(\upsilon \circ
  \lambda(\omega)) =\bupsilon \circ \lambda(\omega)$.

Let now also $\omega' \in \hat A$. Then $\lambda(\omega')$ commutes with
  $\rho(\omega)$ by Lemma \ref{lemma:convolution} and hence
  \begin{align*}
    \omega' \circ \rho(\upsilon) \circ \rho(\omega) = \upsilon \circ
    \lambda(\omega')\circ \rho(\omega) = \upsilon \circ \rho(\omega)
    \circ \lambda(\omega') = \omega' \circ \rho(\upsilon \circ
    \rho(\omega)).
  \end{align*}
  Since $\omega' \in \hat A$ was arbitrary and $\hat A$ separates the   points of $A$, we can conclude $    \rho(\upsilon \circ \rho(\omega)) =
    \rho(\upsilon)\rho(\omega)$.   A similar argument proves the remaining equation.
\end{proof}
The first part of the preceding result can be rewritten as follows.
\begin{corollary} \label{corollary:convolution}
  Let $(\mathcal{A},\mu_{B},\mu_{C},\bpsib,\cphic)$ be a measured
  regular multiplier Hopf algebroid.  Then for all $\upsilon \in
  A^{\sqcup}$ and $a,b,c\in A$,
  \begin{gather}
    a   ( (\upsilon \underset{\mu_{B}}{\otimes} \id)(\Delta_{B}(b)(1\otimes  c)))  = a \lambda(\upsilon)(b)c=  (\upsilon \underset{\mu_{C}} \otimes \id)((1\otimes a)\Delta_{C}(b))c,\\
    a((\id \underset{\mu_{B}}{\otimes} \upsilon)(\Delta_{B}(b)(c\otimes 1))) = a\rho(\upsilon)(b)c =  ((\id \underset{\mu_{C}}{\otimes} \upsilon)((a\otimes 1)\Delta_{C}(b)))c.
  \end{gather}
\end{corollary}

The results above imply that $\hat{A}$ is an algebra and $A$ and $A^{\sqcup}$ are $\hat{A}$-bimodules:
\begin{theorem} \label{theorem:dual-algebra} Let
  $(\mathcal{A},\mu_{B},\mu_{C},\bpsib,\cphic)$ be a measured regular multiplier Hopf
  algebroid. Then the subspace $\hat A =A\cdot\phi \subseteq A^{\sqcup}$ is  an algebra and $A$ and $A^{\sqcup}$ are $\hat A$-bimodules with respect to the
  products given by
  \begin{gather*}
    \omega \omega' := \omega \circ \rho(\omega') = \omega' \circ
    \lambda(\omega) \quad \text{for all } \omega,\omega' \in \hat A,
    \\ \begin{aligned} \omega \ast a &= \rho(\omega)(a), & a \ast
      \omega &= \lambda(\omega)(a) && \text{for all } a\in A, \omega
      \in \hat
      A, \\
      \omega\upsilon &= \upsilon\circ \lambda(\omega), &
      \upsilon\omega &= \upsilon \circ \rho(\omega) && \text{for all }
      \upsilon \in A^{\sqcup}, \omega \in \hat A.
    \end{aligned}
  \end{gather*}
As such, all are non-degenerate, $\hat A$ and $A$ are idempotent, and $A$ and $A^{\sqcup}$ are faithful.
\end{theorem}
\begin{proof}
  We first show that the product defined on $\hat A$ takes values in  $\hat A$ again.  Let $\omega,\omega' \in \hat A$. Then $\omega
  \circ \rho(\omega')=\omega' \circ \lambda(\omega)$ by Proposition  \ref{proposition:convolution}. To see that this functional lies in
  $\hat A$, write $\omega=a\cdot \phi$ and $\omega' = b\cdot \psi$  with $a,b\in A$ and $a \otimes b = \sum_{i} \Delta_{B}(d_{i})(e_{i}
  \otimes 1)$ with $d_{i},e_{i} \in A$. Then $    (\omega \circ \rho(\omega'))(c)$ is equal to
  \begin{align*}
    (\phi \underset{\mu_{B}}{\otimes} \phi)(\Delta_{B}(c)(a \oo b)) 
    &=\sum_{i} (\phi  \underset{\mu_{B}}{\otimes} \phi)(\Delta_{B}(cd_{i})(e_{i} \otimes 1))
    \\ & = \sum_{i} \phi(\cphic(cd_{i})e_{i}) =\sum_{i}\phi(cd_{i}\cphic(e_{i}))
  \end{align*}
  for all $c \in A$ and hence
 \begin{align} \label{eq:dual-product} \omega \circ \rho(\omega') =
   f\cdot \phi \ \text{ if } \ \omega=a\cdot \phi, \ \omega'=b\cdot
   \phi, \ f= (\cphic \otimes \id)(\Tl^{-1}(a\otimes b)).
 \end{align}
  
Proposition \ref{proposition:convolution} now implies that  the products defined above turn $\hat A$ into an algebra and $A$ and  $A^{\sqcup}$ into $\hat A$-bimodules.

  Equation \eqref{eq:dual-product} and bijectivity of the canonical  maps $\Tl,\Tr$ imply that $\hat A$ is idempotent as an algebra and
  $A$ is idempotent as an $\hat A$-bimodule. These facts and  non-degeneracy of the pairing $\hat A \times A \to A$, $(\upsilon,a)
  \mapsto \upsilon(a)$ imply that $A$ is non-degenerate and faithful  as an $\hat A$-bimodule. But $A$ being faithful and idempotent as an
  $\hat A$-bimodule implies that the algebra $\hat A$ is  non-degenerate, and that  $A^{\sqcup}$ is non-degenerate and
  faithful as an $\hat A$-bimodule.
\end{proof}
In  \cite{timmermann:integrals}, we show that the algebra $\hat{A}$ constructed above can be endowed with the structure of a measured regular multiplier Hopf algebroid again, which can be regarded as a generalized Pontrjagin dual to the original measured regular multiplier Hopf algebroid.

We shall also need the following relations.
\begin{lemma} \label{lemma:convolution-2}
Let $\upsilon
  \in A^{\sqcup}$. Then the following relations hold:
\begin{enumerate}
\item $\rho(\upsilon) \circ S = S \circ \lambda(\upsilon \circ S)$
    and $\lambda(\upsilon) \circ S = S \circ \rho(\upsilon \circ S)$;
  \item for all $x,x',x'' \in B$, $y,y',y'' \in C$,
    \begin{align*}
      \rho(S_{B}(x'')x \cdot\upsilon \cdot y''x') (yby')&= S_{C}(y'')y
      \rho(\upsilon)(xbx) y'x'', \\
      \lambda(x'' y \cdot\upsilon \cdot S_{C}(y'') y')(xbx') &=
      y''x\lambda(\upsilon)(y'by)x'S_{B}(x'');
    \end{align*}
  \item if  $(\mathcal{A},\mu_{B},\mu_{C},\bpsib,\cphic)$ is a measured multiplier Hopf $*$-algebroid, then we have
 $\rho(\ast \circ \upsilon\circ \ast) = \ast \circ \rho(\upsilon)
    \circ \ast$ and $\lambda(\ast \circ \upsilon \circ \ast) = \ast
    \circ \lambda(\upsilon) \circ \ast$.
  \end{enumerate}
\end{lemma}
\begin{proof}
  Straightforward.
\end{proof}

\subsection{The modular element}
\label{subsection:modular-element}
 We now  determine the behaviour of the comultiplication, counit and antipode on the modular elements relating a left integral $\phi$ to the right integrals $\phi
\circ S^{-1}$ and $\phi\circ S$. 
\begin{theorem} \label{theorem:modular-element-second} Let $(\mathcal{A},\mu_{B},\mu_{C},\bpsib,\cphic)$ be a measured regular multiplier Hopf algebroid. Write $\psi^{-}=\phi\circ S^{-1}$ and $\psi^{+}=\phi\circ S$. Then there exist   unique invertible multipliers $\delta^{-}, \delta^{+} \in M(A)$ such   that
  $\psi^{+} = \delta^{+} \cdot \phi$ and $\psi^{-} =
  \phi \cdot \delta^{-}$. These elements satisfy
  \begin{gather*}
    \begin{aligned}
      \bphi(a)\delta^{+} &= \lambda(\phi)(a) = \delta^{-} \phib(a)
      \text{ for all } a \in A, &
    \end{aligned} \\
    \begin{aligned}
             S(\delta^{+}) &= (\delta^{-})^{-1}, &
      \eps \cdot \delta^{-} &= \eps = \delta^{+} \cdot \eps, 
    \end{aligned} \\
    \begin{aligned}
 \Delta_{B}(\delta^{+}) &= \delta^{+} \oo
  \delta^{+}, &      \Delta_{B}(\delta^{-}) &= \delta^{+} \oo \delta^{-}, &
\Delta_{C}(\delta^{-}) &= \delta^{-} \oo
  \delta^{-}, &      \Delta_{C}(\delta^{+}) &= \delta^{-} \oo \delta^{+}.
    \end{aligned}
\end{gather*}
  If  $(\mathcal{A},\mu_{B},\mu_{C},\bpsib,\cphic)$ is a measured  multiplier Hopf $*$-algebroid, then $\delta^{+}=(\delta^{-})^{*}$.

Finally, if $\sigma^{\phi}(M(B))=M(B)$, then
\begin{align*}
 x \delta^{-}  &= \delta^{-} S^{2}(\sigma^{\phi}(x)), &   
x \delta^{+} &=
  \delta^{+} \sigma^{\phi}(S^{2}(x)) &&\text{for all } x\in M(B), \\
\delta^{-}y   &= S^{-2}(\sigma^{\psi^{-}}(y)) \delta^{-}, &
y\delta^{+} &=  \delta^{+}\sigma^{\psi^{+}}(S^{-2}(y)) &&\text{for all } y\in M(C),
\end{align*}
where
 $\sigma^{\psi^{-}}=S\circ(\sigma^{\phi})^{-1}\circ S^{-1}$ and $\sigma^{\psi^{+}} = S^{-1}\circ (\sigma^{\phi})^{-1}\circ S$ are the modular automorphisms of $\psi^{-}$ and $\psi^{+}$, respectively.
\end{theorem} 
\begin{proof}  
  The compositions $\psi^{+}$ and $\psi^{-}$ are full and faithful right integrals by Corollary
  \ref{corollary:integrals-bimodule-antipode} and Remark \ref{remark:full},
  and of the form $\psi^{+} = \delta^{+} \cdot \phi$ and $\psi^{-} =
  \phi \cdot \delta^{-}$ with unique invertible multipliers
  $\delta^{-},\delta^{+} \in M(A)$ by Corollary
  \ref{corollary:modular-element} and Theorem
  \ref{theorem:modular-automorphism}.  By Proposition
  \ref{proposition:convolution} and 
  \eqref{eq:convolution-left-invariance},
  \begin{align*}
    \phi(\lambda(\phi)(a)b) &= ((b \cdot \phi) \circ \lambda(\phi))(a)
    = (\phi \circ \rho(b\cdot \phi))(a) =
    ((\phi \circ S^{-1}) \circ \rho(\phi \cdot a))(b).
\end{align*}
Another application of  Proposition
  \ref{proposition:convolution} and 
  \eqref{eq:convolution-right-invariance} shows that this is equal to
  \begin{align*}  ((\phi \cdot a) \circ \lambda(\psi^{-}))(b) &= (\phi \cdot     a)(\bpsib^{-}(b)) = \phi(a\bpsib^{-}(b)) = \psi^{-}(\phib(a)b) =  \phi(\delta^{-}\phib(a)b)
  \end{align*}
  for all $a,b\in A$, and hence $\lambda(\phi)(a) = \delta^{-}\phib(a)$ for   all $a\in A$. A similar calculation shows that $\lambda(\phi)(a) =
  \bphi(a)\delta^{+}$ for all $a\in A$.
We have $(\delta^{-})^{-1} = S(\delta^{+})$ because
\begin{align*}
  \phi
= (\delta^{+} \cdot \phi) \circ S^{-1} =
  (\phi \circ S^{-1}) \cdot S(\delta^{+}) = \phi \cdot \delta^{-}   S(\delta^{+}).
\end{align*}

 The properties of the counit imply that
  \begin{align*}
    \eps(\delta^{-}\phib(a)b) &= \eps(\lambda(\phi)(a)b) = (\phi \underset{\mu_{B}}{\oo}
    \eps)(\Delta_{B}(a)(1 \oo b)) = \phi(a\beps(b)) = \eps(\phib(a)b)
  \end{align*}
  for all $a,b \in A$, whence $\eps \cdot \delta^{-} = \eps$. A similar   calculation shows that $\delta^{+}
  \cdot \eps = \eps$.

  The relations for the comultiplication require a bit more work. For all
  $a,b\in A$,
  \begin{align*}
    (\phi \underset{\mu_{B}}{\oo} \phi)(\Delta_{B}(\delta^{-} a)(1 \oo b)) &=
    \phi(\lambda(\phi)(\delta^{-} a)b) \\ &= \phi(\delta^{-}\phib(\delta^{-} a) b) = (\psi^{-} \underset{\mu_{B}}{\oo} \psi^{-})(\Delta_{B}(a)(1\oo b)).
  \end{align*}
 The map $\Tr$ being surjective, we can conclude  that for all
  $c,d \in A$,
  \begin{align*}
    (\phi \underset{\mu_{B}}{\oo} \phi)(\Delta_{B}(\delta^{-})(c \oo d)) &= (\psi^{-} \underset{\mu_{B}}{\oo}
    \psi^{-})(c \oo d) 
    = \phi(\delta^{-} S_{B}(\phib(\delta^{-} c)) d).
  \end{align*}
Since $\phib(\delta^{-} c)S^{-1}(\delta^{-}) = \phib(\delta^{-}   c)(\delta^{+})^{-1} = (\delta^{-})^{-1}\bphi(\delta^{-} c)$, the expression above equals
  \begin{align*}
         \phi(S(\phib(\delta^{-} c)S^{-1}(\delta^{-}))d) &=
         \phi(S((\delta^{-})^{-1}\bphi(\delta^{-} c))d) \\ &=\phi(S_{B}(\bphi(\delta^{-}          c))S(\delta^{-})^{-1}d) = (\phi \underset{\mu_{B}}{\oo} \phi)(\delta^{-} c \oo S(\delta^{-})^{-1}d).
  \end{align*}
  Consequently, $\Delta_{B}(\delta^{-}) = \delta^{-} \oo S(\delta^{-})^{-1}$.
  Since the antipode reverses the comultiplication (see \eqref{dg:galois-inverse}), we can conclude
  \begin{align*}
    \Delta_{C}(\delta^{+}) = \Delta_{C}(S^{-1}(\delta^{-})^{-1}) =
    \delta^{-} \oo S^{-1}(\delta^{-})^{-1} =  \delta^{-} \oo \delta^{+}.
  \end{align*}
  A similar argument shows that $\Delta_{B}(\delta^{-}) =\delta^{+} \oo
  \delta^{-}$. 

Let us compute $\Delta_{C}(\delta^{-})$. For all $a\in A$,
  \begin{align*}
    \Delta_{C}(\delta^{-}) (1 \oo \phib(a)) = \Delta_{C}(\delta^{-} \phib(a))
 &    = \Delta_{C}(\bphi(a) \delta^{+}) \\ &= \delta^{-} \oo
    \bphi(a)\delta^{+}  = (\delta^{-} \oo \delta^{-})(1 \oo \phib(a)).
  \end{align*}
Since $\phib(A)A=A$ and $\ArA$ is non-degenerate as a right module over $1 \oo A$   by assumption, this relation  implies $\Delta_{C}(\delta^{-}) = \delta^{-} \oo
  \delta^{-}$. A similar reasoning shows that 
  $\Delta_{B}(\delta^{+})=\delta^{+} \oo \delta^{+}$.

  If $(\mathcal{A},\mu_{B},\mu_{C},\bpsib,\cphic)$ is a measured multiplier Hopf $*$-algebroid,  then
  \begin{align*}
    \phi(a^{*}(\delta^{-})^{*}) &= \overline{\phi(\delta^{-} a)} = \overline{\phi
    \circ S^{-1}(a)} = \phi(S^{-1}(a)^{*}) = \phi(S(a^{*})) = \phi(a^{*}\delta^{+})
  \end{align*}
  for all $a \in A$, where we used \eqref{eq:counit-antipode-involution}, and hence $(\delta^{-})^{*}=\delta^{+}$. 

Finally, suppose that $\sigma^{\phi}(M(B))=M(B)$. Of the intertwining relations for $\delta^{+},\delta^{-}$  and multipliers of $B$ or $C$, we only prove the    first one; the others follow similarly. From Theorem \ref{theorem:modular-automorphism}, we conclude that
 for all $a\in A$ and $x\in B$,
  \begin{align*}
    \delta^{-} S^{2}(\sigma^{\phi}(x)) \phib(a) = \delta^{-} \phib(xa) &=
    \lambda(\phi)(xa) = x\lambda(\phi)(a) = x \delta^{-} \phib(a). \qedhere
  \end{align*}
\end{proof}

Let  $(\mathcal{A},\mu_{B},\mu_{C},\cphic,\bpsib)$ be a measured regular multiplier Hopf algebroid. Suppose that $\mathcal{A}$ is proper and that $\sigma^{\phi}(B)=B$. Under a mild non-degeneracy assumption, the left integral $\phi$ can be rescaled such that it becomes a left and right integral, like a Haar integral in the unital case except for the normalization. 

To formulate this condition, we use the following observation.  By \eqref{eq:bphi-phib}, we can define a multiplier $\bphi(y) \in M(B)$ such that for all $x\in B$,
  \begin{align} \label{eq:bphiy}
    x\bphi(y) &= \bphi(xy) = \bphi(yx) = \bphi(y)(\sigma^{\phi}\circ S^{2})^{-1}(x).
  \end{align}
\begin{theorem} \label{theorem:integrals-proper} 
Let $(\mathcal{A},\mu_{B},\mu_{C},\cphic,\bpsib)$ be a measured regular multiplier Hopf  algebroid.  Assume that $\mathcal{A}$ is proper, that  $\sigma^{\phi}(M(B))=M(B)$   and  that  $ \bphi(C)$ contains an invertible multiplier $z$. Then the functional $h:=z^{-1} \cdot \phi$ is a 
 left and a right integral for $(\mathcal{A},\mu_{B},\mu_{C})$,  and $(\mathcal{A},\mu_{B},\mu_{C},{_{B}h_{B}},{_{C}h_{C}})$ is a measured regular
    multiplier Hopf algebroid.
\end{theorem}
\begin{proof}
Denote by $\delta^{+}\in M(A)$  the multiplier satisfying $\delta^{+} \cdot \phi = \phi\circ S$ as in Theorem \ref{theorem:modular-element-second}. Let $a\in A$, $x,x' \in B$ and choose $y\in C$ such that $\bphi(y)=z$.
Then 
\begin{align*}
xz\delta^{+}S(x')a=\bphi(xy)\delta^{+} S(x')a =  \lambda(\phi)(xy)S(x')a = (\bphi \otimes \id)(\Delta_{B}(xy)(1 \otimes S(x')a)
\end{align*}
which is equal to
\begin{align*}
  (\bphi \otimes \id)(yx' \otimes xa) = S(x'z)xa = xS(z)S(x')a.
\end{align*}
Since $x,x'\in B$ and $a\in A$ were arbitrary, we can conclude that
\begin{align*}
  \delta^{+} = z^{-1}S(z).
\end{align*}
The functional $z^{-1}\cdot h$ is a left integral by Corollary \ref{corollary:integrals-bimodule-antipode}, and clearly full and faithful.  Subsequently using Lemma \ref{lemma:integrals-modular-base}, the definition of $\delta^{+}$ and the formula above, we find
\begin{align*}
  h(S(a))  = \phi(S(a)z^{-1}) &=\phi(S(S^{-1}(z^{-1})a)) \\ &= \phi(S^{-1}(z^{-1})a\delta^{+}) = \phi(a\delta^{+}S(z^{-1}))
=\phi(az^{-1})=h(a)
\end{align*}
for all $a\in A$. By Corollary  \ref{corollary:integrals-bimodule-antipode}, $h\circ S=h$ is also a right integral. 
\end{proof}

Let us look at the examples listed in Subsection \ref{subsection:measured} again:
\begin{example} \label{example:modular-groupoid-functions}
Let $G$ be a locally compact, \'etale Hausdorff groupoid and let $\mu$ be a Radon measure on the space of units $G^{0}$ with full support.  Consider the tuple
\begin{align*} 
  (\mathcal{A},\mu_{B},\mu_{C},\bpsib,\cphic)
\end{align*}
formed by the  multiplier Hopf $*$-algebroid of functions on $G$  defined in Example \ref{example:hopf-groupoid-functions}, the 
 base weight  defined in \eqref{eq:base-weight-groupoid-functions}  and the left- and the right-invariant maps defined in \eqref{eq:integrals-groupoid-functions}.
Note that the compositions $\phi=\mu_{C} \circ \cphic$ and $\psi=\mu_{B} \circ \bpsib$ satisfy $\psi=\phi\circ S=\phi\circ S^{-1}$.

We saw in Example \ref{example:measured-groupoid-functions} that this tuple  is a measured multiplier Hopf $*$-algebroid if  the measure $\mu$ is continuously quasi-invariant. Conversely, if the tuple is a measured multiplier Hopf $*$-algebroid, then   by Theorem \ref{theorem:modular-element-second}, $\psi=\delta \cdot \phi$ with $\delta:=\delta^{+}=\delta^{-}$ so that $\mu$ is continuously quasi-invariant with Radon-Nikodym derivative   $D=\delta^{-1}$.
\end{example}
For the measured multiplier Hopf algebroids associated to the convolution algebra of a locally compact, \'etale, Hausdorff groupoid, see Example \ref{example:measured-groupoid-algebra}, and to a tensor product $A=C\otimes B$, see Example \ref{example:measured-cb}, the modular elements $\delta^{-}$ and $\delta^{+}$ are evidently trivial.
\begin{example}
   Consider the measured regular multiplier Hopf algebroid associated to a two-sided crossed product $C \# H\# B$ and suitable $H$-invariant functionals $\mu_{B}$ and  $\mu_{C}$ as in Example \ref{example:measured-chb}. By \cite[Proposition 3.10]{daele},  $H$ has a modular element $\delta_{H}$ which  satisfies $\phi_{H} \circ S_{H} = \delta_{H} \cdot \phi_{H}$, and therefore  $\phi_{H} \circ S_{H}^{-1} = \phi_{H} \cdot S_{H}(\delta^{-1}_{H}) = \phi_{H} \cdot \delta_{H}$.  Short calculations show that the modular multipliers $\delta^{+}$ and $\delta^{-}$ of Theorem  \ref{theorem:modular-element-second} coincide with $\delta_{H}$.
\end{example}
\subsection{Faithfulness of integrals}
\label{subsection:faithful}
Non-zero integrals on multiplier Hopf algebras are always faithful \cite{daele}.  We now prove a corresponding statement for integrals on regular multiplier Hopf algebroids,
where the former need to be full, the latter locally projective, and the base algebra full. Note that all of these assumptions  become vacuous in the case of multiplier Hopf algebras. 
\begin{theorem} \label{theorem:integrals-faithful}
Let $\mathcal{A}=(A,B,C,S_{B},S_{C},\Delta_{B},\Delta_{C})$ be a   regular multiplier Hopf algebroid with counital base weight $(\mu_{B},\mu_{C})$. If $\mathcal{A}$ is locally projective and $B$ has local units, then  every full left or right integral for $(\mathcal{A},\mu_{B},\mu_{C})$ is faithful.
\end{theorem}
\begin{proof}
  We only prove the assertion for left integrals, closely following   the argument in \cite{daele}. A similar reasoning applies to right   integrals.

  Let $\phi$ be a full left integral on a projective regular   multiplier Hopf algebroid $\mathcal{A}$ with base weight $(\mu_{B},\mu_{C})$, let  $a\in A$ and suppose $ a\cdot \phi=0$.  Since $\mu_{C}$ is faithful,
  we can conclude that then also $a\cdot \cphic=0$.  Let $b\in A$ and $\upsilonb
  \in \Hom(\Ab,\Bb)$.   We show that
  \begin{align}\label{eq:faithful}
   \eps(c)=0, \quad \text{where }  c := \lambda(\upsilonb \cdot b)(a),
  \end{align}
 and then we can   conclude from Lemma \ref{lemma:convolution} that
  \begin{align*}
    (\mu_{B} \circ \upsilonb)(ba) = (\eps \circ \lambda(\upsilonb
    \cdot b))(a) = \eps(c)=0. 
  \end{align*}
  Using the facts that $\mu_{B}$ is faithful, $b\in A$ is arbitrary   and $A$ is non-degenerate, and that $\upsilonb \in \dAb$ is arbitrary   and $\dAb$ separates the points of $A$ because the module $\Ab$ is   locally projective, we can deduce that $a=0$.

Let us prove \eqref{eq:faithful}. 
  Lemma \ref{lemma:convolution} and \eqref{eq:convolution-left-invariance} imply for all $d\in A$ that
  \begin{align*}
      \rho(\phic \cdot d)(c) &= (\lambda(\upsilonb \cdot b) \circ \rho(\phic \cdot d) )(a) =
      (\lambda(\upsilonb \cdot b) \circ S \circ \rho(a \cdot \cphi))(d)
      = 0,
  \end{align*}
 and hence for all $f\in A$ and $\omega\in \dA$  of the form $\omega=\mu_{B}\circ \omegab$ with $\omegab\in \Hom(\Ab,\Bb)$,
  \begin{align*}
0=   (\omega \cdot f)( \rho(\phic \cdot d)(c)) = (\phi \cdot d)(\lambda(\omegab \cdot f)(c)).
  \end{align*}
Writing $(f \oo
  1)\Delta_{C}(c) = \sum_{j} f_{j} \oo c_{j}$ with $f_{j},c_{j} \in   A$, the equation above becomes
  \begin{align} \label{eq:faithful-aux}
    0 =  \sum_{j} \phi(dc_{j}S^{-1}(\omegab(f_{j}))).
  \end{align}
Since $\mathcal{A}$ is locally projective, we can find finitely many $\omega^{i}_{B} \in \Hom(\Ab,\Bb)$ and $e_{i} \in \Hom(\Bb,\Ab)$ such that  $\sum_{i} e_{i}(\omega^{i}_{B}(f_{j}))=f_{j}$ for all $j$, and since $B$ has local units, we can without loss of generality assume that $e_{i} \in A$ in the sense that $e_{i}(x) = e_{i}x$ for all $i$. 
By Theorem \ref{theorem:modular}, we can find   elements $d_{i} \in A$ such that $S^{-1}(e_{i}) \cdot \phi = \phi
  \cdot d_{i}$ for all $i$. Now, \eqref{eq:faithful-aux} implies
  \begin{align*}
    0 = \sum_{i,j} \phi(d_{i}c_{j}S^{-1}(\omega_{B}^{i}(f_{j}))) = \sum_{i,j}
    \phi(c_{j}S^{-1}(e_{i}\omega_{B}^{i}(f_{j}))) = \sum_{j} \phi(c_{j}S^{-1}(f_{j})).
  \end{align*}
  where we used Lemma \ref{lemma:convolution} again.  The second   diagram in \eqref{dg:antipode} shows that
 \begin{align*}
   \sum_{j} c_{j}S^{-1}(f_{j}) = \sum_{j} S^{-1}(f_{j}S(c_{j})) =
   S^{-1}(f S_{B}(\beps(c))) =\beps(c)S^{-1}(f),
 \end{align*}
and hence
\begin{align*}
   0 = \phi(\beps(c)S^{-1}(f)) = \mu_{B}(\beps(c)\bphi(S^{-1}(f))).
 \end{align*}
Since $f\in A$ was arbitrary, $\bphi$ is surjective and $\mu_{B}$ is faithful, we can conclude $\beps(c)=0$ and hence $\eps (c)=0$.
\end{proof}

\section{Modification}
\label{section:modification}

Our key assumption on a base weight for a regular multiplier Hopf algebroid, the counitality condition $\mu_{B} \circ \beps = \mu_{C}\circ \epsc$, turned out to be quite restrictive in several examples considered in subsection \ref{subsection:measured}, where it corresponded to invariance of $\mu_{B}$  and $\mu_{C}$ under certain actions of an underlying groupoid or a multiplier Hopf algebra. In case that the functionals are only quasi-invariant with respect to these actions, one can independently modify the left and the right comultiplication and accordingly the left and right counits  such that $\mu_{B}$ and $\mu_{C}$ will form a base weight for this modified  multiplier Hopf algebroid. 
We now describe this modification, which  is inspired by  \cite{daele:modified}, in  a systematic manner, starting with left and right multiplier bialgebroids and then turning to multiplier Hopf algebroids. Examples will be given in the next section.   

\subsection{Modification of left multiplier bialgebroids}

Let $\mathcal{A}_{B}=(A,B,s,t,\Delta)$ be a  left multiplier bialgebroid.

Given an automorphism $\theta$ of $B$, we write $_{\theta}A$ and $A_{\theta}$  when we regard $A$ as a left or right $B$-module via $x\cdot a:=s(\theta(x))a$ or  $a\cdot x:=as(\theta(x))$, respectively, and  $^{\theta}A$ and $A^{\theta}$  when $x\cdot a:=at(\theta^{-1}(x))$ or $a\cdot x:=t(\theta^{-1}(x))a$ for all $a\in A$ and $x\in B$.  

Note that then the quotients ${_{\theta}A} \otimes A^{B}$ and ${_{B}A} \otimes A^{\theta}$ of $A\otimes A$ coincide.

Suppose that $\Theta_{\lambda}$ and $\Theta_{\rho}$ are automorphisms of $A$ which, when extended to multipliers, satisfy 
\begin{align}
  \label{eq:theta-base}
\Theta_{\lambda}
\circ s&=s \circ \theta, & \Theta_{\rho} \circ t &= t\circ \theta^{-1}.
\end{align}
 Then the  isomorphisms
\begin{align} \label{eq:theta-isomorphisms}
  \Theta_{\lambda} \otimes \id, \, \id \otimes  \Theta_{\rho} &\colon {_{B}A} \otimes A^{B} \to {_{\theta}A} \otimes A^{B} = \bsA \otimes A^{\theta}
\end{align}
induce, by conjugation,  isomorphisms
\begin{align*}
    \Theta_{\lambda} \overline{\times}\id, \,  \id
    \overline{\times} \Theta_{\rho} &\colon \End({_{B}A} \otimes
    A^{B}) \to \End({_{\theta}A} \otimes A^{B})= \End({_{B}A} \otimes A^{\theta}).
\end{align*}
\begin{definition} \label{definition:left-modifier}
A \emph{modifier} of a left  multiplier bialgebroid    $\mathcal{A}_{B}=(A,B,s,t,\Delta)$ consists  of an automorphism $\theta$ of $B$ and automorphisms $\Theta_{\lambda}$, $\Theta_{\rho}$ of $A$ satisfying \eqref{eq:theta-base} and
\begin{align} \label{eq:theta-comultiplication}
 (\Theta_{\lambda} \overline{\times}\id) \circ
  \Delta = ( \id \overline{\times} \Theta_{\rho}) \circ \Delta.
\end{align}
\end{definition}
Evidently, such modifiers form a  group with respect to the composition given by
\begin{align*}
  (\theta,\Theta_{\lambda},\Theta_{\rho})(\theta',\Theta'_{\lambda},\Theta'_{\rho}) = (\theta\theta',\Theta_{\lambda}\Theta'_{\lambda},\Theta'_{\rho}\Theta_{\rho}).
\end{align*}

\begin{proposition} \label{prop:left-modification}
Let $\mathcal{A}_{B}=(A,B,s,t,\Delta)$ be a  left multiplier bialgebroid with a modifier $(\theta,\Theta_{\lambda},\Theta_{\rho})$. Denote by $\tilde \Delta$ the composition in \eqref{eq:theta-comultiplication}. Then 
$\tilde{\mathcal{A}}_{B}:=(A,B,s,t\circ \theta^{-1},\tilde \Delta)$
 is a left multiplier bialgebroid with canonical maps $(\tilde T_{\lambda},\tilde T_{\rho})$   given by
  \begin{align}\label{eq:theta-tl}
    \tilde T_{\lambda} &= (\id \otimes \Theta_{\rho})\circ T_{\lambda} = (\Theta_{\lambda} \otimes \id) \circ T_{\lambda}  \circ (\Theta_{\lambda} \otimes \id)^{-1}, \\ \label{eq:theta-tr}
    \tilde T_{\rho} &= (\Theta_{\lambda} \otimes \id) \circ T_{\rho} = (\id \otimes \Theta_{\rho}) \circ T_{\rho} \circ (\id \otimes \Theta_{\rho})^{-1}.
  \end{align}
\end{proposition}
\begin{proof}
To see that  the space $\bsA \otimes A^{\theta}= {_{\theta}A} \otimes A^{B}$ is non-degenerate as a right module over $A\otimes 1$ and $1 \otimes A$, use the isomorphisms \eqref{eq:theta-isomorphisms}.  To check that $\tilde \Delta$ is bilinear with respect to $s$ and $t$ and co-associative is straightforward,  for example, 
\begin{align*}
\tilde\Delta(s(x)as(x')) &= (\Theta_{\lambda} \bar\times\id)((1\otimes s(x))\Delta(a)(1\otimes s(x')))
= (1\otimes s(x))\tilde\Delta(a)(1\otimes s(x'))
\end{align*}
 for all $x,x'\in B$ and $a\in A$, and
 \begin{align*}
(\tilde \Delta \otimes \iota)(\tilde \Delta(b)(1 \otimes c)) (a\otimes 1 \otimes 1) &=
(\Theta_{\lambda} \otimes \id \otimes \Theta_{\rho})(( \Delta \otimes \iota)( \Delta(b)(1 \otimes c)) (a\otimes 1 \otimes 1)) \\
&=
(\Theta_{\lambda} \otimes \id \otimes \Theta_{\rho})((\iota\otimes \Delta)(\Delta(b)(a \otimes 1))(1 \otimes 1 \otimes c))
 \\
&=(\iota\otimes  \tilde\Delta)(\tilde\Delta(b)(a \otimes 1))(1 \otimes 1 \otimes c)
 \end{align*}
for all $a,b,c\in A$.  The formulas \eqref{eq:theta-tl} and  \eqref{eq:theta-tr} follow from the definition of $\tilde \Delta$.
\end{proof}
We call $\tilde{\mathcal{A}}_{B}$ the \emph{modified} left multiplier bialgebroid or briefly \emph{modification} associated to $(\theta,\Theta_{\lambda},\Theta_{\rho})$.

\begin{remark}
In the situation above,
 the map
 \begin{align*}
   (\theta',\Theta'_{\lambda},\Theta'_{\rho}) \mapsto (\theta'\theta^{-1},\Theta_{\lambda}'\Theta_{\lambda}^{-1},\Theta'_{\rho}\Theta_{\rho}^{-1})
 \end{align*}
 is a bijection between all modifiers of $\mathcal{A}_{B}$ and all modifiers of $\tilde{\mathcal{A}}_{B}$, as one can easily check.
\end{remark}

\begin{lemma} \label{lemma:modified-full}
  Let $(\theta,\Theta_{\lambda},\Theta_{\rho})$ be a modifier of a full left multiplier bialgebroid $\mathcal{A}_{B}$. Then
  \begin{align*}
\Theta_{\lambda} \circ t&=t, & \Theta_{\rho} \circ  s &= s, &
    \Delta \circ \Theta_{\lambda} &= (\id \overline{\times}
    \Theta_{\lambda}) \circ \Delta, & \Delta \circ \Theta_{\rho} &=
    (\Theta_{\rho} \overline{\times} \id) \circ \Delta.
  \end{align*}
\end{lemma}
\begin{proof}
 We only prove the relations involving $\Theta_{\lambda}$; similar arguments apply to $\Theta_{\rho}$.

 Formula \eqref{eq:theta-tr} implies
  \begin{align*}
    (\Theta_{\lambda}\otimes \id)((t(x) \otimes 1)T_{\rho}(a\otimes b)) &= 
    (\Theta_{\lambda}\otimes \id)(T_{\rho}(t(x)a\otimes b)) \\ &= (t(x)\otimes 1)     (\Theta_{\lambda}\otimes \id)(T_{\rho}(a\otimes b))
  \end{align*}
  for all $a,b\in A$ and $x\in B$.  Applying slice maps, we find that $t(x)\Theta_{\lambda}(c)=\Theta_{\lambda}(t(x)c)$ for all elements $c \in A$ of the  form $(\id \otimes \omega)(T_{\rho}(a\otimes b))$, where $a,b\in A$  and $\omega \in \Hom(A^{B},B_{B})$. Since $\mathcal{A}_{B}$ is full, such  elements span $A$ and hence $\Theta_{\lambda} \circ t = t$.

  Next, the formula for $\tilde T_{\rho}$ and the second  diagram in \eqref{dg:left-galois-1} show that the   outer cell and the left cell in the following diagram commute:
  \begin{align*}
    \xymatrix@C=30pt@R=18pt{
      \btA \otimes {^{B}A_{B'}} \otimes {_{B'}A}
       \ar[r]^{\id \otimes \Tr} \ar[d]_{T_{\lambda} \otimes \id} &
       A^{B} \otimes {^{B}_{B'}A} \otimes A^{B'}
       \ar[r]^{\id \otimes \Theta_{\lambda} \otimes \id} \ar[d]_{T_{\lambda} \otimes \id} &  \ar[d]^{T_{\lambda} \otimes \id} 
       A^{B} \otimes {^{\theta}_{\theta'}A} \otimes A^{B'}
       \\
       {_{B}A} \otimes A^{B}_{B'} \otimes {_{B'}A}
       \ar[r]_{\id \otimes \Tr}  & 
       {_{B}A} \otimes {_{B'}A^{B}} \otimes A^{B'}
       \ar[r]_{\id \otimes \Theta_{\lambda} \otimes \id}&
       {_{B}A} \otimes {_{\theta'}A^{\theta}} \otimes A^{B'}
    }
  \end{align*}
Here, we use the notation explained in Notation \ref{notation:tensor-products}.
  We apply slice maps of the form $\id \otimes \id \otimes \omega$,  where $\omega \in \Hom(A^{B},B_{B})$, use the assumption that  $\mathcal{A}_{B}$ is full, and conclude that $(\id \otimes \Theta_{\lambda}) T_{\lambda} = T_{\lambda}(\id \otimes \Theta_{\lambda})$ and   hence $\Delta \circ \Theta_{\lambda} = (\id\overline{\times}\Theta_{\lambda})\Delta$.
\end{proof}
The preceding result implies that in the full case, the modified left multiplier bialgebroid is isomorphic to the original one in the following sense.
\begin{definition}
An \emph{isomorphism} between   left  multiplier bialgebroids
\begin{align*}
  \mathcal{A}_{1}=(A_{1},B_{1},s_{1},t_{1},\Delta_{1}) \quad \text{and} \quad \mathcal{A}_{2}=(A_{2},B_{2},s_{2},t_{2},\Delta_{2}) 
\end{align*}
is a pair of isomorphisms $\Theta\colon A_{1} \to A_{2}$ and $\theta\colon B_{1}\to B_{2}$ such that for all $a,b,c\in A$,
\begin{align*}
\Theta\circ s_{1} &= s_{2} \circ \theta, &  \Theta \circ t_{1} &= t_{2} \circ \theta, &
  (\Theta \otimes \Theta)(\Delta_{1}(a)(b\otimes c)) &= \Delta_{2}(\Theta(a))(\Theta(b)\otimes \Theta(c)).
\end{align*}
\end{definition}
\begin{proposition} \label{proposition:modified-isomorphism}
Let $\mathcal{A}_{B}$ be a full  left multiplier bialgebroid with a modifier $(\theta,\Theta_{\lambda},\Theta_{\rho})$. Then     $(\Theta_{\lambda},\theta)$ and $(\Theta_{\rho},\id)$ are isomorphisms from $\mathcal{A}_{B}$ to the modification $\tilde{\mathcal{A}}_{B}$.
\end{proposition}
\begin{proof}
We only prove the assertion for $(\Theta_{\lambda},\theta)$. 
Lemma \ref{lemma:modified-full}  and the definition of $\tilde \Delta$ imply
that  for all $a,b,c\in A$,
  \begin{align*}
    \tilde \Delta(\Theta_{\lambda}(a))(\Theta_{\lambda}(b)\otimes \Theta_{\lambda}(c)) &=
    (\Theta_{\lambda} \otimes \id)(\Delta(\Theta_{\lambda}(a))(b \otimes \Theta_{\lambda}(c))) \\
    &= (\Theta_{\lambda} \otimes \Theta_{\lambda})(\Delta(a)(b\otimes c)). \qedhere
  \end{align*}
\end{proof}
 As in subsection \ref{subsection:modular-automorphism}, we consider  convolution operators 
$\lambda(\bupsilon),  \rho(\omega^{B}) \colon A \to L(A)$ associated to module maps $\bupsilon \in \Hom(\bA,\bB)$ and $\omega^{B} \in \Hom(A^{B},B_{B})$ by the formulas
\begin{align*}
 \lambda(\bupsilon)(a)b &:= (\bupsilon \otimes \id)(\Delta(a)(1 \otimes b)), &
 \rho(\omega^{B})(a)b &:= (\id \otimes \omega^{B})(\Delta(a)(b \otimes 1)).
\end{align*}
\begin{proposition} \label{proposition:modified-counits}
  Let $\mathcal{A}_{B}=(A,B,s,t,\Delta_{B})$ be a  left multiplier bialgebroid with a  left counit $\beps$ and a modifier $(\theta,\Theta_{\lambda},\Theta_{\rho})$.
  \begin{enumerate}
  \item If $\mathcal{A}_{B}$ is full, then $\theta \circ \beps \circ
    \Theta^{-1}_{\lambda} = \beps \circ \Theta_{\rho}^{-1}$, and this
    composition, denoted by $_{B}\tilde\varepsilon$, is the unique left counit of
    the modified left multiplier bialgebroid $\tilde{\mathcal{A}}_{B}$.
  \item If $\tilde{\mathcal{A}}_{B}$ has a left counit $_{B}\tilde{\varepsilon}$, then
    $\lambda(_{B}\tilde\varepsilon) = \Theta_{\rho}^{-1}$ and $
\rho(\theta\circ {_{B} \tilde\epsilon}) =
      \Theta_{\lambda}^{-1}$,
    where the convolution operators are formed with respect to
    $\Delta$.
  \end{enumerate}
\end{proposition}
\begin{proof}
(1) The isomorphisms  in  Proposition \ref{proposition:modified-isomorphism} yield two left counits $\beps \circ \Theta_{\rho}^{-1}$ and $\theta \circ \beps \circ \Theta_{\lambda}^{-1}$ of $\tilde{\mathcal{A}}$, which necessarily coincide by \cite[ Proposition 3.5]{timmermann:regular} because $\tilde{A}_{B}$ is full.  

(2) We only prove the first equation. The following diagram 
commutes,
 \begin{align*}
\xymatrix@C=30pt@R=15pt{
  \Ab\otimes \bA \ar[r]^{T_{\rho}} \ar[d]_{\id \otimes \Theta_{\rho}} & \bsA \otimes A^{B} \ar[d]_{\id \otimes \Theta_{\rho}} \ar[r]^(0.6){_{B}\tilde\varepsilon \otimes \id} & A \ar[d]^{\Theta_{\rho}} \\
  \Ab \otimes \bA \ar[r]_{\tilde T_{\rho}} & \bsA \otimes A^{\theta} \ar[r]_(0.6){_{B}\tilde\varepsilon \otimes \id} &A,
} & 
\end{align*}
 and shows that
$\Theta_{\rho}( \lambda(_{B}\tilde\varepsilon)(a)b) = a\Theta_{\rho}(b)$
for all $a,b\in A$.
\end{proof}

Suppose that  $\mathcal{A}_{B}$ is \emph{unital} in the sense that 
the algebras $A,B$ and the maps $s,t,\Delta_{B}$ are unital. Then modifiers have a nice description in terms of  the maps from $A$ to $ B$ considered above. 

Consider the convolution product on the space $\Hom(_{B}A^{B},\bB_{B})$,  given by
\begin{align*}
  \bupsilon^{B}\ast \bomega^{B} &:=   (\bupsilon^{B} \otimes \bomega^{B}) \circ \Delta \colon a \mapsto \sum \bomega^{B}(a_{(2)})\bupsilon^{B}(a_{(1)}),
\end{align*}
where $\bupsilon^{B},\bomega^{B} \in \Hom(\bsA^{B},\bB_{B})$ and $\AltkA$ is identified with the Takeuchi product inside $\bsA \otimes \btA$. If it exists, then the left counit $\beps$ is the unit for this product.
We call an element $_{B}\chi^{B} \in \Hom(_{B}A^{B},\bB_{B})$ a \emph{character} if for all $a,b\in A$,
\begin{align*}
  {_{B}\chi^{B}}(ab) = {_{B}\chi^{B}}(as(_{B}\chi^{B}(b))) = {_{B}\chi^{B}}(at(_{B}\chi^{B}(b))).
\end{align*}
\begin{lemma}
  Let $\mathcal{A}_{B}=(A,B,s,t,\Delta)$ be a unital full left multiplier bialgebroid with a left counit. Then there exist canonical bijections between
  \begin{enumerate}
  \item  all modifiers $(\id,\Theta_{\lambda},\Theta_{\rho})$ for $\mathcal{A}_{B}$;
  \item  all automorphisms $\Theta_{\lambda}$ of $A$ satisfying $\Theta_{\lambda}\circ s=s$, $\Theta_{\lambda}\circ t = t$, $\Delta \circ \Theta_{\lambda} = (\id \overline{\times} \Theta_{\lambda}) \circ \Delta$;
  \item  all automorphisms $\Theta_{\rho}$ of $A$ satisfying $\Theta_{\rho} \circ t=t$, $\Theta_{\rho} \circ s=s$, $\Delta \circ \Theta_{\rho} = (\Theta_{\rho} \overline{\times} \id)$;
  \item all invertible characters $_{B}\chi^{B}  \in \Hom({_{B}A}^{B},_{B}B_{B})$.
  \end{enumerate}
\end{lemma}
\begin{proof}
For every modifier $(\id,\Theta_{\lambda},\Theta_{\rho})$, the automorphisms $\Theta_{\lambda}$ and $\Theta_{\rho}$ satisfy the conditions in (2) and (3) by Lemma \ref{lemma:modified-full}.

Assume that $\Theta_{\lambda}$ is an automorphism as in (2) and denote by $\beps$ the counit of $\mathcal{A}_{B}$.  Then  the map ${_B\chi^B}:=\beps \circ \Theta_{\lambda} \colon A \to B$ lies in $\Hom(\bsA^{B},\bB_{B})$ because of \eqref{eq:left-counit-bimodule}, is a character because of \cite[ Proposition 3.5]{timmermann:regular}, and the composition  ${_B\bar\chi^B}:=  \beps \circ \Theta^{-1}_{\lambda}$ is its convolution inverse because
 \begin{align*}
   {_B\chi^B} \ast {_B\bar\chi^B} =  {_B\chi^B}\circ (\id \otimes {_B\bar\chi^B}) \circ \Delta = {_B\chi^B} \circ (\id \otimes  \beps)  \circ \Delta \circ \Theta^{-1}_{\lambda} = {_B\chi^B} \circ \Theta^{-1}_{\lambda} = \beps
 \end{align*}
and similarly ${_B\bar\chi^B} \ast {_B\chi^B} = \beps$. 

Similar arguments show that for every automorphism $\Theta_{\rho}$ as in (3), the map $\beps \circ \Theta_{\rho}$ is an invertible character.

Finally, assume that ${_B\chi^B}$ is an invertible character as in (4). Then the  maps $\Theta_{\lambda}:=\rho({_B\chi^B})$ and $\Theta_{\rho}:=\lambda({_B\chi^B})$ are bijections of $A$ because ${_B\chi^B}$ is invertible in the convolution algebra, and they are automorphisms because ${_B\chi^B}$  is a character. Using co-associativity, one easily verifies that $(\id,\Theta_{\lambda},\Theta_{\rho})$ is a modifier.
\end{proof}

Right multiplier bialgebroids can be modified similarly. 
\begin{definition}
  A \emph{modifier} of a right multiplier bialgebroid $\mathcal{A}_{C}=(A,C,s,t,\Delta)$ consists of an automorphism $\theta$ of $C$ and automorphisms ${_{\lambda}\Theta}$ and ${_{\rho}\Theta}$ of $A$ satisfying
  \begin{align*}
    {_{\lambda}\Theta} \circ t &= t\circ \theta^{-1}, &
    {_{\rho}\Theta} \circ s &= s\circ \theta, &
    ({_{\lambda}\Theta} \overline{\times} \id) \circ \Delta &=
    (\id \overline{\times} {_{\rho}\Theta})\circ \Delta.
  \end{align*}
\end{definition}
The results obtained above carry over to right multiplier bialgebroids  in a straightforward way.
In particular, for every modifier $(\theta,{_{\lambda}\Theta},{_{\rho}\Theta})$ of a right multiplier bialgebroid $\mathcal{A}_{C}=(A,C,s,t,\Delta)$, we obtain a modified right multiplier bialgebroid
\begin{align*}
  \tilde{\mathcal{A}}_{C}:=(A,C,s,t\circ \theta^{-1},\tilde \Delta), \quad
  \text{where } \tilde \Delta=({_{\lambda}\Theta} \overline{\times}
  \id)\circ \Delta = (\id \overline{\times} {_{\rho}\Theta})\circ
  \Delta.
\end{align*}
 If $\mathcal{A}_{C}$ is full, then we have isomorphisms $(\Theta_{\lambda},\id)$ and $(\Theta_{\rho},\theta)$ from $\mathcal{A}_{C}$ to $\tilde{\mathcal{A}}_{C}$, where the notion of an isomorphism between right multiplier bialgebroids is evident. 

\subsection{Modification of regular multiplier Hopf algebroids} 

 Let us now consider the two-sided case.
 \begin{definition}
   A \emph{modifier} for a multiplier bialgebroid $\mathcal{A}=(A,B,C,S_{B},S_{C},\Delta_{B},\Delta_{C})$ is a tuple $(\Theta_{\lambda},\Theta_{\rho},{_{\lambda}\Theta},{_{\rho}\Theta})$ such that
   \begin{enumerate}
   \item $\Theta_{\lambda}$ extends to an automorphism of $B$ and
 $(\Theta_{\lambda}|_{B},\Theta_{\lambda},\Theta_{\rho})$ is a modifier of the  associated left   multiplier bialgebroid  $\mathcal{A}_{B}$,
\item $_{\rho}\Theta$ extends to an automorphism of $C$ and
$({_{\rho}\Theta|_{C}},{_{\lambda}\Theta,{_{\rho}\Theta}})$ is a
     modifier of the associated  right multiplier bialgebroid $\mathcal{A}_{C}$.
   \end{enumerate}
We call such a modifier
 \emph{trivial on the base} if all four automorphisms
  act trivially on $B$ and on $C$. We call it  \emph{self-adjoint} if
  $\mathcal{A}$ is a multiplier $*$-bialgebroid and
  \begin{align*}
    * \circ
  \Theta_{\lambda} = {_{\lambda}\Theta} \circ * \quad\text{and}\quad * \circ   \Theta_{\rho} = {_{\rho}\Theta} \circ \ast. 
  \end{align*}
 \end{definition}
Note that every  modifier as above satisfies
\begin{align}
  \label{eq:modifier-antipode-base}
  \Theta_{\rho} \circ S_{B} \circ \Theta_{\lambda} &= S_{B} &&\text{and} & {_{\lambda}\Theta} \circ S_{C} \circ {_{\rho}\Theta} &= S_{C},
\end{align}
 and that modifiers form a group with respect to the composition
\begin{align*}
  (\Theta_{\lambda},\Theta_{\rho},{_{\lambda}\Theta},{_{\rho}\Theta}) \cdot (\Theta_{\lambda}',\Theta_{\rho}',{_{\lambda}\Theta}',{_{\rho}\Theta}') &:=
(\Theta_{\lambda}\Theta_{\lambda}',\Theta_{\rho}'\Theta_{\rho},{_{\lambda}\Theta}'{_{\lambda}\Theta},{_{\rho}\Theta}{_{\rho}\Theta'}).
\end{align*}
 \begin{proposition} \label{prop:modification}
   Let  $\mathcal{A}=(A,B,C,S_{B},S_{C},\Delta_{B},\Delta_{C})$ be a multiplier bialgebroid with a modifier $(\Theta_{\lambda},\Theta_{\rho},{_{\lambda}\Theta},{_{\rho}\Theta})$ and 
 let
   \begin{align*}
     \tilde S_{B} &:= S_{B} \circ \Theta_{\lambda}|_{B}^{-1}, & \tilde S_{C} &:= S_{C} \circ {_{\rho}\Theta}  |_{C}^{-1}, &
     \tilde \Delta_{B} &:= (\Theta_{\lambda} \overline{\times}\id)\circ \Delta_{B}, &
     \tilde \Delta_{C} &:= ({_{\lambda}\Theta} \overline{\times}\id) \circ \Delta_{C}.
   \end{align*}
   Then $\tilde{\mathcal{A}}=(A,B,C, \tilde S_{B},\tilde S_{C},\tilde \Delta_{B},\tilde \Delta_{C})$ is a  multiplier bialgebroid.
If moreover $\mathcal{A}$ is a multiplier $*$-bialgebroid and the modifier is self-adjoint, then also $\tilde{\mathcal{A}}$ is a  multiplier $*$-bialgebroid.
 \end{proposition}
 \begin{proof}
   By Proposition \ref{prop:left-modification} and its right-handed analogue, $\tilde{\mathcal{A}}_{B}:=(A,B,\id_{B},\tilde S_{B},\tilde \Delta_{B})$ is a left and
   $\tilde{\mathcal{A}}_{C}:=(A,C,\id_{C},\tilde S_{C},\tilde\Delta_{C})$ is a right multiplier bialgebroid. The mixed coassociativity relations follow by similar arguments as the coassociativity of $\tilde \Delta_{B}$, see the proof of  Proposition \ref{prop:left-modification}. Thus, $\tilde{\mathcal{A}}$ is a multiplier bialgebroid.

Assume that $\mathcal{A}$ is a multiplier $*$-bialgebroid and that the modifier is self-adjoint. Then
\begin{align*}
  \tilde S_{B} \circ \ast \circ \tilde S_{C} \circ \ast &=  S_{B} \circ \Theta_{\lambda}^{-1} \circ \ast \circ S_{C} \circ \Theta_{\rho}^{-1} \circ \ast \\ & = S_{B} \circ \ast \circ {_{\lambda}\Theta}^{-1} \circ S_{C} \circ \Theta_{\rho}^{-1} \circ \ast = S_{B} \circ \ast \circ S_{C} \circ \ast = \id_{C}
\end{align*}
and similarly $\tilde S_{C} \circ \ast \circ \tilde S_{C} \circ \ast =\id_{B}$. Finally, self-adjointness of the modifier immediately  implies that $\tilde \Delta_{B}$ and $\tilde \Delta_{C}$ satisfy condition (3) in Definition  \ref{definition:involution}.
 \end{proof}
In the situation above, we call $\tilde{\mathcal{A}}$ the \emph{modification} of $\mathcal{A}$.

We can now formulate the main result of this section:
 \begin{theorem}\label{theorem:modification}
Let $\mathcal{A}=(A,B,C,S_{B},S_{C},\Delta_{B},\Delta_{C})$ be a regular multiplier Hopf algebroid with a modifier $(\Theta_{\lambda},\Theta_{\rho},{_{\lambda}\Theta},{_{\rho}\Theta})$. Then  the modification $\tilde {\mathcal{A}}$ is  a regular multiplier Hopf algebroid again. The counits and antipode $\beps,\epsc,S$ of $\mathcal{A}$  and the counits and antipode  ${_{B}\tilde \epsilon}, \tilde \epsilon_{C},\tilde S$  of $\tilde{\mathcal{A}}$ are related by 
\begin{align}
    _{B}\tilde \epsilon &= \beps \circ \Theta^{-1}_{\rho} =
    \Theta_{\lambda}\circ \beps \circ \Theta^{-1}_{\lambda}, \label{eq:modified-lt-counit} \\ \tilde
    \epsilon_{C} &= \epsc \circ {_{\lambda}\Theta}^{-1}
    ={_{\rho}\Theta} \circ \epsc \circ {_{\rho}\Theta}^{-1}, \label{eq:modified-rt-counit}\\
  \tilde S &= \Theta_{\rho} \circ S \circ {_{\rho}\Theta}^{-1} = {_{\lambda}\Theta}\circ S \circ  \Theta_{\lambda}^{-1}. \label{eq:modified-antipode}
\end{align}
If $\mathcal{A}$ is a multiplier Hopf $*$-algebroid and the modifier is self-adjoint, then also $\tilde{\mathcal{A}}$ is a  multiplier Hopf $*$-algebroid.
 \end{theorem}
 \begin{proof}
   The  canonical maps of the     multiplier bialgebroid  $\tilde{\mathcal{A}}$ are     bijective by construction, see Proposition \ref{prop:left-modification}, and the formulas for the counits follow from Proposition \ref{proposition:modified-counits} and its right-handed analogue. To prove \eqref{eq:modified-antipode}, consider the following diagrams.
    \begin{align*}
     \xymatrix@R=20pt@C=-12pt{ \Ab \otimes \bA 
       \ar[rrrr]^{S_{C}\epsc \otimes \id} \ar[ddd]_{T_{\rho}} \ar[rd]^{\id \otimes \Theta^{-1}_{\rho}} &&\qquad\qquad&& A  &
\qquad &  \cA \otimes \Ac 
       \ar[rrrr]^{\tilde S \otimes \id} \ar[ddd]_{\tilde T_{\lambda}} \ar[rd]^{\Theta_{\lambda}^{-1} \otimes \id} &&\qquad\qquad && A^{\theta_{C}} \otimes \Ac  
\\
       & \Ab \otimes \bA \ar[rr]^(0.55){S_{C}\epsc \otimes \id}  \ar[d]_{T_{\rho}} && \ar[ru]^{\Theta_{\rho}} A  &  &&
       & \cA \otimes \Ac \ar[rr]^{S \otimes \id}  \ar[d]_{T_{\lambda}} && \ar[ru]^{{_{\lambda}\Theta} \otimes \id} A^{C} \otimes  \Ac & \\
       & \bA \otimes A^{B} \ar[rr]_{S\otimes \id} \ar[ld]^{\Theta_{\lambda}^{-1}\otimes \Theta_{\rho}} && \ar[rd]_{\Theta_{\rho} \otimes \Theta_{\rho}} \ar[u]_{m} {^{B}A} \otimes  A^{B} & && 
       & \bA \otimes A^{B} \ar[rr]_{\Sigma T_{\rho}^{-1}} \ar[ld]^{\Theta_{\lambda} \otimes \id} &&  \ar[u]_{\rT} \ar[rd]_{\id \otimes \id} \bA  \otimes \Ab & 
\\
\bA \otimes A^{B} \ar[rrrr]_{S \otimes \id} &&&& {^{B}A} \otimes A^{B}\ar[uuu]_{m} && \bA \otimes A^{\theta_{B}} \ar[rrrr]_{\Sigma \tilde T_{\rho}^{-1} } &&&& \bA \otimes \Ab \ar[uuu]_{_{\rho}\tilde T}
     }
   \end{align*}
   In the diagram on the left hand side, the outer and inner square commute by the defining property of the antipode, and since the left, the right and the upper cells commute,  so does the lower one, showing that $\Theta_{\rho} S = \Theta_{\lambda}^{-1}S$.  In the diagram on the right hand side, the outer and inner square commute by   \cite[Proposition 5.8]{timmermann:regular}. Since the left, the right and the lower cell commute as well, so does the upper one, showing that $\tilde S = {_{\lambda}\Theta} S \Theta_{\lambda}^{-1}$.
 \end{proof}
The following example is the counterpart to  Van Daele's modification \cite{daele:modified}:
 \begin{example}
Let $\mathcal{A}=(A,B,C,S_{B},S_{C},\Delta_{B},\Delta_{C})$ be a regular multiplier Hopf algebroid and let $u,v\in B$ be invertible. Then the formulas
\begin{align*}
  \Theta_{\lambda}(a) &=v^{-1}av, & \Theta_{\rho}(a) &= S_{B}(v^{-1})aS_{B}(v), &  {_{\lambda}\Theta}(a) &= uau^{-1}, & {_{\rho}\Theta}(a) &= S_{C}^{-1}(u)aS_{C}^{-1}(u^{-1})
\end{align*}
define a modifier $(\Theta_{\lambda},\Theta_{\rho},{_{\lambda}\Theta},{_{\rho}\Theta})$ and the counits and antipode of the associated modification $\tilde{\mathcal{A}}$ are given by
\begin{align*}
  _{B}\tilde\varepsilon(a) &=\beps(av^{-1})v, & \tilde\varepsilon_{C}(a) &= S_{C}^{-1}(u^{-1})\epsc(ua), & \tilde S(a) &= uS(vav^{-1})u^{-1},
\end{align*}
compare with the formulas given in  \cite[Proposition 1.12 and 1.13]{daele:modified} and \eqref{eq:wmha-source-target}.
 \end{example}

Partial integrals do not change when a multiplier bialgebroid is modified and the modifier is trival on the base:
\begin{lemma} \label{lemma:modification-integrals}
Let $\mathcal{A}=(A,B,C,S_{B},S_{C},\Delta_{B},\Delta_{C})$ be a regular multiplier Hopf algebroid with a modifier $(\Theta_{\lambda},\Theta_{\rho},{_{\lambda}\Theta},{_{\rho}\Theta})$ that is trivial on the base. Then partial integrals of $\mathcal{A}$ and of  the modification $\tilde {\mathcal{A}}$ coincide.
\end{lemma}
\begin{proof}
  Straightforward and left to the reader.
\end{proof}

As outlined above and illustrated in the next section,  some naturally appearing multiplier Hopf algebroids admit a counital base weight only after modification. We now investigate when such a modification exists.

Similarly as before, we consider for a map $^{C}\chi_{C} \in \Hom(^{C}A_{C},\cC_{C})$ the  convolution operators $\lambda(^{C}\chi_{C}),\rho(^{C}\chi_{C}) \colon A \to R(A)$ defined by
\begin{align*}
a\lambda(^{C}\chi_{C})(b)   &= ( {^{C}\chi_{C}} \otimes \id)((1\otimes a)\Delta_{C}(b)), &
a\rho(^{C}\chi_{C})(b)  &= (\id \otimes {^{C}\chi_{C}})((a\otimes 1)\Delta_{C}(b)).
\end{align*}
 \begin{proposition}
   Let $\mathcal{A}=(A,B,C,S_{B},S_{C},\Delta_{B},\Delta_{C})$ be a regular multiplier Hopf algebroid with an antipodal base weight $(\mu_{B},\mu_{C})$.  The following conditions are equivalent:
   \begin{enumerate}
   \item there exists a modifier $(\Theta_{\lambda},\Theta_{\rho},{_{\lambda}\Theta},{_{\rho}\Theta})$, trivial on the base, such that $(\mu_{B},\mu_{C})$ is a counital base weight for the associated modification $\tilde{\mathcal{A}}$;
   \item there exists a functional $\chi \in A^{\sqcup}$   such that
     \begin{enumerate}
     \item  $_{B}\chi \in \Hom(\bA^{B},\bB_{B})$ and $\chi_{C} \in \Hom(^{C}A_{C},\cC_{C})$,
     \item  the maps
       $\rho(_{B}\chi)$, $\lambda({_{B}\chi})$, $
\rho(\chi_{C})$, $\lambda(\chi_{C})$ are automorphisms of $A$.
     \end{enumerate}
   \end{enumerate}
 \end{proposition}
 \begin{proof}
Suppose that $(\Theta_{\lambda},\Theta_{\rho},{_{\lambda}\Theta},{_{\rho}\Theta})$ is a modifier as in (1). Then the counit functional $\chi:=\tilde\epsilon$ of the associated modification $\tilde{\mathcal{A}}$  satisfies (2a) and (2b)   by  Proposition \ref{proposition:modified-counits} and by its right-handed analogue.

Conversely, suppose $\chi\in A^{\sqcup}$ satisfies the conditions in (2), and denote by $\Theta_{\lambda},\Theta_{\rho},{_{\lambda}\Theta},{_{\rho}\Theta}$ inverses of the convolution operators in  (2b). Using coassociativity, one easily verifies that these automorphisms form a modifier of $\mathcal{A}$.  By   Proposition \ref{proposition:modified-counits}, the left counit of the associated modification $\tilde{\mathcal{A}}$ satisfies $\lambda(_{B}\tilde\varepsilon) = \Theta_{\rho}^{-1} = \lambda(_{B}\chi)$, whence $_{B}\tilde\varepsilon = {_{B}\chi}$. Likewise, $\tilde\varepsilon_{C} = \chi_{C}$. Therefore,
$\mu_{B}   \circ {_{B}\tilde\varepsilon} = \chi = \mu_{C} \circ {_{C}\tilde\varepsilon}$.
 \end{proof}

\section{Examples of modified measured multiplier Hopf algebroids}

For several examples of  regular multiplier Hopf algebroids considered in the preceding sections, our assumptions on base weights translated into quite restrictive invariance conditions. We now show that if  weaker and quite natural quasi-invariance assumptions hold, then the original multiplier Hopf algebroid can be modified in such a way that the modification meets all of our assumptions and becomes a measured multiplier Hopf algebroid.

\subsection{The convolution algebra of an \'etale groupoid} \label{subsection:modification-convolution}
Let $G$ be a locally compact, \'etale Hausdorff groupoid, and
consider the multiplier Hopf $*$-algebroid
\begin{align*}
 \hat{\mathcal{A}} = (\hat A,\hat B,\hat C, \hat S_{\hat B},\hat S_{\hat C}, \hat \Delta_{\hat B},\hat \Delta_{\hat C})
\end{align*}
associated to the convolution algebra $\hat A =C_{c}(G)$ of $G$ as in Example \ref{example:hopf-groupoid-algebra}.

Let furthermore $\mu$ be a Radon measure on $G^{0}$ and consider the associated functional 
\begin{align*} 
\hat  \mu\colon C_{c}(G^{0}) \to \C, \quad f\mapsto \int_{G^{0}} f \intd \mu.
\end{align*}
We saw in Example \ref{example:measured-groupoid-algebra} that the base weight  $(\hat\mu,\hat\mu)$  for $\hat{\mathcal{A}}$  is counital if and only if $\mu$ is invariant. 
Suppose now that $\mu$ is only continuously quasi-invariant in the sense explained in Example \ref{example:quasi-groupoid-functions}, that is,
 the measures $\nu$ and $\nu^{-1}$ on $G$ defined by \eqref{eq:dnu} 
are related by a  continuous Radon-Nikodym derivative $D \in C(G)$ such that 
$\nu  =  D\nu^{-1}$.

 We can then modify $\hat{\mathcal{A}}$ and obtain  a measured multiplier Hopf $*$-algebroid as follows.

The Radon-Nikodym cocycle $D$ yields a one-parameter family of automorphisms
\begin{align*}
  \sigma_{t} \colon C_{c}(G) \to C_{c}(G), \quad (\sigma_{t}(f))(\gamma) = f(\gamma)D^{t}(\gamma),
\end{align*}
on the convolution algebra $C_{c}(G)$.  The automorphisms
\begin{align*}
  \Theta_{\lambda}&:=\Theta_{\rho}:=\sigma_{1/2} &&\text{and} & {_{\lambda}\Theta}&:={_{\rho}\Theta}:=\sigma_{-1/2}
\end{align*}
form a modifier of $\hat{\mathcal{A}}$  which is trivial on the base and self-adjoint.
  The   associated modification  is the multiplier Hopf $*$-algebroid 
\begin{align*}
\tilde{\mathcal{A}}=  (\hat A, \hat B,\hat C, \hat S_{\hat C},\hat S_{\hat C}, \tilde \Delta_{\hat B},\tilde\Delta_{\hat C}),
\end{align*}
where $\hat A = C_{c}(G)$ and $\hat B = \hat C = C_{c}(G^{0})$ with $\hat S_{\hat B} = \hat S_{\hat C} = \id_{C_{c}(G^{0})}$ as before, but
\begin{align*}
  (\tilde \Delta_{\hat B}(f)(g\otimes h))(\gamma', \gamma'') &= \sum_{s(\gamma)=r(\gamma')} f(\gamma)D^{1/2}(\gamma) g(\gamma^{-1}\gamma')h(\gamma^{-1}\gamma''), \\
  ((g\otimes h)\tilde \Delta_{\hat C}(f))(\gamma',\gamma'') &=
 \sum_{r(\gamma)=s(\gamma')} f(\gamma)D^{-1/2}(\gamma) g(\gamma'\gamma^{-1})h(\gamma''\gamma^{-1})
\end{align*}
for all $f,g,h\in C_{c}(G)$. Here, we use the isomorphisms \eqref{eq:groupoid-isomorphism-convolution}. Short calculations show that the antipode  $\tilde S$ remains unchanged, that is,
\begin{align*}
 (\tilde S(f))(\gamma) =  f(\gamma^{-1}), &
\end{align*}
 and   the counits $_{\hat B}\tilde\varepsilon$ and $\tilde\varepsilon_{\hat C}$  are given by
\begin{align*}
  (_{\hat B}\tilde \varepsilon (f))(u) &= \sum_{r(\gamma)=u} f(\gamma)D^{-1/2}(\gamma), &
  (\tilde\varepsilon_{\hat C}(f))(u) &= \sum_{s(\gamma)=u} f(\gamma)D^{1/2}(\gamma).
\end{align*}

For the modification $\tilde{\mathcal{A}}$, the base weight  $(\hat \mu,\hat \mu)$ is counital because
for all $f\in C_{c}(G)$,
\begin{align*}
  \hat\mu(_{\hat B}\tilde \varepsilon(f)) = \int_{G} f D^{-1/2} \intd \nu = \int_{G} fD^{1/2} \intd \nu^{-1} = \hat\mu(\tilde\varepsilon_{\hat C}(f)).
\end{align*}

 By Example \ref{example:integrals-groupoid-algebra} and Lemma \ref{lemma:modification-integrals}, the restriction map $C_{c}(G) \to C_{c}(G^{0})$ is left- and right-invariant with respect to the modified comultiplications, and the composition with $\hat \mu$ gives a total left and right integral
\begin{align*} 
  \hat \phi=\hat \psi \colon C_{c}(G) \to \C, \quad f \mapsto \int_{G^{0}} f|_{G^{0}} \intd \mu.
\end{align*}
We thus obtain a measured multiplier Hopf $*$-algebroid $(\hat{\mathcal{A}},\hat\mu,\hat\mu,\hat\phi,\hat\psi)$.

Since $\hat\phi\circ S = \hat \phi$, the modular element is trivial.  The modular automorphism of $\hat \phi$ is
  $\sigma_{1}$ because for all $f,g\in C_{c}(G)$,
\begin{align*}
   \hat \phi(f\ast g) = \int_{G} f(\gamma) g(\gamma^{-1}) \intd\nu(\gamma) = \int_{G} g(\gamma^{-1})f(\gamma)D(\gamma) \intd\nu^{-1}(\gamma) = \hat \phi(g\ast \sigma_{1}(f)).
\end{align*}

\subsection{Crossed products for symmetric actions on commutative algebras} \label{subsection:modification-crossed}
Let $C$ be a non-degenerate,  idempotent and commutative algebra with a left action of a regular multiplier Hopf algebra $(H,\Delta_{H})$ which is symmetric in the sense that \eqref{eq:ch-symmetric} holds, denote by $A=C\# H$ the associated crossed product and consider the regular multiplier Hopf algebroid $\mathcal{A}=(A,C,C,\id,\id,\Delta_{B},\Delta_{C})$ defined in Example \ref{example:hopf-ch}.
Suppose moreover that $(H,\Delta_{H})$ has a left and a right integral $\phi_{H}$ and $\psi_{H}$, and define a partial left integral $\cphic$ and a partial right integral    $\bpsib$ as in \eqref{eq:integrals-ch}. 
 We saw in Example \ref{example:measured-ch} that  for  every faithful $H$-invariant functional $\mu$, the tuple $(\mathcal{A},\mu,\mu,\bpsib,\cphic)$  is  a measured regular multiplier Hopf algebroid.
We now show that if $\mu$ only satisfies a weaker quasi-invariance condition,  then we can modify $\mathcal{A}$ so that we obtain a  measured regular multiplier Hopf algebroid  again. To simplify the discussion, we shall only consider the unital  case. 

We start with a few preliminaries on functionals that are quasi-invariant with respect to an action of a Hopf algebra. For the application in the next subsection, we drop   the commutativity assumption on $C$ and the symmetry assumption for a moment.
Recall that a Hopf algebra is regular if its antipode is invertible.

  Let $C$ be a unital algebra with a left action of a regular Hopf  algebra  $(H,\Delta_{H})$ so that  $C$ becomes a left $H$-module algebra. As before, we identify $C$ and $H$ with subalgebras of $C\# H$.

A \emph{unital one-cocycle} for $(H,\Delta_{H})$ with values in $C$ is a map $\omega\colon H\to C$ satisfying the following equivalent conditions:
\begin{enumerate}
\item $\omega(1_{H})=1_{C}$ and   $\omega(hg) = \omega(h_{(1)})(h_{(2)} \actright \omega(g))$ for all $h,g\in H$;
\item the map $\alpha_{\omega}\colon H \to C\# H$ given by $h\mapsto \omega(h_{(1)}) h_{(2)}$ is a unital homomorphism.
\end{enumerate}
We call
 a faithful   functional $\mu$ on $C$  \emph{quasi-invariant} with respect to $H$ if there exists a  map $D\colon H \to C$, $h\mapsto D_{h}$,  such that
\begin{align*}
  \mu(S(h)\actright y)  &= \mu(D_{h}y) \quad \text{for all }  h\in H, y\in C.
\end{align*}
 We then call $D$ the \emph{Radon-Nikodym cocycle} of $\mu$. This terminology is justified:
\begin{lemma}
If $\mu$ is a faithful and quasi-invariant functional on $C$, then its Radon-Nikodym cocycle  $D$ is a one-cocycle.
\end{lemma}
\begin{proof}
Let $h\in H$ and $y\in C$. Then  by definition,
  \begin{align} \label{eq:chb-aux-1} \mu(y'(S(h) \actright y)) =
    \mu(S(h_{(1)}) \actright ((h_{(2)} \actright y')y)) =
    \mu(D_{h_{(1)}}(h_{(2)} \actright y')y).
  \end{align}
  Taking $y'=D_{g}$, we find
  \begin{align*}
    \mu(D_{hg}y) = \mu(S(hg) \actright y) &= \mu (D_{g}(S(h) \actright
    y)) = \mu (D_{h_{(1)}}(h_{(2)} \actright D_{g})y). 
  \end{align*}
Since $\mu$ is faithful, the assertion follows. 
\end{proof}

 Regard
the left action of $H$ on  $C$  as a left action of the co-opposite Hopf algebra $H^{\co}$ on the opposite algebra
$C^{\op}$.  If $C$ is commutative and the action of $H$ is symmetric, then $C^{\op} \# H^{\co}$ is canonically isomorphic to $C\# H$.
\begin{lemma} \label{lemma:chb-modification}
  Let $\mu$ be a faithful, quasi-invariant functional on $C$ and 
 suppose that
 \begin{align}
   \label{eq:2}
   D_{h_{(1)}}(h_{(2)} \actright y) = (h_{(1)} \actright y)D_{h_{(2)}}
 \end{align}
 for all $h\in H$ and $y\in C$. Then there exist  automorphisms $\beta_{D}$ of $C\# H$ and $\beta^{\dag}_{D}$ of $C^{\op}\# H^{\co}$ such that
    \begin{align} \label{eq:ch-modifier}
      \beta_{D}(y\#h)&= yD_{h_{(1)}}\# h_{(2)}, &
        \beta^{\dag}_{D} (y\# h)&= D_{h_{(2)}}y \# h_{(1)}.
    \end{align}
\end{lemma}
\begin{proof}
 We only prove the assertion on $\beta_{D}$; the existence of $\beta^{\dag}_{D}$ follows similarly. 

 The formula for $\beta_{D}$ defines a homomorphism because the map $\alpha_{D} \colon h\mapsto D_{h_{(1)}}\# h_{(2)}$ is a homomorphism and
  \begin{align*}
    \beta_{D}(1\# h)\beta_{D}(y\# 1) &=
D_{h_{(1)}}(h_{(2)} \actright y) \# h_{(3)} = (h_{(1)} \actright y)D_{h_{(2)}} \# h_{(3)} = \beta_{D}((h_{(1)} \actright y) \# h_{(2)})
  \end{align*}
for all $h\in H$ and $y\in C$. It is bijective because the map 
\begin{align*}
  \bar\beta_{D} \colon C\# H \to C\# H, \quad
y\# h \mapsto  y (h_{(1)} \actright \omega(S_{H}(h_{(2)}))) \# h_{(3)},
\end{align*}
where $S_{H}$ denotes the antipode of $(H,\Delta_{H})$,
is  inverse to $\beta_{D}$.  Indeed, both are  $C$-linear on the left hand side and satisfy
\begin{align*}
\bar  \beta_{D}(\beta_{D}(1\# h)) &= D_{h_{(1)}}(h_{(2)} \actright D_{S_{H}(h_{(3)})}) \# h_{(4)} = D_{h_{(1)}S_{H}(h_{(2)})} \# h_{(3)} =  1 \# h, \\
  \beta_{D}(\bar \beta_{D}(1\# h)) &=
  (h_{(1)} \actright D_{S_{H}(h_{(2)})}) D_{h_{(3)}} \# h_{(4)}
\\ &= (h_{(1)} \actright (D_{S_{H}(h_{(3)})}(S_{H}(h_{(2)}) \actright D_{h_{(4)}}))) \# h_{(5)}  \\ &= (h_{(1)} \actright D_{S_{H}(h_{(2)})h_{(3)}}) \# h_{(4)}  = 1\# h. \qedhere
\end{align*}
\end{proof}

We now apply  the preceding considerations to our example. 
\begin{proposition}
  Let $C$ be a unital, commutative algebra with a left action of a regular Hopf algebra $(H,\Delta_{H})$ that is symmetric in the sense that for all $h\in H$ and $y\in C$,
  \begin{align*}
    h_{(1)} \otimes h_{(2)} \actright y = h_{(2)} \otimes h_{(1)} \actright y.
  \end{align*}
  Suppose that $\mu$ is a faithful, quasi-invariant functional on $C$, and that $\phi_{H}$ is an integral on $(H,\Delta_{H})$.
Define the  regular multiplier Hopf algebroid $\mathcal{A}=(A,B,C,S_{B},S_{C},\Delta_{B},\Delta_{C})$  as in Example \ref{example:hopf-ch}, the maps  $\beta_{D}$ and $\beta^{\dag}_{D}$ as in  \eqref{eq:ch-modifier}, and  $\cphic$  as in \eqref{eq:integrals-ch}. Then:
 \begin{enumerate}
 \item   $((\beta^{\dag}_{D})^{-1},(\beta_{D})^{-1},\id,\id)$ is a modifier of $\mathcal{A}$;
 \item $(\mu,\mu)$ is a counital base weight for the associated modification $\tilde{\mathcal{A}}$;
 \item   $(\tilde{\mathcal{A}},\mu,\mu,\cphic,\cphic)$ is a measured multiplier Hopf algebroid;
 \item the modular automorphism $\sigma^{\phi}$ of $\phi=\mu\circ \cphic$ is given by
   \begin{align*}
     \sigma^{\phi}(yh) =y\sigma_{H}(h_{(2)})D_{S^{-1}(h_{(1)})}
   \end{align*}
for all $y\in C$ and $h\in H$,
where $\sigma_{H}$ denotes the modular automorphism of $\phi_{H}$.
 \end{enumerate}
\end{proposition}
\begin{proof}
  (1) 
Since the action is symmetric, we can apply Lemma \ref{lemma:chb-modification} and conclude that $\beta^{\dag}_{D}$ and $\beta_{D}$ are automorphisms. They form a modifier of the left multiplier bialgebroid $\mathcal{A}_{B}$ because 
  \begin{align*}
    (\beta^{\dag}_{D} \overline{\times} \id)(\Delta_{B}(y\# h)) &= (yD_{h_{(2)}} \# h_{(1)}) \otimes (1\# h_{(3)}) \\ & = (y\# h_{(1)}) \otimes (D_{h_{(2)}} \# h_{(3)}) = (\id \overline{\times} \beta_{D})(\Delta_{B}(y\# h)).
  \end{align*}
  Therefore, their inverses form a modifier as well.

(2) Let $y\in C$ and $h\in H$. Then by \eqref{eq:modified-lt-counit} and \eqref{eq:ch-counits-antipode},
\begin{align*}
  \mu({_{B}\tilde\varepsilon}(y\# h)) = \mu(\beps(\beta_{D}(y\# h))) = \mu(\beps(D_{h_{(1)}}y \# h_{(2)})) = \mu(D_{h}y) 
\end{align*}
and, because the right comultiplication remains unchanged and the action is symmetric,
\begin{align*}
  \mu(\tilde\varepsilon_{C}(y\# h)) = \mu(\varepsilon_{C}(h_{(1)}(S(h_{(2)}) \actright y))) = \mu(S(h) \actright y).
\end{align*}

(3)  By Example \ref{example:integrals-ch} and Lemma \ref{lemma:modification-integrals}, $\cphic$ is a partial integral for $\tilde{\mathcal{A}}$, by  (2), $(\mu,\mu)$ is a counital base weight, and using the fact that $\phi_{H}$ is faithful \cite[Theorem 3.7]{daele}, it is not difficult to see that $\phi$ is faithful as well.

(4) By Lemma \ref{lemma:integrals-modular-base}, $\sigma^{\phi}(y)=S^{2}(y)=y$ for all $y\in C$. Let $h,h'\in H$ and $y\in C$. Then
\begin{align*}
  \phi(y h\sigma_{H}(h'_{(2)})D_{S^{-1}(h'_{(1)})}) &= \phi(h\sigma_{H}(h'_{(2)})D_{S^{-1}(h'_{(1)})}y)  \\ &=
  \mu(D_{S^{-1}(h'_{(1)})}y) \phi_{H}(h\sigma_{H}(h'_{(2)})) \\
& = \mu(h'_{(1)} \actright y) \phi_{H}(h'_{(2)}h) = \phi((h'_{(1)} \actright y)h'_{(2)}h) = \phi(h'yh).
\end{align*}
The assertion follows.
\end{proof}
\subsection{Two-sided crossed products}\label{subsection:modification-twosided}
 Consider the regular multiplier Hopf algebroid $\mathcal{A}$ obtained from a two-sided crossed product $A=C\# H \# B$ associated to compatible left and right actions of a regular multiplier Hopf algebra $(H,\Delta_{H})$ on idempotent, non-degenerate algebras $C$ and $B$ with given anti-isomorphisms $S_{B}$ and $S_{C}$ as in Example  \ref{example:hopf-chb}.
In Example \ref{example:measured-chb}, we saw that an antipodal, modular base weight $(\mu_{B},\mu_{C})$ for $\mathcal{A}$ is counital if and only if $\mu_{B}$ and $\mu_{C}$  are invariant with respect to the actions of $H$. 
We  now show that if $\mu_{B}$ and $\mu_{C}$ are only  quasi-invariant, then we can modify $\mathcal{A}$ so that we obtain a measured regular multiplier Hopf algebroid.  To simplify the discussion, we  assume
 the algebras $B$, $C$  and $H$ to be unital again.

First, we need further preliminaries about quasi-invariant functionals.  Let $C$ be a unital algebra with a left action of a regular Hopf algebra $(H,\Delta_{H})$  and a faithful, quasi-invariant functional $\mu$. 
We regard 
the action also as a left action of the co-opposite Hopf algebra $H^{\co}$ on the opposite algebra
$C^{\op}$ again. Given a functional $\mu$ on $C$, we denote by $\mu^{\op}$ the corresponding functional on $C^{\op}$.
\begin{lemma} \label{lemma:chb-modification-2}
  Let $\mu$ be a faithful, quasi-invariant functional on $C$ that admits a modular automorphism $\sigma$ such that $\sigma(h\actright y)=S^{2}(h) \actright \sigma(y)$ for all  $h\in H$ and  $y\in C$. Then:
  \begin{enumerate}
  \item $\sigma(D_{h}) = D_{S^{2}(h)}$ for all $h\in H$;
  \item the functional $\mu^{\op}$ is quasi-invariant and its Radon-Nikodym cocycle is $D$;
  \item  $D_{h_{(1)}}(h_{(2)} \actright y) = (h_{(1)} \actright y)D_{h_{(2)}}$ for all $h\in H$ and $y\in C$.
  \end{enumerate}
\end{lemma}
\begin{proof}
We repeatedly use faithfulness of $\mu$. Let $y,y' \in C$ and $h \in H$.

(1)  The relation $\mu\circ \sigma=\mu$ and the assumption on $\sigma$ imply that 
\begin{align*}
  \mu(D_{h}y) = \mu(S(h) \actright y) = \mu(S^{3}(h) \actright \sigma(y)) = \mu(D_{S^{2}(h)}\sigma(y)) = \mu(cD_{S^{2}(h)}).
\end{align*}

(2) The antipode $S_{H}^{\co}$ of $H^{\co}$ is the inverse of the antipode  $S_{H}$ of $(H,\Delta_{H})$, whence
\begin{align*}
  \mu^{\op}(S_{H^{\co}}(h) \actright c) = \mu(S_{H}^{-1}(h) \actright c) =\mu(D_{S^{-2}(h)}c) = \mu(cD_{h}).
\end{align*}

(3) By assumption on $\sigma$,
\begin{align}
  \mu((h_{(1)} \actright y')D_{h_{(2)}}y) &= \mu(D_{h_{(2)}}y(S^{2}(h_{(1)}) \actright \sigma(y')))  \nonumber \\
& = \mu(S(h_{(2)}) \actright
y(S^{2}(h_{(1)}) \actright \sigma(y')) )  
\nonumber \\ &= \mu((S(h) \actright y)\sigma(y')) = \mu(y'(S(h) \actright y)), \label{eq:chb-aux-base}
\end{align}
and by  \eqref{eq:chb-aux-1}, this is equal to $ \mu(D_{h_{(1)}}(h_{(2)} \actright y')y)$.
\end{proof}

In the following proposition, we write $y^{\op}$ if we regard an element $y$ of an algebra $C$ as an element in the opposite algebra $C^{\op}$.
\begin{proposition}
Let $(H,\Delta_{H})$ be a Hopf algebra with  invertible antipode $S_{H}$ and let
 $C$  be a unital left $H$-module algebra with a faithful functional $\mu$  that
 is quasi-invariant with respect to $H$ and admits a modular automorphism $\sigma$ such that  for all $h\in H$, $c\in C$,
 \begin{align} \label{eq:chb-modular}
   \sigma(h\actright c) = S_{H}^{2}(h) \actright \sigma(c).
 \end{align}
Regard the opposite algebra $B:=C^{\op}$  as a right $H$-module algebra via $x\actleft h := S_{H}(h)\actright x$. 
\begin{enumerate}
\item  The anti-isomorphisms
 \begin{align*}
   S_{B} \colon B=C^{\op} &\to C, \ y^{\op}\mapsto y, &
   S_{C} \colon C&\to C^{\op} = B, \ y\mapsto \sigma(y)^{\op},
 \end{align*}
satisfy 
$S_{B}(x\actleft h) =S_{H}(h) \actright S_{B}(x)$ and $S_{C}(h\actright y) = S_{C}(y) \actleft S_{H}(h)$.
\end{enumerate}
Define the associated regular multiplier Hopf algebroid 
 $\mathcal{A}=(A,B,C,S_{B},S_{C},\Delta_{B},\Delta_{C})$, where $A=C\# H \# B$, as in Example \ref{example:hopf-chb}.
 \begin{enumerate} \setcounter{enumi}{1}
 \item  There exist automorphisms $\Theta_{\lambda}$,
   $\Theta_{\rho}$ of $A$ such that
   \begin{align*}
     \Theta_{\lambda}(yhx) &= y(D_{h_{(2)}})^{\op}h_{(1)}x, &
     \Theta_{\rho}(yhx) &= yD_{h_{(1)}}h_{(2)}x
   \end{align*}
   for all $y\in C$, $h \in H$, $x\in B$, and
   $(\Theta^{-1}_{\lambda},\Theta^{-1}_{\rho},\id_{A},\id_{A})$ is a modifier of
   $\mathcal{A}$.
 \item  The pair $(\mu^{\op},\mu)$ is a counital base weight for the
associated   modification $\tilde{\mathcal{A}}$.
 \item Suppose that $(H,\Delta_{H})$ has a left and right integral $\phi_{H}$. Then the formulas
   \begin{align*}
     \cphic(yhx) &:=y\phi_{H}(h)\mu^{ \op}(x), &
     \bpsib(yhx) &:=\mu(y)\phi_{H}(h)x,
        \end{align*}
where $y\in C$, $h\in H$ and $x\in B$, define a partial left and a partial right integral, and
$(\tilde{\mathcal{A}},\mu^{\op},\mu,\bpsib,\cphic)$ is a measured regular multiplier Hopf algebroid.
\item The  modular automorphism of $\phi=\mu\circ \cphic$ is given by
\begin{align*}
\sigma^{\phi}(y) &=\sigma(y), &\sigma^{\phi}(y^{\op}) &= \sigma^{-1}(y)^{\op}, &
 \sigma^{\phi}(h)&= \sigma_{H}(h_{(2)}) D_{S(h_{(1)})} (D_{S^{-1}(h_{(3)})})^{\op}
\end{align*}
for all $y\in C$ and $h\in H$, where $\sigma_{H}$ denotes the modular automorphism of $\phi_{H}$.
 \end{enumerate}
\end{proposition}
\begin{proof}
  (1) This follows immediately from the definitions and \eqref{eq:chb-modular}.

(2)  The map $\Theta_{\rho}$ acts trivially on the subalgebra $B\subseteq A$,  and like the automorphism $\beta_{D}$ of $C\# H$ defined in  Lemma \ref{lemma:chb-modification} on the subalgebra $C\# H \cong CH \subseteq A$.  Thus,
   the canonical linear isomorphism $A \cong (C\# H) \otimes B$ identifies $\Theta_{\rho}$ with the map $\beta_{D} \otimes \id$ which is bijective. Now,  $\Theta_{\rho}$ is an automorphism because for all $x\in B$, $y\in C$, $h\in H$,
   \begin{align*}
  \Theta_{\rho}(x)     \Theta_{\rho}(yh) = xyD_{h_{(1)}}h_{(2)} = yD_{h_{(1)}}h_{(2)}(x \actleft h_{(3)}) = \Theta_{\rho}(yh_{(1)}) \Theta_{\rho}((x\actleft h_{(2)})).
   \end{align*}

To prove the assertions concerning the map $\Theta_{\lambda}$, note that the subalgebra  $BH \subseteq A$ is isomorphic to $C^{\op} \# H^{\co}$ because
\begin{align*}
  h y^{\op} &= (y^{\op} \actleft S^{-1}(h_{(2)})) h_{(1)} = (h_{(2)} \actright y)^{\op} h_{(1)}
\end{align*}
for all $h\in H$, $y\in C$. Now, similar arguments as above show that the desired properties of $\Theta_{\lambda}$ follow easily from the corresponding properties of the automorphism $\beta^{\dag}_{D}$  of $C^{\op} \# H^{\co}$.

The tuple $(\Theta_{\lambda},\Theta_{\rho},\id,\id)$ is a modifier of $\mathcal{A}$ because
\begin{align*}
 (\Theta_{\lambda}\overline{\times} \id)(\Delta_{B}(yhx)) 
&= y(D_{h_{(2)}})^{\op} h_{(1)} \otimes h_{(3)}x   = yh_{(1)} \otimes D_{(2)}h_{(3)}x = (\id \overline{\times} \Theta_{\rho})(\Delta_{B}(yhx))
\end{align*}
for all $y\in C$, $h\in H$ and $x\in B$.

(3) By assumption and definition, the base weight $(\mu^{\op},\mu)$ is antipodal and modular, and
the relations \eqref{eq:modified-lt-counit},  \eqref{eq:modified-rt-counit} and \eqref{eq:chb-base-weights} imply
\begin{align*}
  (\mu^{\op} \circ {_{B}  \tilde\varepsilon})(y^{\op}hy') &= (\mu^{\op} \circ \beps)(\Theta_{\lambda}(y^{\op}hy'))  \\ &=(\mu^{\op} \circ \beps)(y^{\op}(D_{h_{(2)}})^{\op}h_{(1)}y') = \mu((h_{(1)} \actright y')D_{h_{(2)}}y), \\
(\mu \circ \tilde\varepsilon_{C})(y^{\op}hy') &= (\mu_{C} \circ \varepsilon_{C})(y^{\op}hy') =
   \mu(S^{-1}(h \actright y)y') =   \mu_{C}(y'(S(h) \actright y))
\end{align*}
for all $y,y'\in C$ and $h\in H$.
 By \eqref{eq:chb-aux-base}, both expressions coincide.

(4) By Example \ref{example:integrals-chb}, $\cphic$ is a partial left and $\bpsib$ is a partial right integral. Since $\mu\circ \cphic = \mu^{\op} \circ \bpsib$, the base weight $(\mu^{\op},\mu)$ is quasi-invariant with respect to these partial integrals, and by (3), it is counital. Hence, it only remains to verify that the functional $\phi=\mu \circ \cphic$ is faithful, and this is easy.

(5)  By Lemma \ref{lemma:integrals-modular-base}, $\sigma^{\phi}(y)=S^{2}(y)=\sigma(y)$ for all $y\in C$. 
The same lemma and the relation $\mu^{\op}\circ \bpsib=\phi$ imply $\sigma^{\phi}(y^{\op})=S^{-2}(y^{\op})=\sigma^{-1}(y)^{\op}$ for all $y\in C$. Let now $y,y' \in C$ and $h,h'\in H$. Then
We compute
\begin{align*}
  \phi(y'hy^{\op} \cdot \sigma_{H}(h'_{(2)})D_{S(h'_{(1)})}(D_{S^{-1}(h'_{(2)})})^{\op}).
\end{align*}
Since $\Delta_{H} \circ \sigma_{H} = (\sigma_{H} \otimes S_{H}^{-2})\circ \Delta_{H}$, the expression above is equal to
\begin{align*}
  \phi(D_{S^{-1}(h'_{(1)})}y'& \cdot h\sigma_{H}(h'_{(2)})  \cdot (y^{\op} \actleft S^{-2}(h'_{(3)}))(D_{S^{-1}(h'_{(4)})})^{\op}) \\
&= \mu(D_{S^{-1}(h'_{(1)})}y')  \cdot \phi_{H}(h\sigma_{H}(h'_{(2)})) \cdot \mu^{\op}((y^{\op} \actleft S^{-2}(h'_{(3)}))(D_{S^{-1}(h'_{(4)})})^{\op}) \\
&=\mu(h'_{(1)} \actright y') \cdot \phi_{H}(h'_{(2)}h) \cdot \mu(D_{S^{-1}(h'_{(4)})}(S^{-1}(h'_{(3)}) \actright y)) \\
&= \mu(h'_{(1)} \actright y') \cdot \phi_{H}(h'_{(2)}h) \cdot \mu(y) \\
&= \phi(h' \cdot y'hy^{\co}). \qedhere
\end{align*}
\end{proof}

\subsection{Dynamical quantum groups}
In \cite{timmermann:dynamical}, we studied integration on dynamical quantum groups in order to construct operator-algebraic completions in the form of measured quantum groupoids. 
The present results generalize and clarify the results  obtained  in \cite{timmermann:dynamical}, for example, they explain the  deviation of the antipode on the operator-algebraic level from the  algebraic  antipode observed in \cite[Proposition 2.7.13]{timmermann:dynamical}.

\subsubsection*{Multiplier $(\mathcal{B},\Gamma)$ Hopf $*$-algebroids}
To recall the notion of  a dynamical quantum group as defined in \cite{timmermann:dynamical}, we need some preliminaries.

Let $\mathcal{B}$   be a commutative $*$-algebra with local units and  a left action of a discrete group
  $\Gamma$. Denote by $e\in \Gamma$ the unit element. 
A \emph{$(\mathcal{B},\Gamma)^{\ev}$-algebra} is a $*$-algebra $A$ with  a grading by $\Gamma\times \Gamma$, local units in $A_{e,e}$, and a  non-degenerate $*$-homomorphism $\mathcal{B}\otimes \mathcal{B} \to M(A)$  satisfying
\begin{align*}
  a(x\otimes y) = (\gamma(x)\otimes \gamma'(y))a \in A_{\gamma,\gamma'}  \quad \text{for all } a \in A_{\gamma,\gamma'}, x,y\in \mathcal{B}.
\end{align*}
Such  $(\mathcal{B},\Gamma)^{\ev}$-algebras form a monoidal category as follows.  Morphisms  are non-degenerate $\mathcal{B}\otimes \mathcal{B}$-linear $*$-homomorphisms into multipliers that preserve the grading. The product $A \tilde{\otimes} C$ of $(\mathcal{B},\Gamma)^{\ev}$-algebras $A$ and $C$ is the quotient of
  \begin{align*}
   \bigoplus_{\gamma,\gamma',\gamma''} A_{\gamma,\gamma'} \otimes C_{\gamma',\gamma''} 
  \end{align*}
by the subspace spanned by all elements of the form $(1\otimes x)a\otimes c - a\otimes (x \otimes 1)c$, which becomes a $(\mathcal{B},\Gamma)^{\ev}$-algebra in a natural way, and the unit object is the  crossed product $\mathcal{B}\rtimes \Gamma$ with diagonal $\Gamma\times \Gamma$-grading and the multiplication map $\mathcal{B}\otimes \mathcal{B}\to \mathcal{B} \hookrightarrow M(\mathcal{B}\rtimes \Gamma)$.

 A \emph{multiplier $(\mathcal{B},\Gamma)$-Hopf $*$-algebroid}  consists of a $(\mathcal{B},\Gamma)^{\ev}$-algebra $A$ with a comultiplication, counit and antipode, which are morphisms 
 \begin{align*}
   \Delta &\colon A \to M(A\tilde\otimes A),&  \bm{\varepsilon} &\colon A  \to \mathcal{B}\rtimes \Gamma, & S &\colon A \to A^{\co,\op}
 \end{align*}
satisfying natural conditions, see \cite[Definition 1.3.5]{timmermann:dynamical}, where $A^{\co,\op}$ is a suitably defined bi-opposite of $A$.

 Such a multiplier $(\mathcal{B},\Gamma)$-Hopf $*$-algebroid can be regarded as a  multiplier Hopf $*$-algebroid with certain extra structure as follows.  The $*$-homomorphism $\mathcal{B}\otimes \mathcal{B} \to M(A)$ extends to embeddings of  $\mathcal{B}\otimes 1$ and $1\otimes \mathcal{B}$ into $M(A)$. Denote by $B$ and $C$ the respective images, and by $S_{B}\colon B\to C$ and $S_{C} \colon C \to B$ the canonical (anti-)isomorphism. Then $\AlA$ is a left and  $\ArA$ a  right  $A\tilde\otimes A$-module, and the formulas
 \begin{align*}
   \Delta_{B}(a)(b\otimes c) &:= \Delta(a)(b\otimes c), & (b\otimes c)\Delta_{C}(a) &:= (b\otimes c)\Delta(a)
 \end{align*}
define a left and a right comultiplication so that
\begin{align*}
  \mathcal{A}:=(A,B,C,S_{B},S_{C},\Delta_{B},\Delta_{C})
\end{align*}
 becomes a multiplier $*$-bialgebroid. Its  canonical maps are bijective  by \cite[Proposition 1.3.8]{timmermann:dynamical}, so it is a multiplier Hopf $*$-algebroid. One easily verifies that its antipode is $S$ and its left and right counits $\beps$ and $\epsc$ are equal to the composition of $\bm{\varepsilon}$ with the linear maps $\mathcal{B}\rtimes \Gamma$ given by $\sum_{\gamma} x_{\gamma} \gamma \mapsto \sum_{\gamma} x_{\gamma}$ and $\sum_{\gamma} \gamma x_{\gamma}\mapsto \sum_{\gamma}x_{\gamma}$, respectively.

\subsubsection*{Integration}
The ingredients for integration considered in  \cite[Definition 1.6.1]{timmermann:dynamical} and here are related as follows. 

Assume that the multiplier $(\mathcal{B},\Gamma)$-Hopf $*$-algebroid $(A,\Delta,S,\varepsilon)$ is \emph{measured} in the sense of \cite[Definition 1.6.1]{timmermann:dynamical}, that is, it comes with
\begin{enumerate}
\item a map $  \cphic \colon A \to \mathcal{B}\cong C$, called a left integral in \cite{timmermann:dynamical}, which has to be
 $C$-linear, left-invariant with respect to $\Delta$, and to vanish on $A_{\gamma,\gamma'}$ if $\gamma\neq e$;
\item a map $\bpsib \colon A \to \mathcal{B} \cong B$, called a right integral in \cite{timmermann:dynamical}, which has to be $B$-linear, right-invariant with respect to $\Delta$, and to  vanish on $A_{\gamma,\gamma'}$ if $\gamma'\neq e$;
\item a  faithful, positive linear functional on $\mathcal{B}$ that is quasi-invariant with respect to the action of $\Gamma$ in a suitable sense and satisfies $\mu\circ \cphic = \mu\circ \bpsib$.
\end{enumerate}

Let us first consider the conditions in (1) and (2).  The first two conditions  are easily seen to be equivalent to $\cphic$ being left-invariant and $\bpsib$ being right-invariant with respect to $\Delta_{B}$ and $\Delta_{C}$ in the sense of Definition \ref{definition:invariant-elements}. 
 The third conditions in (1) and (2) follows from the first ones if the action of $\Gamma$ on $\mathcal{B}$ is free in the sense that for every non-zero $x\in \mathcal{B}$ and every $\gamma\neq e$, there is some $x'\in \mathcal{B}$ such that  $x(x'-\gamma(x'))$ is non-zero.

Let us next consider the conditions in (3). Write $\mu_{B}$ and $\mu_{C}$ for $\mu$, regarded as a functional on $B$ or $C$, respectively. Then $\mu_{B}\circ S_{C}=\mu_{C}$ and $\mu_{C} \circ S_{B}=\mu_{B}$ by definition.
If
 $a\in A$ and $\bm{\varepsilon}(a)=\sum_{\gamma} x_{\gamma} \gamma \in B\rtimes \Gamma$, then
\begin{align*}
  \mu_{B}(\beps(a)) &= \sum_{\gamma}\mu(x_{\gamma}), &
  \mu_{C}(\epsc(a)) &= \sum_{\gamma}\mu(\gamma^{-1}(x_{\gamma})).
\end{align*}
If $\mu$ is \emph{invariant} under the action of $\Gamma$ on $\mathcal{B}$, then
 $(\mu_{B},\mu_{C})$ is a counital base weight for $\mathcal{A}$  and we obtain a measured multiplier Hopf $*$-algebroid.

\subsubsection*{Modification} The condition of quasi-invariance imposed on the functional $\mu$ in (c) is strengthened in \cite[Section 2.1, condition (A2)]{timmermann:dynamical} as follows. There, one assumes existence of a family of self-adjoint, invertible multipliers $D^{1/2}_{\gamma} \in M(\mathcal{B})$  satisfying 
\begin{align*}
  D^{1/2}_{e} &= 1, & D^{1/2}_{\gamma\gamma'} &= \gamma'{}^{-1}(D^{1/2}_{\gamma})D^{1/2}_{\gamma'}, & \mu(\gamma(xD_{\gamma})) &= \mu(x)
\end{align*}
for all $\gamma,\gamma' \in \Gamma$ and $x\in \mathcal{B}$. Given such a family, we can modify $\mathcal{A}$ and obtain a measured multiplier Hopf $*$-algebroid as follows.  The formulas
\begin{align*}
\bar{D}^{\pm 1/2}(a) &= a(1\otimes D^{\mp 1/2}_{\gamma'}), &
D^{\pm 1/2}(a) &= a(D^{\mp 1/2}_{\gamma} \otimes 1), \quad \text{ where } a \in A_{\gamma,\gamma'},
\end{align*}
define automorphisms of $A$ and 
 $(\bar{D}^{1/2},D^{1/2},\bar{D}^{-1/2},D^{-1/2})$ is a modifier for $\mathcal{A}$ which is   trivial on the base and self-adjoint. This can be checked easily, see also \cite[Lemma 1.6.3]{timmermann:dynamical}.
 We thus obtain a modified multiplier Hopf $*$-algebroid
 \begin{align*}
   \tilde{\mathcal{A}} = (A,B,C,S_{B},S_{C},\tilde \Delta_{B},\tilde \Delta_{C}).
 \end{align*}
Now, $(\mu_{B},\mu_{C})$ is a counital base weight for $\tilde{\mathcal{A}}$. Indeed, if $\bm{\varepsilon}(a) = \sum_{\gamma} \gamma x_{\gamma}$, then by \eqref{eq:modified-lt-counit} and \eqref{eq:modified-rt-counit},
\begin{align*}
\mu_{B}(_{B}\tilde\varepsilon(a)) &= \sum_{\gamma} \mu(\gamma(x_{\gamma} D^{1/2}_{\gamma})) = 
 \sum_{\gamma} \mu(x_{\gamma}D^{-1/2}_{\gamma}) =  \sum_{\gamma} \mu(x_{\gamma}D^{-1/2}_{\gamma}) = \mu_{C}( \tilde\varepsilon_{C}(a)).
\end{align*}
 Thus, every measured multiplier $(\mathcal{B},\Gamma)$-Hopf $*$-algebroid satisfying  \cite[Section 2.1, condition (A2)]{timmermann:dynamical} gives rise to  a measured multiplier Hopf $*$-algebroid. 

Theorem \ref{theorem:integrals-proper} shows that in the proper case, the assumption $\mu\circ \cphic = \mu \circ \bpsib$ in \cite{timmermann:dynamical} does not restrict generality.

The measured quantum groupoid associated to $(A,\Delta,\bm{\varepsilon},S)$ in \cite{timmermann:dynamical} is rather a completion of the modification $\tilde{\mathcal{A}}$ than of $\mathcal{A}$. Indeed,
the formula for the comultiplication on the Hopf-von Neumann bimodule given  in  \cite[Lemma 2.5.8]{timmermann:dynamical} looks like the modified comultipliation $\tilde \Delta_{B}$, and the antipode  of the associated measured quantum groupoid extends the modified antipode $\tilde S =D^{1/2}SD^{1/2}$ instead of the original antipode $S$, see \cite[Proposition 2.7.13]{timmermann:dynamical}.

 \subsection*{Acknowledgements} The author would like to thank Alfons Van Daele for inspiring and fruitful discussions.

     \bibliographystyle{abbrv}
 \def\cprime{$'$}

\end{document}